\documentclass[11pt,reqno]{amsart}

\usepackage[T1]{fontenc}
\usepackage{lmodern}
\usepackage{amsmath,amssymb,amsthm,mathtools}
\usepackage{mathrsfs}
\usepackage{enumitem}
\usepackage{microtype}
\usepackage[colorlinks=true,linkcolor=blue,citecolor=blue,urlcolor=blue,pagebackref=true]{hyperref}

\newtheorem{theorem}{Theorem}[section]
\newtheorem{proposition}[theorem]{Proposition}
\newtheorem{lemma}[theorem]{Lemma}
\newtheorem{corollary}[theorem]{Corollary}
\theoremstyle{definition}

\newtheorem{example}[theorem]{Example}
\theoremstyle{remark}
\newtheorem{remark}[theorem]{Remark}

\newcommand{\Hom}{\operatorname{Hom}}
\newcommand{\End}{\operatorname{End}}

\DeclareMathOperator{\Sym}{Sym}
\DeclareMathOperator{\Mat}{Mat}
\DeclareMathOperator{\Ann}{Ann}

\numberwithin{equation}{section}
\allowdisplaybreaks[2]
\hypersetup{pdfauthor={Zhi-Wei Li},pdftitle={Root-power quotients of nil-Hecke algebras and Weyl invariant theory}}
\title[Root-power quotients and Weyl invariants]{Root-power quotients of nil-Hecke algebras\protect\\ and Weyl invariant theory}

\author{Zhi-Wei Li}
\address{School of Mathematics and Statistics, Jiangsu Normal University,
Xuzhou 221116, Jiangsu, P. R. China}
\email{zhiweili@jsnu.edu.cn}
\date{}

\begin{document}
\begin{abstract}
We study nil-Hecke quotients defined by powers of root coordinates.
We identify their central ideals through the determinant components
of root-orbit power ideals. Apolarity and discriminant removal reduce
the graded dual centers to directional bounds on invariant polynomials.
We determine their Hilbert series for every exponent and root orbit in
all irreducible crystallographic types except $E_6,E_7,E_8$.
Relative discriminants compare centers for reflection subgroups and
transfer equivariant resolutions. Coordinate-root centers in types
$B$ and $C$ are Grassmannian cohomology rings. We classify
self-injectivity in type $A$, in rank two, and in type $F_4$,
and determine the first nonsemisimple centers in the remaining
classical families. For an explicit open set of reflection parameters,
the polynomial image of the rational Cherednik algebra at $t=0$
is the full nil-Hecke algebra. This realizes our algebras as
Cherednik quotients and their finite centers as closed subschemes
on a section of the generalized Calogero--Moser space.
\end{abstract}
\maketitle
\clearpage
\setcounter{tocdepth}{2}
\tableofcontents
\clearpage

\section{Introduction}\label{sec:introduction}

We study the quotient of a polynomial nil-Hecke algebra obtained by
imposing a power of one root coordinate. Its center is a quotient
of the Weyl invariant ring. The problem is to identify the defining
ideal and determine how the root system and the root orbit control
the resulting finite-dimensional algebra.

The central ideal has two complementary descriptions: by polynomial
content and the determinant component of a Weyl-orbit power ideal,
and by apolar conditions on the projective root orbit. Removing the
Weyl discriminant converts the graded dual of the center into a
directional-degree bound on invariant polynomials. We solve this
bound for every exponent and root orbit in all irreducible
crystallographic types except $E_6,E_7,E_8$. Relative discriminants
compare the centers for reflection subgroups, and the differential
action recovers multiplication and socles from their inverse systems.

We also identify these algebras as quotients of rational Cherednik
algebras at $t=0$, through their polynomial representation.
Their centers define closed subschemes on a specified section
of the generalized Calogero--Moser space. A formal exponential
comparison gives their completed zero-residue quiver Demazure
interpretation. The polynomial proofs are independent of the
general QDA construction and its mixed-residue PBW theory.
\vskip5pt
\noindent\textbf{Root coordinates and differentiation.}
Let $\Phi$ be a finite reduced crystallographic root system with simple
roots $\Delta=\{\alpha_i:i\in I\}$, positive roots $\Phi^+$, Weyl group
$W$, and root lattice $Q=\mathbb Z\Phi$.  Put
\[
 V=\mathbb C\otimes_{\mathbb Z}Q,\qquad
 P=\Sym_{\mathbb C}(V),\qquad S=P^W,
 \qquad T=\Sym_{\mathbb C}(V^*).
\]
For $\lambda\in Q$, write
$x_\lambda=1\otimes\lambda\in V=P_1\subseteq P$, and extend
this notation $\mathbb C$-linearly to $\lambda\in V$.
These elements have degree one and satisfy
$x_{\lambda+\mu}=x_\lambda+x_\mu$. In particular,
$P=\mathbb C[x_{\alpha_i}:i\in I]$.
The algebra $T$ is the polynomial algebra on $V$, whereas
$P$ is naturally the polynomial algebra on $V^*$.
Write $S_+=\bigoplus_{k>0}S_k$ and $P_+=\bigoplus_{k>0}P_k$.
For a finite-dimensional graded vector space $E$, its Hilbert
series is $H_E(q)=\sum_{k\geq0}(\dim_{\mathbb C}E_k)q^k$.

We fix a $W$-invariant positive-definite inner product on the real
root space and extend it complex bilinearly to $(\, ,\,)$ on $V$.
The resulting identification $V\simeq V^*$ sends $x_\lambda$
to the linear function $v\mapsto(\lambda,v)$ and identifies
$P$ with $T$. For $\beta\in V$, directional differentiation on $T$
is the derivation
\[
 \mathsf d_\beta\colon T\longrightarrow T,\qquad
 (\mathsf d_\beta F)(v)
 =\left.\frac{d}{dt}\right|_{t=0}F(v+t\beta)
 \quad(F\in T,\ v\in V).
\]
Thus $\mathsf d_\beta(\ell)=\ell(\beta)$ for $\ell\in V^*$.
On $P$ we use the same notation for the transported derivation;
it satisfies
\[
 \mathsf d_\beta(x_\lambda)=(\lambda,\beta)
 \quad(\lambda,\beta\in V).
\]
For a monomial $p=x_{\beta_1}\cdots x_{\beta_k}\in P$, put
\[
 p(\mathsf d)=\mathsf d_{\beta_1}\cdots\mathsf d_{\beta_k},
 \qquad 1(\mathsf d)=\operatorname{id},
\]
and extend linearly to all $p\in P$. Since the directional
derivatives commute and depend linearly on their directions,
this defines an algebra homomorphism from $P$ to the
constant-coefficient differential operators on $T$, and also on $P$
under the chosen identification. These are ordinary derivatives,
not divided-difference operators.

Put $\Delta_\Phi=\prod_{\beta\in\Phi^+}x_\beta$ and
$N=|\Phi^+|$. Its image in $T$ is the dual discriminant
\[
 \Delta_\Phi^\vee(v)=\prod_{\beta\in\Phi^+}(\beta,v).
\]

Let $M_f$ denote multiplication by $f\in P$. The polynomial
nil-Hecke algebra has its classical endomorphism realization
$NH_W=\operatorname{End}_S(P)$, recalled in
Lemma~\ref{lem:completed-nilhecke}. For a root $\alpha$ and $m\geq1$,
put
\[
 \begin{gathered}
 A_{\alpha,m}=NH_W/\langle M_{x_\alpha^m}\rangle,
 \qquad Z(\alpha,m)=S/J^{\rm pol}_{\alpha,m},\\
 J^{\rm pol}_{\alpha,m}
   =(\varphi(x_\alpha^m):\varphi\in\operatorname{Hom}_S(P,S)).
 \end{gathered}
\]
The brackets denote a two-sided ideal. The invariant ring has
its polynomial grading. Matrix identifications below are algebra
isomorphisms; homogeneous bases give the corresponding degree
shifts on matrix entries.
Also put
\[
 I_{\alpha,m}=(x_{w\alpha}^m:w\in W)\subseteq P,
 \qquad
 \Sigma_\alpha=\{[w\alpha]:w\in W\}\subseteq\mathbb P(V).
\]
The following theorem collects the conceptual core of the paper.  The
individual assertions are proved in
Theorems~\ref{rc:thm:cyclotomic}, \ref{str:thm:alternation},
\ref{str:thm:rootpower-geometry}, \ref{str:thm:discriminant-removal},
and Proposition~\ref{geo:prop:flag-threshold}.

\begin{theorem}[Root-power center theorem]\label{intro:thm:rootpower}
Assume that $\Phi$ is irreducible.  For every root $\alpha$ and every
$m\geq1$ the following statements hold.
\begin{enumerate}[label=\textup{(\roman*)}]
\item The quotient is a full matrix algebra over its center:
\[
 A_{\alpha,m}
 \cong
 \Mat_{|W|}\bigl(Z(\alpha,m)\bigr).
\]
Moreover
\begin{equation}\label{intro:eq:alternation}
 \Delta_\Phi J^{\rm pol}_{\alpha,m}
 =I_{\alpha,m}\cap P^{\det}.
\end{equation}
\item If $\mathfrak p_\xi\subseteq T$ is the homogeneous ideal of
$\xi\in\Sigma_\alpha$, then
\begin{equation}\label{intro:eq:fatpoint-dual}
 Z(\alpha,m)_k^*
 \cong
 \left(
   \bigcap_{\xi\in\Sigma_\alpha}
   \mathfrak p_\xi^{\max\{N+k-m+1,0\}}
 \right)_{N+k}^{\det}.
\end{equation}
Equivalently, if
\[
 \mathcal D_{\alpha,m}
 =\{F\in T^{\det}:\mathsf d_{w\alpha}^{\,m}F=0
                    \text{ for all }w\in W\},
\]
then $J^{\rm pol}_{\alpha,m}=\Ann_S(\mathcal D_{\alpha,m})$.
\item Let
\[
 \nu_\alpha
 =1+\#\{\beta\in\Phi^+:(\alpha,\beta)\neq0\}.
\]
If $J(\alpha)$ is the set of simple coroots orthogonal to the dominant
root in $W\alpha$, then
\[
 \nu_\alpha
 =1+\dim_{\mathbb C}(\mathbf G/\mathbf P_{J(\alpha)}),
 \qquad
 A_{\alpha,m}\neq0
 \Longleftrightarrow m\geq\nu_\alpha.
\]
At the threshold,
\[
 A_{\alpha,\nu_\alpha}
 =A_{\alpha,\nu_\alpha+1}
 \cong\Mat_{|W|}(\mathbb C).
\]
\item For $m\geq\nu_\alpha$, writing $\Delta_\Phi^\vee$ for the dual
discriminant, one has
\begin{equation}\label{intro:eq:disc-removal}
 \mathcal D_{\alpha,m}
 =\Delta_\Phi^\vee
 \{H\in T^W:\deg_{w\alpha}H\leq m-\nu_\alpha
                  \text{ for all }w\in W\}.
\end{equation}
Thus the Hilbert function of the center is reduced to a finite
invariant-theoretic directional-degree problem.
\end{enumerate}
\end{theorem}

\begin{proof}
The matrix quotient and the definition of the content ideal are
Theorem~\ref{rc:thm:cyclotomic}; the equality
\eqref{intro:eq:alternation} is Theorem~\ref{str:thm:alternation}.
The apolar description and annihilator statement are
Theorem~\ref{str:thm:rootpower-geometry}.  The geometric threshold is
Proposition~\ref{geo:prop:flag-threshold}, together with the matrix
criterion in Theorem~\ref{rc:thm:cyclotomic}.  Finally,
\eqref{intro:eq:disc-removal} is
Theorem~\ref{str:thm:discriminant-removal}, and the first nonzero
matrix quotient is Corollary~\ref{str:cor:first-rootpower}.
\end{proof}

The theorem reduces a two-sided ideal problem in an operator algebra
to a finite calculation in the invariant ring. The nil-Hecke endomorphism theorem supplies
the matrix description, while alternation identifies the precise
central ideal. The apolar and discriminant descriptions then give
both graded dimensions and a differential action that determines
the center algebra. This last distinction matters: an identification
of inverse-system vector spaces alone does not determine products
or a Frobenius form.

For a long root, the integer $\nu_\alpha$ is
$2h^\vee(\Phi)-2$, the dimension of the minimal nonzero nilpotent orbit.
For a short root the corresponding statement uses the dual root system; see
Corollary~\ref{geo:cor:minimal-orbit} and~\cite{WangMin}.  The divided-difference relation also gives the canonical parity equality
\[
 A_{\alpha,2a}=A_{\alpha,2a+1}\qquad(a\geq1),
\]
proved in Theorem~\ref{str:thm:parity}.

\subsection*{Hilbert series outside the exceptional \texorpdfstring{$E$}{E} types}
The directional-degree problem in
Theorem~\ref{intro:thm:rootpower} can be solved uniformly for
every root orbit in the classical types, $F_4$, and $G_2$. The relevant invariant generators have
two gradings: their polynomial degrees $d_j$, which grade the
center, and auxiliary weights $\omega_j$, which measure half their
directional degrees along the chosen root. Suitable generators make
the directional-degree bound a bound on the weighted degree of a
monomial.

\begin{theorem}[All irreducible types except $E_6,E_7,E_8$]
\label{intro:thm:classical}
For the types and root orbits listed below, the center
vanishes for $m<\nu_\alpha$. In the nonzero range its Hilbert
series is determined by
\begin{equation}\label{intro:eq:classical-generating-function}
 \sum_{r\geq0}H_{Z(\alpha,\nu_\alpha+2r)}(q)t^r
 =\frac{1}{(1-t)\prod_j(1-t^{\omega_j}q^{d_j})}.
\end{equation}
The pairs $(d_j,\omega_j)$ and thresholds are
\[
\begin{array}{c|c|l}
 \Phi\text{ and root orbit}&\nu_\alpha&(d_j,\omega_j)\\ \hline
 A_n\ (n\geq1)&2n&(j,1),\quad 2\leq j\leq n+1\\
 B_n\text{ short},\ C_n\text{ long}
     &2n&(2j,1),\quad 1\leq j\leq n\\
 B_n\text{ long},\ C_n\text{ short}
     &4n-4&(2,1),(4,1),(2j,2)\quad(3\leq j\leq n)\\
 D_n\ (n\geq4)&4n-6
     &(2,1),(4,1),(n,1),(2j,2)\quad(3\leq j\leq n-1)\\
 F_4&16&(2,1),(6,2),(8,2),(12,3)\\
 G_2&6&(2,1),(6,2).
\end{array}
\]
The two $B_n,C_n$ rows have $n\geq2$, and the $F_4,G_2$
rows apply to both root lengths. Repeated pairs are retained, so the $D_4$ row has two pairs $(4,1)$. The same
Hilbert series occurs at exponent $\nu_\alpha+2r+1$, and
\[
 \dim A_{\alpha,m}=|W|^2\dim Z(\alpha,m).
\]
\end{theorem}

\begin{proof}
Proposition~\ref{str:prop:weighted-invariants} gives the
generating function from a weighted invariant basis. The bases
and thresholds are established in
Theorem~\ref{str:thm:an-centers},
Corollary~\ref{str:cor:coordinate-hilbert},
Theorems~\ref{str:thm:bd-hilbert} and~\ref{str:thm:f4-hilbert},
and Corollary~\ref{str:cor:root-orbit-ranktwo}.
Parity and the matrix
description are Theorems~\ref{str:thm:parity}
and~\ref{rc:thm:cyclotomic}.
\end{proof}

Thus a single rational generating function describes every exponent
in each root orbit covered by the theorem. At $q=1$ it gives the dimensions:
for type $A_n$ and the coordinate-root $B_n,C_n$ families,
\[
 \dim Z(\alpha,m)=\binom an,\qquad a=\lfloor m/2\rfloor\geq n.
\]
For the remaining $B_n,C_n$ orbits and type $D_n$, the
dimension sequences are quasipolynomials of degree $n$ and
period dividing two; explicit finite sums appear in
Corollary~\ref{str:cor:bd-dimensions}. In type $F_4$, the center
dimensions have generating function
\[
 \frac1{(1-t)^2(1-t^2)^2(1-t^3)}
\]
and form a degree-four quasipolynomial of period dividing six
(Corollary~\ref{str:cor:f4-dimensions}). These formulas treat
all exponents simultaneously.

\subsection*{Algebra structures and failure of duality}
The Hilbert series do not by themselves determine multiplication.
Discriminant removal transports the differential action to
\[
 \mathscr D_s(H)=(\Delta_\Phi^\vee)^{-1}
                         s(\mathsf d)(\Delta_\Phi^\vee H),
 \qquad s\in S,\ H\in T^W.
\]
Proposition~\ref{str:prop:conjugated-action} recovers the
content ideal as the annihilator of the finite invariant inverse
system under this action. This supplies a way to compute products
and socles after the Hilbert series is known.

In type $A_n$, $n\geq2$, the center is Gorenstein exactly
when $a=n$. Thus every nonsemisimple root-power quotient in
these types fails to be self-injective
(Theorem~\ref{str:thm:an-centers}). For short roots in $B_n$
and long roots in $C_n$, the behavior is different:
\[
 Z(\alpha,m)\cong H^*(\operatorname{Gr}(n,a);\mathbb C)
 \qquad(a\geq n).
\]
Every nonzero center in these coordinate-root families is a complete
intersection, and every corresponding matrix quotient is symmetric
(Theorem~\ref{str:thm:bc-grassmannian}). Although these families
have the same center dimensions as type $A_n$, their graded
algebra structures have different duality properties.

The first exponent pair above the semisimple quotient makes the
distinction among the remaining classical roots explicit.
For $m=\nu_\alpha+2,\nu_\alpha+3$,
Theorem~\ref{str:thm:first-bd-centers} gives
\[
\begin{array}{c|c|c}
 \Phi\text{ and root orbit}&Z(\alpha,m)&\text{generator degrees}\\ \hline
 B_n\text{ long},\ C_n\text{ short}
   &\mathbb C[u]/(u^3)&2\\
 D_4&\mathbb C[u,v,w]/(u,v,w)^2&2,4,4\\
 D_n\ (n\geq5)&\mathbb C[u,v]/(u^3,uv,v^2)&2,n.
\end{array}
\]
The type-$D_n$ centers therefore have socle dimension two for
$n\geq5$, and three for $n=4$. In particular, these quotients
cannot be Morita equivalent to nonzero ordinary cyclotomic
nil-Hecke algebras. The latter have Grassmannian centers and are
symmetric. The existence of a flag-variety interpretation of the
threshold consequently does not imply a Poincar\'e-duality
interpretation of every center.

In rank two we go further and determine presentations, minimal
resolutions, dimensions, and socles in every exponent, for both
root lengths. Type $G_2$ is linked to type $A_2$ by an
odd-part transfer. With $a=\lfloor m/2\rfloor$, its center
has dimension $\lfloor(a-1)^2/4\rfloor$ for $a\geq3$,
and is Gorenstein exactly for $3\leq a\leq5$.
In type $B_2/C_2$, both root lengths give Grassmannian
centers. Type $A_1$ has center
$\mathbb C[c_2]/(c_2^a)$ for every $a\geq1$.
The complete rank-two self-injectivity classification is
Corollary~\ref{str:cor:symmetry}.

The exceptional calculation in type $F_4$ also determines the
complete self-injectivity range. For either root length,
\[
 \begin{array}{c|c}
 m&Z(\alpha,m)\\ \hline
 16,17&\mathbb C\\
 18,19&\mathbb C[u]/(u^2)\\
 20,21&\mathbb C[u]/(u^5)
 \end{array}
 \qquad\deg u=2,
\]
and every center with $m\geq22$ is non-Gorenstein
(Theorem~\ref{str:thm:f4-symmetry}). The proof uses triality
for the invariant generators and the conjugated Laplacian
for multiplication. Thus the exceptional-type result includes
an algebra classification in the entire self-injective range.

\subsection*{Reflection subgroups and structural computation}
The comparison with type $A_2$ in the $G_2$ calculation is a
case of a general relative-discriminant theorem. Let
$W'\subseteq W$ be a reflection subgroup with
$W'\alpha=W\alpha$, put $S_W=P^W$, $S_{W'}=P^{W'}$, and
write $\Delta_W=\Delta_{W'}\delta$.

\begin{theorem}[Relative-discriminant comparison]
\label{intro:thm:transfer}
For every $m\geq1$,
\[
 J^W_{\alpha,m}
   =(J^{W'}_{\alpha,m}:\delta)\cap S_W.
\]
Here the ideals are the polynomial content ideals for their
respective invariant rings. If $W'$ is normal, let $\chi$ be
the character of $W/W'$ on $\delta$. Multiplication by
$\delta$ gives
\[
 Z_W(\alpha,m)_k
  \cong Z_{W'}(\alpha,m)_{k+\deg\delta}^{\chi}
\]
as $S_W$-modules, with the indicated degree shift.
Equivariant free resolutions transfer by taking this character
component and lowering degrees by $\deg\delta$; the resulting
resolutions may require cancellations to become minimal.
\end{theorem}

\begin{proof}
Apply Theorem~\ref{str:thm:subgroup-transfer} to the common
orbit ideal and use Corollary~\ref{str:cor:resolution-transfer}.
\end{proof}

This gives uniform comparisons for $A_2\subset G_2$,
$D_n\subset B_n$, and $D_4\subset F_4$.
Proposition~\ref{str:prop:leading-stabilizer} embeds the
associated graded directional filtration into the polynomial
invariants of the root stabilizer, with one extra grading
variable. Proposition~\ref{str:prop:betti-complex} computes
all graded Betti numbers and the socle from a finite complex
of the conjugated differential operators. These statements
provide the structural tools used in the explicit calculations
of the later sections.

\subsection*{A rational Cherednik interpretation}

The nil-Hecke quotient has a further interpretation through
the polynomial representation of a rational Cherednik algebra.
Let $\mathfrak h=V^*$, so $P=\mathbb C[\mathfrak h]$,
and write $H_{t,\mathbf k}(W)$ for the rational Cherednik
algebra with reflection parameters $\mathbf k$ in the
convention of~\eqref{ch:eq:relations}. At $t=0$ its
polynomial action uses only the divided differences
$\partial_\alpha=(1-s_\alpha)/x_\alpha$.
The standard Dunkl construction is recalled in
Section~\ref{ch:sec:interpretation}.

\begin{theorem}[Cherednik quotients and central support]
\label{intro:thm:cherednik}
Assume $k_\alpha>0$ real for every root and put
$\mathfrak a_{\mathbf k}
=\operatorname{Ann}_{H_{0,\mathbf k}(W)}P$.
The polynomial representation induces
\[
 H_{0,\mathbf k}(W)/\mathfrak a_{\mathbf k}
 \cong NH_W,
\]
and hence
\[
 \frac{H_{0,\mathbf k}(W)}
 {\mathfrak a_{\mathbf k}+\langle x_\alpha^m\rangle}
 \cong\operatorname{Mat}_{|W|}(Z(\alpha,m)).
\]
The action on $1\in P$ gives a split surjection
\[
 \chi_{\mathbf k}:Z(H_{0,\mathbf k}(W))\longrightarrow S.
\]
Thus $\operatorname{Spec}Z(\alpha,m)$ is a closed
subscheme of the section $\operatorname{Spec}S$ in the
generalized Calogero--Moser space, with defining ideal
$\chi_{\mathbf k}^{-1}(J^{\rm pol}_{\alpha,m})$.
\end{theorem}

\begin{proof}
These are Theorem~\ref{ch:thm:surjectivity},
Corollary~\ref{ch:cor:quotients}, and
Proposition~\ref{ch:prop:central-section}.
\end{proof}

Theorem~\ref{ch:thm:surjectivity} gives the larger explicit
open parameter set~\eqref{ch:eq:parameter-hyperplanes}.
Its proof first tests irreducibility on the coinvariant
fiber, using the central element
$\sum_{\alpha>0}k_\alpha(1-s_\alpha)$, and then lifts
surjectivity by graded Nakayama. This connects the center
calculations with the polynomial representation while
keeping their invariant-theory proofs independent.
The specialization $t=0$ is essential for this finite-dimensional
interpretation with the prescribed Weyl-group action; see
Proposition~\ref{ch:prop:trace-obstruction}.

\vskip5pt
\noindent\textbf{Relation to existing work and scope.}
The nil-Hecke endomorphism theorem and its relation to flag
varieties are classical~\cite{BGG,Demazure,KostantKumar,KumarFlag}.
The harmonic description of coinvariants is due to
Steinberg~\cite{Steinberg}. The correspondence between powers
of linear forms and fat-point inverse systems is also classical;
see Emsalem--Iarrobino and Ardila--Postnikov
\cite{EmsalemIarrobino,ArdilaPostnikovPower}.
Crowley--Larson interpret arrangement power-ideal inverse systems
as spaces of sections of line bundles~\cite{CrowleyLarsonPower}.
We use these established inputs to identify the determinant
component selected by a root-power content ideal and then remove
the discriminant. The resulting invariant directional filtration
is the computational object of this paper.

Ordinary strand-power cyclotomic nil-Hecke algebras have matrix
descriptions over Grassmannian cohomology rings
\cite{Lauda,HuLiang}. Root coordinates give different quotients.
In type $A_n$, $n\geq2$, the centers are non-Gorenstein beyond
the first nonzero exponent pair. The coordinate-root families in
$B_n,C_n$ recover Grassmannian rings, while other root orbits
have different Hilbert functions and socles. Our main results
identify these ideals, compute the Hilbert series outside the
three exceptional $E$ types, transfer centers and resolutions
through reflection subgroups, and determine the stated algebra
structures and self-injectivity ranges.

Geometric descriptions of centers of cyclotomic quiver Hecke
algebras occur in~\cite{WebsterCenters,SVVCenter,HuLinCenter}.
No identification of our root-power centers with those geometric
centers is assumed. The non-Gorenstein examples in particular
require care with any proposed cohomological interpretation.
Section~\ref{ii:sec:comparison} gives the zero-residue
interpretation needed here and states the precise comparison
hypothesis for transfer to a cyclotomic Hecke orbit summand.
The Dunkl representation is classical~\cite{EtingofMa}.
Our Cherednik comparison is a consequence of that representation
and the nil-Hecke endomorphism theorem, with a direct proof of
the parameter condition. Central supports of $t=0$ Verma
modules have been studied by Bellamy~\cite{BellamyVerma};
we identify the specific content ideals on the polynomial
section and do not require a description of the full space.
Liu's comparison~\cite{LiuKZ} concerns a broader
residue-dependent setting and is not used in the proofs of the
center formulas in this paper.

The scope of the classification is precise. Hilbert series are
determined in every exponent for $A,B,C,D,F_4,G_2$. Complete
algebra presentations in every exponent are obtained for the
coordinate-root $B_n,C_n$ families and in rank two. We give
the first nonsemisimple centers in the remaining classical
families and all $F_4$ centers in the self-injective range.
The finite differential complex is a method for calculating
higher-rank relations and socles; no unproved uniform rank
formula is asserted.

\vskip5pt
\noindent\textbf{Organization.}
Section~\ref{ii:sec:nilhecke} constructs the polynomial quotients.
Section~\ref{ii:sec:threshold} proves the exact threshold and its
partial-flag interpretation. Sections~\ref{str:sec:content}
and~\ref{str:sec:structural-transfer} give apolarity,
discriminant removal, subgroup transfer, and differential models.
Section~\ref{ch:sec:interpretation} gives the Cherednik quotient
and its central-support interpretation.
The classical, $F_4$, and rank-two calculations are in
Sections~\ref{str:sec:classical-centers}, \ref{str:sec:f4},
and~\ref{str:sec:ranktwo}. The final section,
Section~\ref{ii:sec:comparison}, gives the completed zero-residue
interpretation and records the precise hypothesis under which the
root-power quotients transfer to a cyclotomic Hecke orbit summand.

Unless a statement specifies otherwise, the type calculations
assume that $\Phi$ is irreducible. The general content and
reflection-subgroup constructions also allow reducible subgroups.
We reserve $I_{\alpha,m}$ for the orbit power ideal in $P$ and
$J^{\rm pol}_{\alpha,m}$ for its content ideal in $S$.
The superscript $\det$ denotes the determinant isotypic component;
it does not in general denote a subalgebra.

\section{Polynomial nil-Hecke algebras and root-power quotients}
\label{ii:sec:nilhecke}

We first construct the quotients in the polynomial setting. This
makes their dependence on the root system explicit and separates
their invariant-theory description from any Hecke comparison.
Throughout the paper the ground field is $\mathbb C$.

\subsection{The endomorphism algebra}

Use the root data and polynomial rings $P=\operatorname{Sym}(V)$,
$S=P^W$ of the introduction, and put $d=|W|$.
For $f\in P$, let $M_f$ be multiplication by $f$. The simple
divided differences are
\[
 \partial_i(f)=\frac{f-s_i(f)}{x_{\alpha_i}}.
\]
The numerator is divisible by $x_{\alpha_i}$, so these are
operators on $P$. They are $S$-linear. The polynomial nil-Hecke
algebra $NH_W$ is the subalgebra of $\operatorname{End}_{\mathbb C}(P)$
generated by the $M_f$ and the $\partial_i$.

\begin{lemma}[The classical nil-Hecke endomorphism theorem]
\label{lem:completed-nilhecke}
The extension $P/S$ is finite free of rank $d$, and the polynomial
action identifies
\[
 NH_W=\operatorname{End}_S(P)\cong\operatorname{Mat}_d(S).
\]
Let $\widehat S$ be the $S_+$-adic completion and let
$\widehat P=\mathbb C[[x_{\alpha_i}:i\in I]]$.
Then $\widehat P=\widehat S\otimes_S P$ and
\[
 \widehat S\otimes_S NH_W
 \cong\operatorname{End}_{\widehat S}(\widehat P).
\]
The endomorphism identifications are canonical; a matrix
identification requires a choice of basis.
\end{lemma}

\begin{proof}
The polynomial invariant ring is generated by homogeneous basic
invariants, and the coinvariant basis lifts to an $S$-basis of $P$.
The divided-difference action gives the full endomorphism algebra;
this is the classical nil-Hecke theorem of~\cite{KostantKumar}.
See also~\cite[Proposition 4.10(iv)]{Wojciechowski} for this
formulation. The origin is the only point above the origin of the
invariant quotient, so completion of $P$ over $S$ is its completion
at $P_+$. Finite freeness allows endomorphisms to commute with
this base change, proving the completed statement.
\end{proof}

The grading is the polynomial grading: $\deg x_\alpha=1$ and
$\deg\partial_i=-1$. A homogeneous $S$-basis of $P$ produces
a matrix algebra with the corresponding row and column degree
shifts. Whenever an ordinary matrix algebra is displayed below,
the assertion is an algebra isomorphism; the grading on its center
is always the polynomial grading inherited from $S$.

\subsection{Content and simultaneous quotients}
\label{rc:sec:cyclotomic}

\begin{lemma}[Content of a multiplication operator]
\label{str:lem:content}
Let $T_0$ be a commutative ring and $B$ a commutative $T_0$-algebra
that is finite free over $T_0$. For $b\in B$, define
\[
 c_{T_0}(b)=\bigl(\varphi(b):\varphi\in\operatorname{Hom}_{T_0}
                                          (B,T_0)\bigr).
\]
If $b=\sum_j b_jv_j$ in a $T_0$-basis, then
$c_{T_0}(b)=(b_j)_j$. This is also the ideal generated by the
entries of the matrix of multiplication by $b$.
It is independent of the basis and commutes with base change.
\end{lemma}

\begin{proof}
The coordinate functionals form the dual basis, so their values
generate $(b_j)_j$. Every coordinate of $bv_k$ belongs to that
ideal. Conversely, applying the multiplication matrix to the
coordinate vector of $1$ gives the vector $(b_j)_j$. This proves
equality with the entry ideal. A transported basis gives the
same coordinate vector after any base change.
\end{proof}

Let $\Xi\subseteq\Phi$ be nonempty and choose integers
$m_\alpha\geq1$ for $\alpha\in\Xi$. Put
\[
 \begin{aligned}
 A_{\Xi,\mathbf m}
   &=NH_W/\langle M_{x_\alpha^{m_\alpha}}:\alpha\in\Xi\rangle,\\
 J^{\rm pol}_{\alpha,m}&=c_S(x_\alpha^m),\qquad
 J^{\rm pol}_{\Xi,\mathbf m}
   =\sum_{\alpha\in\Xi}J^{\rm pol}_{\alpha,m_\alpha}.
 \end{aligned}
\]
Angle brackets mean the generated two-sided ideal. For one root
we write $A_{\alpha,m}$ and
$Z(\alpha,m)=S/J^{\rm pol}_{\alpha,m}$.
We also use $J_{\alpha,m}=J^{\rm pol}_{\alpha,m}\widehat S$
for the completed ideal. The coinvariant algebra is
\[
 \mathcal R_W=P/S_+P.
\]
It is a $d$-dimensional graded local algebra. Define
$\nu_\alpha$ to be the nilpotence exponent of $x_\alpha$ in
$\mathcal R_W$.

\begin{theorem}[Matrix quotient and nonvanishing]
\label{rc:thm:cyclotomic}
There is an algebra isomorphism
\begin{equation}\label{rc:eq:cyclotomic-matrix}
 A_{\Xi,\mathbf m}
 \cong\operatorname{Mat}_d(S/J^{\rm pol}_{\Xi,\mathbf m}).
\end{equation}
The quotient is nonzero precisely when
\begin{equation}\label{rc:eq:criterion}
 m_\alpha\geq\nu_\alpha\qquad(\alpha\in\Xi).
\end{equation}
If $W\Xi$ spans $V$, its center is finite-dimensional and,
when nonzero, is a graded local algebra with residue field
$\mathbb C$. In that case the quotient has a unique simple
module, namely $\mathcal R_W$, of dimension $d$.
Moreover, completion at $S_+$ does not change this quotient.
The spanning condition is automatic for a single root in an
irreducible root system.
\end{theorem}

\begin{proof}
In a matrix algebra over a commutative ring, the two-sided ideal
generated by a matrix $F=(f_{ij})$ is the matrix algebra over
its entry ideal: indeed,
$E_{ai}FE_{jb}=f_{ij}E_{ab}$. Apply this to the multiplication
matrices and use Lemma~\ref{str:lem:content}. This proves
\eqref{rc:eq:cyclotomic-matrix} and identifies its center.

Each content ideal is homogeneous. Their sum is proper if and
only if it is contained in $S_+$, which holds exactly when each
multiplication matrix becomes zero modulo $S_+$.
Reduction of the endomorphism algebra gives
$\operatorname{End}_{\mathbb C}(\mathcal R_W)$.
Multiplication by an element of the unital algebra $\mathcal R_W$
is zero if and only if the element itself is zero, as can be
checked on $1$. This proves~\eqref{rc:eq:criterion}.

Write $J=J^{\rm pol}_{\Xi,\mathbf m}$. The ideal $PJ$ contains
$x_\alpha^{m_\alpha}$ and all its Weyl translates: content is
unchanged by the $S$-linear action of $W$ on $P$.
If $W\Xi$ spans $V$, choose a basis of $V$ from these root
directions. A positive power of each corresponding coordinate
belongs to $PJ$, so $P/PJ$ is finite-dimensional. Averaging
over $W$ gives $PJ\cap S=J$, whence $S/J$ embeds in $P/PJ$.
Thus $S/J$ is finite-dimensional. Its positive-degree ideal
is nilpotent, and its degree-zero part is $\mathbb C$ when the
ring is nonzero. This proves the local assertion and shows that
$S/J\cong\widehat S/J\widehat S$.

The matrix algebra over a nonzero local Artinian ring has a
unique simple module of dimension $d$. Reduction of its natural
polynomial representation identifies that module with
$P/S_+P=\mathcal R_W$. The criterion already proved ensures
that all imposed powers kill this module. Finally, base change
of~\eqref{rc:eq:cyclotomic-matrix} to $\widehat S$ preserves it,
which proves the completion assertion.
\end{proof}

\begin{example}[Rank one]\label{ex:a1-centre}
For $P=\mathbb C[x]$ and $S=\mathbb C[c]$, $c=x^2$, use the
$S$-basis $1,x$. If $a=\lfloor m/2\rfloor$, the coordinates
of $x^m$ generate $(c^a)$. Therefore
\[
 Z(\alpha,m)=\mathbb C[c]/(c^a),\qquad
 A_{\alpha,m}\cong\operatorname{Mat}_2(\mathbb C[c]/(c^a)),
 \qquad\deg c=2.
\]
The quotient is zero for $m=1$ and nonzero for $m\geq2$.
\end{example}

\section{Exact thresholds and partial flag varieties}
\label{ii:sec:threshold}
\subsection{The root-coordinate exponent}
We now evaluate the exponent in \eqref{rc:eq:criterion}.
Use the bilinear form and the differential notation fixed in the
introduction. In this subsection the operators act on $P$ through
$\mathsf d_v(x_\lambda)=(\lambda,v)$. In particular,
$f(\mathsf d)$ acts on the discriminant
\[
 \Delta_\Phi=\prod_{\beta\in\Phi^+}x_\beta\in P.
\]
Directional differentiation is distinct from the divided difference
$\partial_\alpha=(1-s_\alpha)/x_\alpha$.

\begin{lemma}[Differential annihilator]\label{rc:lem:annihilator}
With constant-coefficient differentiation defined by the chosen
inner product,
\begin{equation*}\label{rc:eq:annihilator}
 \Ann(\Delta_\Phi)
 :=\{f\in\Sym(V):f(\mathsf d)\Delta_\Phi=0\}
 =\langle\Sym(V)^W_+\rangle.
\end{equation*}
\end{lemma}

\begin{proof}
This is the classical harmonic description of coinvariants;
see~\cite{Steinberg}. A short proof in the present real reflection
setting is as follows. If $p$ is an invariant of positive degree,
$p(\mathsf d)\Delta_\Phi$ is anti-invariant of degree less than
$N$. Every anti-invariant polynomial is divisible by the product
of the reflecting linear forms $\Delta_\Phi$, so this derivative
vanishes. The invariant ideal therefore lies in the annihilator.

The quotient by homogeneous basic invariants is an Artinian
graded complete intersection, with one-dimensional socle in
degree $N$. The class of $\Delta_\Phi$ in this quotient is
nonzero: otherwise the inclusion just proved would give
$\Delta_\Phi(\mathsf d)\Delta_\Phi=0$. In real orthonormal
coordinates the latter scalar is a positive sum of squared
coefficients times factorials, so it is nonzero. Thus this
class spans the socle. Any nonzero ideal in an Artinian local
ring meets its socle (multiply a nonzero element by a maximal
possible power of the maximal ideal). The kernel of the map
$\bar f\mapsto f(\mathsf d)\Delta_\Phi$ is an ideal and cannot
meet that socle. It is zero, proving the reverse inclusion.
\end{proof}

\begin{theorem}[Exact root-coordinate threshold]\label{rc:thm:threshold}
For every root $\alpha$,
\begin{equation}\label{rc:eq:exact-threshold}
 \nu_\alpha
 =1+\#\{\beta\in\Phi^+:(\alpha,\beta)\ne0\}
 =N-\#(\Phi^+\cap\alpha^\perp)+1.
\end{equation}
The exponent is constant on Weyl orbits of roots.
\end{theorem}

\begin{proof}
Let $k=\#\{\beta>0:(\alpha,\beta)\ne0\}$.
The polynomial in $t$
\[
 \Delta_\Phi(x+t\alpha)
 =\prod_{\beta>0}\bigl((\beta,x)+t(\beta,\alpha)\bigr)
\]
has degree exactly $k$. Its leading coefficient is the nonzero
polynomial
\[
 \left(\prod_{\substack{\beta>0\\(\beta,\alpha)\ne0}}
 (\beta,\alpha)\right)
 \prod_{\substack{\beta>0\\(\beta,\alpha)=0}}(\beta,x).
\]
Therefore $\mathsf d_\alpha^{k+1}\Delta_\Phi=0$ and
$\mathsf d_\alpha^k\Delta_\Phi\ne0$. Lemma~\ref{rc:lem:annihilator}
identifies these statements with
$x_\alpha^{k+1}=0$ and $x_\alpha^k\ne0$ in
$\mathcal R_W$. This proves \eqref{rc:eq:exact-threshold}.
Weyl invariance follows either
from the counting formula or from the action on the coinvariant
algebra.
\end{proof}

\subsection{Partial flag varieties and minimal nilpotent orbits}
\label{geo:subsec:threshold}

The root count has an equivalent geometric meaning.
Let $\mathbf G$ be the adjoint complex semisimple group with
root system $\Phi$, choose a maximal torus $\mathbf T$ and a
Borel subgroup $\mathbf B$, and write $\mathbf P_J$ for the
standard parabolic attached to $J\subseteq I$.
We use the characteristic-map convention
\begin{equation}\label{geo:eq:characteristic}
 P/S_+P\xrightarrow{\sim}H^*(\mathbf G/\mathbf B;\mathbb C),
 \qquad
 x_\lambda\longmapsto
 c_1(\mathbf G\times^{\mathbf B}\mathbb C_{-\lambda}).
\end{equation}
Polynomial degree $k$ corresponds to cohomological degree $2k$;
see~\cite{BGG,Demazure,KumarFlag}. The following observation
also gives a geometric proof of Theorem~\ref{rc:thm:threshold}.

\begin{proposition}[The partial-flag interpretation]
\label{geo:prop:flag-threshold}
Let $0\ne\lambda\in Q$ and let $\lambda^+$ be the unique
dominant element of $W\lambda$. Set
\[
 J_\lambda=\{i\in I:
       \langle\lambda^+,\alpha_i^\vee\rangle=0\},
 \qquad
 d_\lambda=\dim_{\mathbb C}
       (\mathbf G/\mathbf P_{J_\lambda}).
\]
The exact nilpotence exponent of the class of $x_\lambda$
in $P/S_+P$ is
\begin{equation}\label{geo:eq:flag-threshold}
 \min\{m\geq1:x_\lambda^m=0\text{ in }P/S_+P\}
 =d_\lambda+1
 =1+\#\{\beta\in\Phi^+:(\lambda,\beta)\ne0\}.
\end{equation}
In particular, the root-power quotient $A_{\alpha,m}$ is
nonzero precisely when $m>d_\alpha$.
\end{proposition}

\begin{proof}
Write $\overline{x}_\mu$ for the class of $x_\mu$ in
$P/S_+P$. The action of $W$ on $P$ preserves the ideal
$S_+P$, since every element of $S=P^W$ is $W$-invariant.
Choose $w\in W$ such that $w\lambda=\lambda^+$. The induced
automorphism of the coinvariant algebra satisfies
\(
 w(\overline{x}_\lambda^{\,m})
 =\overline{x}_{\lambda^+}^{\,m}
(m\geq0).
\)
Thus $\overline{x}_\lambda$ and
$\overline{x}_{\lambda^+}$ have the same nilpotence exponent.
It suffices to calculate the exponent for the dominant weight
$\lambda^+$.

Set $J=J_\lambda$, $X=\mathbf G/\mathbf P_J$, and
$d=d_\lambda=\dim_{\mathbb C}X$. Let
\(
 p\colon \mathbf G/\mathbf B\longrightarrow X
\)
be the natural projection. We first specify the line bundle
used in the argument. Since $\mathbf G$ is adjoint and
$\lambda^+\in Q=X^*(\mathbf T)$, the weight $\lambda^+$
is an algebraic character of $\mathbf T$. By the definition
of $J$,
\(
 \langle\lambda^+,\alpha_j^\vee\rangle=0 (j\in J).
\)
These equalities are precisely the condition for this character
to extend to a character
\[
 \chi_{\lambda^+}\colon \mathbf P_J\longrightarrow\mathbb C^\times.
\]
Indeed, its restriction to the maximal torus of the derived
subgroup of the standard Levi subgroup is trivial. It therefore
defines a character of the Levi quotient by its derived
subgroup, and extends to $\mathbf P_J$ by being trivial on
the unipotent radical.

Let $\mathbb C_{-\lambda^+}$ be the one-dimensional
$\mathbf P_J$-module on which $h\in\mathbf P_J$ acts by
$\chi_{\lambda^+}(h)^{-1}$. Define
\[
 \mathcal L_{\lambda^+}
 =\mathbf G\times^{\mathbf P_J}\mathbb C_{-\lambda^+}.
\]
Explicitly, its total space is the quotient of
$\mathbf G\times\mathbb C$ by the equivalence relation
\[
 (g,z)\sim\bigl(gh,\chi_{\lambda^+}(h)z\bigr),
 \qquad h\in\mathbf P_J,
\]
and its projection is $[g,z]\mapsto g\mathbf P_J$.
This is the homogeneous line-bundle convention compatible
with the characteristic map in~\eqref{geo:eq:characteristic}.

Dominance and the definition of $J$ give
\(
 \langle\lambda^+,\alpha_i^\vee\rangle>0 (i\notin J).
\)
The ampleness criterion for homogeneous line bundles on a
partial flag variety therefore shows that
$\mathcal L_{\lambda^+}$ is ample on $X$; see
\cite{KumarFlag}. Here the vanishing of the pairings on $J$
ensures that the bundle is defined on $\mathbf G/\mathbf P_J$,
and strict positivity on the remaining simple coroots gives
ampleness there.

Put $\eta=c_1(\mathcal L_{\lambda^+})\in H^2(X;\mathbb C)$.
Pulling back the associated bundle gives
\[
 p^*\mathcal L_{\lambda^+}
 \cong\mathbf G\times^{\mathbf B}\mathbb C_{-\lambda^+}.
\]
The characteristic-map convention and naturality of the first
Chern class consequently identify
\[
 \overline{x}_{\lambda^+}
 =c_1(p^*\mathcal L_{\lambda^+})
 =p^*\eta
 \quad\text{in}\quad
 P/S_+P\cong H^*(\mathbf G/\mathbf B;\mathbb C).
\]
As in~\eqref{geo:eq:characteristic}, a polynomial of degree
$k$ corresponds to a cohomology class of degree $2k$.

We next check that $p^*$ is injective. The characteristic
presentations for the full and partial flag varieties give
\[
 H^*(X;\mathbb C)
 \cong P^{W_J}/S_+P^{W_J},
 \qquad
 H^*(\mathbf G/\mathbf B;\mathbb C)
 \cong P/S_+P,
\]
and identify $p^*$ with the map induced by
$P^{W_J}\hookrightarrow P$; see~\cite{KumarFlag}.
Its injectivity is equivalent to
\(
 P^{W_J}\cap S_+P=S_+P^{W_J}.
\)
To prove the nontrivial inclusion, take
$f\in P^{W_J}\cap S_+P$ and write
$f=\sum_{k=1}^q s_k g_k$ with $s_k\in S_+$ and $g_k\in P$.
Apply the averaging operator
\[
 \operatorname{Av}_J(g)
 =\frac{1}{|W_J|}\sum_{v\in W_J}v(g).
\]
Since $f$ and the $s_k$ are $W_J$-invariant, we obtain
\[
 f=\operatorname{Av}_J(f)
   =\sum_{k=1}^q s_k\operatorname{Av}_J(g_k)
   \in S_+P^{W_J}.
\]
The reverse inclusion is immediate. This proves injectivity
of $p^*$.

We now calculate the exact exponent of $\eta$. By ampleness,
there exists an integer $a\geq1$ such that
$\mathcal L_{\lambda^+}^{\otimes a}$ is very ample. Let
\(
 \iota\colon X\hookrightarrow\mathbb P^n
\)
be an associated projective embedding, so that
$\iota^*\mathcal O_{\mathbb P^n}(1)
\cong\mathcal L_{\lambda^+}^{\otimes a}$.
The degree of the embedded projective variety is positive,
and hence
\[
 a^d\int_X\eta^d
 =\int_X c_1(\iota^*\mathcal O_{\mathbb P^n}(1))^d
 =\deg\iota(X)>0.
\]
In particular, $\eta^d\ne0$. On the other hand, $X$ has
complex dimension $d$, so
\(
 \eta^{d+1}\in H^{2d+2}(X;\mathbb C)=0.
\)
Using the injectivity of $p^*$ and the identity
$\overline{x}_{\lambda^+}=p^*\eta$, we conclude that
\(
 \overline{x}_{\lambda^+}^{\,d}\ne0\) and \(\overline{x}_{\lambda^+}^{\,d+1}=0.
\)
No smaller positive power can vanish, since such vanishing
would force the $d$-th power to vanish as well. Therefore
the exact nilpotence exponent of $\overline{x}_{\lambda^+}$,
and hence of $\overline{x}_\lambda$, is $d+1$.

It remains to express $d$ as a root count. Let $\Phi_J$ be
the root subsystem generated by $\{\alpha_j:j\in J\}$,
and put $\Phi_J^+=\Phi_J\cap\Phi^+$. The tangent space of
$X$ at the identity coset is
$\operatorname{Lie}(\mathbf G)/\operatorname{Lie}(\mathbf P_J)$.
Its root-space decomposition has one one-dimensional summand
for each root in $\Phi^-\setminus\Phi_J^-$. Thus
\(
 d=|\Phi^+|-|\Phi_J^+|.
\)
For a positive root $\beta=\sum_{i\in I}n_i\alpha_i$,
where $n_i\geq0$, dominance gives
\[
 (\lambda^+,\beta)
 =\sum_{i\in I}n_i(\lambda^+,\alpha_i),
 \qquad
 (\lambda^+,\alpha_i)
 =\frac{(\alpha_i,\alpha_i)}{2}
       \langle\lambda^+,\alpha_i^\vee\rangle\geq0.
\]
The latter quantity is zero exactly when $i\in J$.
Consequently, $(\lambda^+,\beta)=0$ if and only if every
simple root in the support of $\beta$ belongs to $J$,
which is equivalent to $\beta\in\Phi_J^+$. Hence
\[
 d=\#\{\beta\in\Phi^+:(\lambda^+,\beta)\ne0\}.
\]
This count is unchanged when $\lambda^+$ is replaced by
$\lambda$. Indeed, for either vector $\mu$, the count equals
\[
 \frac12\#\{\beta\in\Phi:(\mu,\beta)\ne0\},
\]
because nonorthogonal roots occur in opposite pairs. The
latter number is $W$-invariant, since $W$ permutes $\Phi$
and preserves the inner product. We have therefore proved
\[
 \min\{m\geq1:\overline{x}_\lambda^{\,m}=0\}
 =d_\lambda+1
 =1+\#\{\beta\in\Phi^+:(\lambda,\beta)\ne0\}.
\]

Finally, for a root $\alpha$, the definition of $\nu_\alpha$
in the coinvariant algebra identifies it with $d_\alpha+1$. Applying that theorem with $\Xi=\{\alpha\}$
gives
\[
 A_{\alpha,m}\ne0
 \quad\Longleftrightarrow\quad
 m\geq\nu_\alpha=d_\alpha+1
 \quad\Longleftrightarrow\quad m>d_\alpha.
\]
This proves the final assertion.
\end{proof}
\begin{corollary}[Minimal nilpotent orbits and nonvanishing]
\label{geo:cor:minimal-orbit}
Assume that $\Phi$ is irreducible. Let $\mathfrak g_\Phi$
be the complex simple Lie algebra with root system $\Phi$,
and let $h^\vee(\Phi)$ be its dual Coxeter number.
If $\alpha$ is long, then
\begin{equation}\label{geo:eq:long-orbit}
 \nu_\alpha
 =\dim_{\mathbb C}\mathcal O_{\min}(\mathfrak g_\Phi)
 =2h^\vee(\Phi)-2.
\end{equation}
If $\Phi$ has two root lengths and $\alpha$ is short, then
\begin{equation}\label{geo:eq:short-orbit}
 \nu_\alpha
 =\dim_{\mathbb C}\mathcal O_{\min}(\mathfrak g_{\Phi^\vee})
 =2h^\vee(\Phi^\vee)-2.
\end{equation}
Here $\mathcal O_{\min}$ denotes the minimal nonzero
nilpotent orbit, and $\mathfrak g_{\Phi^\vee}$ is the
Langlands-dual simple Lie algebra. In either case
$A_{\alpha,m}\ne0$ exactly when $m$ is at least the
displayed orbit dimension.
\end{corollary}

\begin{proof}
For a long root, its dominant Weyl conjugate is the
highest root $\theta$. The homogeneous space
$\mathbf G/\mathbf P_{J_\theta}$ is the adjoint variety
$\mathbb P(\mathcal O_{\min}(\mathfrak g_\Phi))$.
Thus its dimension is one less than the orbit dimension,
and Proposition~\ref{geo:prop:flag-threshold} gives the
first equality in~\eqref{geo:eq:long-orbit}. The second
is the classical formula of~\cite[Theorem 1]{WangMin};
its Lemma 1 also gives the same equality directly from
the root count in Theorem~\ref{rc:thm:threshold}.

For a short root, $\alpha^\vee$ is long in $\Phi^\vee$.
The bijection $\beta\mapsto\beta^\vee$ preserves positive
roots and orthogonality. The root-count formula therefore
gives
$\nu_\alpha(\Phi)=\nu_{\alpha^\vee}(\Phi^\vee)$.
Apply the long-root case to $\Phi^\vee$. The assertion
about $A_{\alpha,m}$ is
Theorem~\ref{rc:thm:cyclotomic}.
\end{proof}

For reference, these formulas give the following thresholds:
\begin{equation*}\label{geo:eq:threshold-table}
 \begin{array}{c|c|c}
 \Phi&\nu_{\rm long\ (or\ only)}&\nu_{\rm short}\\ \hline
 A_r&2r&--\\
 B_r&4r-4&2r\\
 C_r&2r&4r-4\\
 D_r&4r-6&--\\
 E_6&22&--\\
 E_7&34&--\\
 E_8&58&--\\
 F_4&16&16\\
 G_2&6&6
 \end{array}
\end{equation*}
The usual irreducible ranges are understood:
$r\geq1$ for $A_r$, $r\geq2$ for $B_r,C_r$,
and $r\geq4$ for $D_r$.
The formulas concern the polynomial root-power quotients.
Their completed zero-residue interpretation is given in
Section~\ref{ii:sec:comparison}; transfer to a Hecke summand
uses the hypotheses of Corollary~\ref{rc:cor:sharp-hecke}.
The geometric input is classical; the conclusion identifies
its numerical invariant with the threshold for our
root-power quotients.

\begin{example}\label{rc:ex:thresholds}
In type $A_r$, the root $\epsilon_1-\epsilon_2$ is nonorthogonal
to itself and to the $2(r-1)$ positive roots involving exactly
one of $1,2$. Hence $\nu_\alpha=2r$ for every root.
For $r\geq2$, the partial flag variety for the highest root is
the incidence variety
\[
 \{(L,H):L\subseteq H\subseteq\mathbb C^{r+1},
             \ \dim L=1,\ \dim H=r\}.
\]
It has dimension $r+(r-1)=2r-1$, which gives the same
exponent by Proposition~\ref{geo:prop:flag-threshold}.
For $r=1$ the variety is $\mathbb P^1$ and the exponent
is two. The rank-two values are
\[
 \begin{array}{c|cccc}
 \text{type}&A_1\times A_1&A_2&B_2&G_2\\ \hline
 \nu_\alpha&2&4&4&6.
 \end{array}
\]
For a single-root quotient the spanning hypothesis of
Theorem~\ref{rc:thm:cyclotomic} is automatic in irreducible type.
In reducible type a coordinate must be imposed in every
component to obtain the stated finiteness conclusion.
\end{example}

\section{Root-power quotients, Weyl invariants, and root-orbit geometry}
\label{str:sec:content}

The polynomial nil-Hecke algebra is a matrix algebra over Weyl
invariants. We first identify the central ideal of a root-power
quotient by content and alternation, and prove equality at successive
even and odd exponents. We then use apolarity to describe its center
by vanishing conditions on the projective Weyl orbit of the root.
The resulting theorem also gives the partial-flag nonvanishing
criterion and a differential criterion for the center to be Gorenstein.
We then remove the discriminant from the inverse system and
identify the first nonzero quotient in every irreducible type.
Throughout this section the ground field is $\mathbb C$.

\subsection{Content and alternation}
Put
\[
 P=\Sym(V),\qquad S=P^W,\qquad d=|W|,
 \qquad \Delta_\Phi=\prod_{\beta\in\Phi^+}x_\beta.
\]
Here $x_\alpha=1\otimes\alpha\in V\subseteq P$, linear
coordinates have degree one, and $S_+$ is the ideal of
positive-degree invariants. The notation $\widehat S$ means
the $S_+$-adic completion from
Lemma~\ref{lem:completed-nilhecke}. We distinguish an ideal
in $S$ from its extension to
$\widehat S$ by a superscript ${\rm pol}$ when needed.

Define
\[
 J^{\rm pol}_{\alpha,m}=c_S(x_\alpha^m),\qquad
 I_{\alpha,m}=(x_{w\alpha}^m:w\in W)\subseteq P.
\]
For a $W$-module of polynomials, a superscript $\det$ means the
anti-invariant subspace. Let
\[
 \operatorname{Alt}(f)=\sum_{w\in W}\det(w)w(f),\qquad
 \operatorname{tr}_\Phi(f)=\frac{\operatorname{Alt}(f)}{\Delta_\Phi}.
\]
Every anti-invariant polynomial is divisible by $\Delta_\Phi$,
so $\operatorname{tr}_\Phi$ takes values in $S$. This divided-alternation functional is
$S$-linear and homogeneous of degree $-N$.

\begin{theorem}[Content and the anti-invariant part of an orbit ideal]
\label{str:thm:alternation}
For every root $\alpha$ and integer $m\geq1$,
\begin{align}
 J_{\alpha,m}&=J^{\rm pol}_{\alpha,m}\widehat S,
 \notag\\
 J^{\rm pol}_{\alpha,m}
 &=\bigl(\operatorname{tr}_\Phi(x_\alpha^m f):f\in P\bigr),
 \label{str:eq:content-trace}\\
 \Delta_\Phi J^{\rm pol}_{\alpha,m}
 &=I_{\alpha,m}\cap P^{\det}.
 \label{str:eq:content-orbit}
\end{align}
In particular, the cyclotomic center can be computed from one
coordinate vector or from the anti-invariant part of the orbit ideal.
If $\Phi$ is irreducible, $S/J^{\rm pol}_{\alpha,m}$ is
finite-dimensional and equals its completion.
\end{theorem}

\begin{proof}
The first assertion is Lemma~\ref{str:lem:content} applied to the
finite free extension $P/S$ and its completion. The pairing
\[
 P\times P\longrightarrow S,\qquad
 (f,g)\longmapsto\operatorname{tr}_\Phi(fg)
\]
is perfect. We recall why the polynomial, rather than just completed,
statement holds. Choose a homogeneous $S$-basis of $P$ lifting a
basis of the coinvariant algebra. Modulo $S_+$ the displayed
pairing is the perfect top-degree pairing of that graded complete
intersection, with socle generated by $\Delta_\Phi$; see
Lemma~\ref{rc:lem:annihilator} and the classical nil-Hecke pairing
in~\cite{KostantKumar}. If $N=|\Phi^+|$, the sum of the basis
degrees is $dN/2$, by the symmetry of the coinvariant Hilbert
series. Consequently the determinant of the pairing matrix is
homogeneous of degree $2(dN/2)-dN=0$. Its nonzero specialization
at $S_+$ makes it a nonzero constant. Thus every $S$-linear
functional on $P$ is $f\mapsto\operatorname{tr}_\Phi(gf)$
for a unique $g\in P$.
This proves~\eqref{str:eq:content-trace}.

The ideal $I_{\alpha,m}$ is $W$-stable. In characteristic zero its
anti-invariant part is the image of $\operatorname{Alt}$ on that
ideal. For each $w\in W$,
\[
 \operatorname{Alt}(x_{w\alpha}^m f)
 =\det(w)\operatorname{Alt}(x_\alpha^m w^{-1}(f)).
\]
Its image is therefore generated, as an $S$-module, by the
alternants on the right of~\eqref{str:eq:content-trace}. This gives
\eqref{str:eq:content-orbit}.

Finally, $I_{\alpha,m}\subseteq P J^{\rm pol}_{\alpha,m}$,
because $x_\alpha^m$ belongs to this extended content ideal and
the latter is $W$-stable. In irreducible type the orbit $W\alpha$
spans $V$. The only common zero of its powers is the origin, so
$P/I_{\alpha,m}$ is finite-dimensional. The Reynolds operator
gives $P J^{\rm pol}_{\alpha,m}\cap S=J^{\rm pol}_{\alpha,m}$.
It follows that $S/J^{\rm pol}_{\alpha,m}$ is finite-dimensional.
It is a homogeneous quotient supported at $S_+$, so completion
does not change it.
\end{proof}

\subsection{Parity of root-power ideals}

We retain the polynomial quotient
$A_{\alpha,m}=NH_W/\langle M_{x_\alpha^m}\rangle$.
The following equality holds before completion.

\begin{theorem}[Equality at successive even and odd exponents]
\label{str:thm:parity}
For every root $\alpha$ and every $a\geq1$,
\begin{equation*}\label{str:eq:parity}
 \langle M_{x_\alpha^{2a}}\rangle
 =\langle M_{x_\alpha^{2a+1}}\rangle
 \quad\text{in }NH_W,
 \qquad J_{\alpha,2a}=J_{\alpha,2a+1}.
\end{equation*}
The natural quotient map $A_{\alpha,2a+1}\to A_{\alpha,2a}$
is an isomorphism. The same equality holds for the polynomial
content ideals. In a quotient imposing several root-coordinate
powers, each exponent at least two can independently be replaced
by its largest even predecessor without changing the quotient.
\end{theorem}

\begin{proof}
Let $x=x_\alpha$ and let
$\partial_\alpha=(1-s_\alpha)/x$ act on $P$ or its completion.
This operator belongs to the nil-Hecke endomorphism algebra:
one can conjugate a simple divided difference by a Weyl operator.
Since $s_\alpha(x)=-x$, its mixed relation gives
\begin{equation}\label{str:eq:parity-identity}
 \partial_\alpha x^{2a+1}+x^{2a+1}\partial_\alpha=2x^{2a}
\end{equation}
as an identity of operators. Thus the two-sided ideal generated
by $x^{2a+1}$ contains $x^{2a}$. The opposite containment follows
by multiplication by $x$. Lemma~\ref{str:lem:content}, and the
polynomial nil-Hecke endomorphism theorem, give equality of the
polynomial content ideals as well. Base change gives equality
of the completed ideals. Equality of sums proves the last assertion.
\end{proof}

\begin{corollary}\label{str:cor:even-threshold}
The exact exponent $\nu_\alpha$ of
Theorem~\ref{rc:thm:threshold} is even. Equivalently,
$\#\{\beta\in\Phi^+:(\alpha,\beta)\ne0\}$ is odd.
\end{corollary}

\begin{proof}
On the coinvariant algebra, equation~\eqref{str:eq:parity-identity}
implies that the vanishing of an odd power $x^{2a+1}$ forces the
vanishing of $x^{2a}$. The minimal vanishing exponent cannot be
odd. The counting assertion follows from
Theorem~\ref{rc:thm:threshold}. Corollary~\ref{geo:cor:minimal-orbit}
gives another explanation of this evenness: in each irreducible
component the threshold is twice a dual Coxeter number minus two.
The stronger equality of the quotient ideals themselves is
the content of Theorem~\ref{str:thm:parity}.
\end{proof}

\subsection{Apolarity and projective root-orbit geometry}
\label{str:subsec:rootpower-geometry}

The matrix description of a root-power quotient and the alternating
description of its coefficient ideal admit a common geometric formulation.
The relevant geometry consists of the projective Weyl orbit of a root,
with a degree-dependent vanishing condition at its points.  We use the
classical apolar pairing between symmetric algebras; the relationship
between powers of linear forms and fat points is treated by
Emsalem--Iarrobino~\cite{EmsalemIarrobino} and
Ardila--Postnikov~\cite{ArdilaPostnikovPower}. We prove the required
Weyl-equivariant version below, including the determinant component
and its degree shift.

Assume in this subsection that $\Phi$ is irreducible.
Retain $P$, $S$, $d$, and $\Delta_\Phi$ as above, and put
\[
 N=|\Phi^+|,\qquad T=\Sym(V^*).
\]
Thus $T$ is the ring of polynomial functions on $V$.
Here $\mathsf d_\beta$ and $p(\mathsf d)$ act on $T$ as defined
in the introduction. This action uses the natural pairing between
$V$ and $V^*$; the bilinear form is needed only when we identify
$P$ with $T$ or write the dual discriminant $\Delta_\Phi^\vee$.
The \emph{apolar pairing} is
\[
 \langle p,F\rangle_{\rm ap}
 =\bigl(p(\mathsf d)F\bigr)(0)
 \qquad(p\in P,\ F\in T).
\]
It is $W$-invariant and perfect between $P_n$ and $T_n$ for every
$n\geq0$. To see its normalization explicitly, choose a basis
$\beta_1,\ldots,\beta_r$ of $V$ and let $y_1,\ldots,y_r$ be the
natural dual coordinates in $V^*$. Then
$\mathsf d_{\beta_i}=\partial/\partial y_i$ on $T$, and
\[
 \left\langle
 x_{\beta_1}^{a_1}\cdots x_{\beta_r}^{a_r},
 y_1^{b_1}\cdots y_r^{b_r}
 \right\rangle_{\rm ap}
 =\begin{cases}
   a_1!\cdots a_r!,& (a_1,\ldots,a_r)=(b_1,\ldots,b_r),\\
   0,&\text{otherwise}.
  \end{cases}
\]
We omit the subscript $\rm ap$ below when there is no ambiguity.

For a root $\alpha$ and $m\geq1$, write
\[
 Z(\alpha,m)=S/J^{\rm pol}_{\alpha,m}.
\]
The second argument is the actual exponent; hence the rank-two
notation used below is $Z_{\alpha,a}=Z(\alpha,2a)$.
Define the anti-invariant inverse system
\begin{equation}\label{str:eq:rootpower-inverse-system}
 \mathcal D_{\alpha,m}
 =\left\{F\in T:
   wF=\det(w)F\ (w\in W),\quad
   \mathsf d_{w\alpha}^{\,m}F=0\ (w\in W)\right\}.
\end{equation}
Finally, let
\[
 \Sigma_\alpha=\{[w\alpha]:w\in W\}\subseteq\mathbb P(V),
\]
and denote the homogeneous ideal of $\xi\in\Sigma_\alpha$
in $T$ by $\mathfrak p_\xi$.  We use the convention
$\mathfrak p_\xi^0=T$.

\begin{theorem}[Root-power quotients and the geometry of a root orbit]
\label{str:thm:rootpower-geometry}
With the preceding notation, the following statements hold.
\begin{enumerate}[label=\textup{(\roman*)}]
\item
The algebra $Z(\alpha,m)$ is finite-dimensional and is either zero
or a graded local algebra with residue field $\mathbb C$.  There are
isomorphisms
\begin{equation}\label{str:eq:rootpower-matrix}
 A_{\alpha,m}\cong\Mat_d(Z(\alpha,m)),
 \qquad Z(A_{\alpha,m})\cong Z(\alpha,m).
\end{equation}
Multiplication by $\Delta_\Phi$ induces, for every $k\geq0$,
\begin{equation}\label{str:eq:rootpower-alternating-piece}
 Z(\alpha,m)_k
 \xrightarrow{\ \sim\ }
 (P/I_{\alpha,m})_{N+k}^{\det}.
\end{equation}
These are isomorphisms of the corresponding graded $S$-modules,
with degrees raised by $N$; the right-hand side is an isotypic
component, not in general a subalgebra.

\item
The vector-space dual of each homogeneous component of the center
has the descriptions
\begin{equation}\label{str:eq:rootpower-fat-points}
 \begin{split}
 Z(\alpha,m)_k^*
 &\cong (\mathcal D_{\alpha,m})_{N+k}\\
 &=\left(
    \bigcap_{\xi\in\Sigma_\alpha}
      \mathfrak p_\xi^{\,\max\{N+k-m+1,0\}}
   \right)_{N+k}^{\det}.
 \end{split}
\end{equation}
The pairing defining the first isomorphism is
\begin{equation}\label{str:eq:rootpower-pairing}
 (\overline f,F)\longmapsto
       (\Delta_\Phi f)(\mathsf d)F(0).
\end{equation}
The space $\mathcal D_{\alpha,m}$ is stable under the differential
action of $S$, and the full ideal and center are recovered by
\begin{equation}\label{str:eq:rootpower-differential-annihilator}
 J^{\rm pol}_{\alpha,m}
   =\Ann_S(\mathcal D_{\alpha,m}),
 \qquad
 Z(\alpha,m)
   =S/\Ann_S(\mathcal D_{\alpha,m}),
\end{equation}
where annihilation means $s(\mathsf d)F=0$ for every
$F\in\mathcal D_{\alpha,m}$.

\item
Let $\alpha^+$ be the dominant conjugate of $\alpha$, and set
\[
 J(\alpha)=\{i:\langle\alpha^+,\alpha_i^\vee\rangle=0\},
 \qquad
 d_\alpha=\dim_{\mathbb C}
            \bigl(\mathbf G/\mathbf P_{J(\alpha)}\bigr).
\]
Here $\mathbf G$ and its standard parabolics are as in
Proposition~\ref{geo:prop:flag-threshold}.  Then
\begin{equation}\label{str:eq:rootpower-geometric-threshold}
 A_{\alpha,m}\ne0
 \quad\Longleftrightarrow\quad
 m>d_\alpha,
 \qquad
 d_\alpha=\#\{\beta\in\Phi^+:(\alpha,\beta)\ne0\}.
\end{equation}
Equivalently, $d_\alpha$ is the largest integer $j$ for which
$c_1(\mathcal L_{\alpha^+})^j\ne0$ on this partial flag
variety. Here the ample line bundle is
\[
 \mathcal L_{\alpha^+}
   =\mathbf G\times^{\mathbf P_{J(\alpha)}}
                        \mathbb C_{-\alpha^+}.
\]

\item
If $Z(\alpha,m)\ne0$, then
\begin{equation}\label{str:eq:rootpower-socle-duality}
 \dim_{\mathbb C}\operatorname{Soc}Z(\alpha,m)
 =\dim_{\mathbb C}
    \frac{\mathcal D_{\alpha,m}}
          {S_+(\mathsf d)\mathcal D_{\alpha,m}}.
\end{equation}
In particular, $Z(\alpha,m)$ is Gorenstein if and only if
$\mathcal D_{\alpha,m}$ is cyclic for this differential
$S$-action.  This is also equivalent to the matrix algebra
$A_{\alpha,m}$ being self-injective, Frobenius, or symmetric.
\end{enumerate}
\end{theorem}

\begin{proof}
By Theorem~\ref{rc:thm:cyclotomic}, the polynomial quotient is
\[
 A_{\alpha,m}\cong\operatorname{Mat}_d
                   (S/J^{\rm pol}_{\alpha,m}).
\]
Theorem~\ref{str:thm:alternation} gives
\begin{equation}\label{str:eq:rootpower-content-alternation}
 \Delta_\Phi J^{\rm pol}_{\alpha,m}
     = I_{\alpha,m}\cap P^{\det},
 \qquad P^{\det}=\Delta_\Phi S.
\end{equation}
Because $I_{\alpha,m}$ is $W$-stable and averaging over the finite
group $W$ is exact in characteristic zero,
\[
 (P/I_{\alpha,m})^{\det}
 =P^{\det}/(I_{\alpha,m}\cap P^{\det}).
\]
Together these identities prove
\eqref{str:eq:rootpower-alternating-piece}.

The roots in $W\alpha$ span $V$, since the reflection
representation of an irreducible Weyl group is irreducible.
Choose a basis $\beta_1,\ldots,\beta_{\dim V}$ from this orbit.
The ideal $I_{\alpha,m}$ contains all $x_{\beta_i}^m$, so
$P/I_{\alpha,m}$ is finite-dimensional.  Its anti-invariant
component, and hence $S/J^{\rm pol}_{\alpha,m}$, is therefore
finite-dimensional.  This quotient is homogeneous.  If it is
nonzero, its degree-zero part is $\mathbb C$ and its
positive-degree ideal is its nilpotent maximal ideal.  Completion
at $S_+$ consequently has no further effect on it.  This proves
\textup{(i)}.

For \textup{(ii)}, use the perfect $W$-invariant apolar pairing
between $P_n$ and $T_n$ defined above.
If $n\geq m$, then for $p\in P_{n-m}$,
\[
 \langle x_\beta^m p,F\rangle
      =\langle p,\mathsf d_\beta^mF\rangle.
\]
The perfection of the pairing in degree $n-m$ shows that the
orthogonal complement of $(I_{\alpha,m})_n$ consists exactly
of the elements annihilated by every
$\mathsf d_{w\alpha}^{\,m}$.  For $n<m$ the same assertion
holds because both restrictions are vacuous.  Taking the
determinant isotypic component is compatible with duality,
as the determinant character is self-dual.  Composing with
\eqref{str:eq:rootpower-alternating-piece} gives
\[
 Z(\alpha,m)_k^*\cong(\mathcal D_{\alpha,m})_{N+k}
\]
and the pairing~\eqref{str:eq:rootpower-pairing}.

We give the local-coordinate verification of the fat-point
description.  Fix $\beta\ne0$ and choose a basis of $V$ whose
first vector is $\beta$, with dual coordinates $y_1,\ldots,y_r$.
For a homogeneous polynomial $F$ of degree $n$, write
\[
 F=\sum_{j=0}^n y_1^j F_{n-j}(y_2,\ldots,y_r).
\]
Since the characteristic is zero,
$\mathsf d_\beta^mF=\partial_{y_1}^mF=0$ if and only if
every nonzero summand has $j<m$.  For $n\geq m$ this is
equivalent to every summand having degree at least $n-m+1$
in $y_2,\ldots,y_r$, or
\[
 F\in(y_2,\ldots,y_r)^{n-m+1}
       =\mathfrak p_{[\beta]}^{n-m+1}.
\]
For $n<m$ there is no condition, which is expressed by the
exponent zero.  Applying this argument to every line in
$\Sigma_\alpha$ proves~\eqref{str:eq:rootpower-fat-points}.
Thus the right-hand side is the anti-invariant part of the
space of forms vanishing to the indicated order at every
point of the projective root orbit.

If $s\in S$, the operator $s(\mathsf d)$ commutes with the
directional derivatives and with the $W$-action.  It therefore
preserves $\mathcal D_{\alpha,m}$.  Moreover,
\[
 \langle\Delta_\Phi sf,F\rangle
   =\langle\Delta_\Phi f,s(\mathsf d)F\rangle.
\]
The pairing consequently identifies this differential module
with the vector-space dual of the regular $S$-module
$Z(\alpha,m)$.  An element of $S$ annihilates that dual
if and only if it annihilates $Z(\alpha,m)$, and the latter
annihilator is $J^{\rm pol}_{\alpha,m}$.  This proves
\eqref{str:eq:rootpower-differential-annihilator}.  Notice that
differentiation lowers degree: the componentwise duality
above specifies the grading convention explicitly.

For \textup{(iii)}, use the fixed $W$-invariant bilinear form.
Recall that
\[
 \Delta_\Phi^\vee(v)
      =\prod_{\beta\in\Phi^+}(\beta,v)\in T_N.
\]
Then $T^{\det}=\Delta_\Phi^\vee T^W$, so that
$T_N^{\det}=\mathbb C\Delta_\Phi^\vee$.  A homogeneous
quotient of $S$ is nonzero exactly when its degree-zero
part is nonzero.  By \textup{(ii)}, this is equivalent to
$\Delta_\Phi^\vee\in\mathcal D_{\alpha,m}$.  The
anti-invariance of $\Delta_\Phi^\vee$ shows that it suffices
to test the derivative in the direction $\alpha$.

The polynomial
\[
 \Delta_\Phi^\vee(v+t\alpha)
    =\prod_{\beta\in\Phi^+}
           \bigl((\beta,v)+t(\beta,\alpha)\bigr)
\]
has exact degree
$b_\alpha=\#\{\beta\in\Phi^+:(\beta,\alpha)\ne0\}$
in $t$.  Its leading coefficient is a nonzero constant
times the product of the remaining nonzero linear forms
$(\beta,v)$, and hence is not the zero polynomial.
It follows that
\[
 \mathsf d_\alpha^m\Delta_\Phi^\vee=0
       \quad\Longleftrightarrow\quad m>b_\alpha.
\]
Proposition~\ref{geo:prop:flag-threshold} identifies
$b_\alpha$ with $d_\alpha$ and proves that
$\mathcal L_{\alpha^+}$ is ample, with
\[
 c_1(\mathcal L_{\alpha^+})^{d_\alpha}\ne0,
 \qquad
 c_1(\mathcal L_{\alpha^+})^{d_\alpha+1}=0.
\]
This proves \textup{(iii)} and makes the nonvanishing
criterion visible directly in the inverse system.

Finally, let $Z=Z(\alpha,m)\ne0$ and let $\mathfrak m$
be its maximal ideal.  Under the duality just constructed,
the orthogonal complement in $Z$ of
$S_+(\mathsf d)\mathcal D_{\alpha,m}$ consists of the
elements annihilated by $\mathfrak m$.  This is exactly
$\operatorname{Soc}Z$.  Hence
\[
 \bigl(\mathcal D_{\alpha,m}/
        S_+(\mathsf d)\mathcal D_{\alpha,m}\bigr)^*
       \cong\operatorname{Soc}Z.
\]
Nakayama's lemma identifies the dimension on the right of
\eqref{str:eq:rootpower-socle-duality} with the minimal
number of generators of the differential module.
A commutative Artinian local algebra is Gorenstein exactly
when its socle is one-dimensional, equivalently when its
regular-module dual is cyclic.  In this case a Frobenius
functional gives a symmetric form $(z,z')\mapsto\ell(zz')$.
The usual matrix trace transports this form to
$\Mat_d(Z)$.  Conversely, self-injectivity
passes through Morita equivalence to $Z$, and for an
Artinian commutative local algebra it implies the same
one-dimensional-socle condition.  This proves \textup{(iv)}.
\end{proof}

\subsection{Removing the discriminant and the first nonzero quotient}

The anti-invariant inverse system can be reduced to a condition on
invariant polynomials. This reduction will give calculations in
arbitrary rank in the next section. For $0\ne H\in T$ and
$0\ne\beta\in V$, define its directional degree by
\[
 \deg_\beta H=\deg_t H(v+t\beta),
\]
where the expression on the right is a polynomial in $t$ with
coefficients in $T$. Equivalently, this is the largest $j$ for
which $\mathsf d_\beta^jH\ne0$. The characteristic-zero assumption
is used in this equivalence. Put
\[
 \mathcal E_{\alpha,b}
 =\{H\in T^W:\mathsf d_{w\alpha}^{\,b+1}H=0
                         \text{ for every }w\in W\}
 \qquad(b\geq0).
\]
We regard this as a graded vector space. Multiplication of its
elements need not preserve it.

\begin{theorem}[Discriminant removal]
\label{str:thm:discriminant-removal}
Assume that $\Phi$ is irreducible and let
$\nu_\alpha=1+\#\{\beta\in\Phi^+:(\alpha,\beta)\ne0\}$.
For $m\geq\nu_\alpha$, multiplication by the dual discriminant
gives an equality of graded vector spaces
\begin{equation}\label{str:eq:discriminant-removal}
 \mathcal D_{\alpha,m}
 =\Delta_\Phi^\vee\mathcal E_{\alpha,m-\nu_\alpha}.
\end{equation}
Consequently, for every $k\geq0$,
\begin{equation}\label{str:eq:invariant-directional-dual}
 Z(\alpha,m)_k^*
 \cong(\mathcal E_{\alpha,m-\nu_\alpha})_k.
\end{equation}
Equivalently, if
$I_{\alpha,b+1}=(x_{w\alpha}^{b+1}:w\in W)\subseteq P$, then
\begin{equation}\label{str:eq:shifted-invariant-hilbert}
 H_{Z(\alpha,m)}(q)
 =H_{(P/I_{\alpha,m-\nu_\alpha+1})^W}(q).
\end{equation}
These statements identify graded vector spaces and Hilbert series.
The multiplication in the center is recovered from the differential
action on $\mathcal D_{\alpha,m}$ in
Theorem~\ref{str:thm:rootpower-geometry}.
\end{theorem}

\begin{proof}
Every anti-invariant polynomial is uniquely $\Delta_\Phi^\vee H$
with $H\in T^W$. For a root $\beta\in W\alpha$, the root-count
calculation in Theorem~\ref{rc:thm:threshold} gives
\[
 \deg_\beta\Delta_\Phi^\vee=\nu_\alpha-1.
\]
For nonzero $H$, the leading coefficients of
$\Delta_\Phi^\vee(v+t\beta)$ and $H(v+t\beta)$ are nonzero
polynomials. Their product is nonzero because $T$ is a domain.
Thus
\[
 \deg_\beta(\Delta_\Phi^\vee H)
 =\nu_\alpha-1+\deg_\beta H.
\]
The defining equation
$\mathsf d_\beta^m(\Delta_\Phi^\vee H)=0$ is therefore
equivalent to $\deg_\beta H\leq m-\nu_\alpha$, or to
$\mathsf d_\beta^{m-\nu_\alpha+1}H=0$.
Applying this to every $\beta\in W\alpha$ proves
\eqref{str:eq:discriminant-removal}.
The degree shift in
Theorem~\ref{str:thm:rootpower-geometry}\textup{(ii)} cancels
the degree $N$ of the discriminant and gives
\eqref{str:eq:invariant-directional-dual}.

The ordinary apolar pairing identifies the inverse system of
$I_{\alpha,b+1}$ with the polynomials killed by all
$\mathsf d_{w\alpha}^{b+1}$. Taking invariants commutes with
this duality, by averaging in characteristic zero. Its invariant
part is exactly $\mathcal E_{\alpha,b}$. This proves
\eqref{str:eq:shifted-invariant-hilbert}. Multiplication by
$\Delta_\Phi^\vee$ is used here only as a vector-space map;
it does not identify the two differential actions or the two
quotient algebra structures.
\end{proof}

\begin{proposition}[The differential action after discriminant removal]
\label{str:prop:conjugated-action}
For $s\in S$ and $H\in T^W$, put
\[
 \mathscr D_s(H)=(\Delta_\Phi^\vee)^{-1}
                     s(\mathsf d)(\Delta_\Phi^\vee H).
\]
This is a polynomial in $T^W$, and $s\mapsto\mathscr D_s$
is an action of $S$ on $T^W$. For $m\geq\nu_\alpha$ it
preserves $\mathcal E_{\alpha,m-\nu_\alpha}$, and
\[
 J^{\rm pol}_{\alpha,m}
 =\{s\in S:\mathscr D_s(H)=0\text{ for every }
                    H\in\mathcal E_{\alpha,m-\nu_\alpha}\}.
\]
Thus discriminant removal also gives the center algebra when
equipped with this conjugated differential action.
\end{proposition}

\begin{proof}
An invariant constant-coefficient differential operator preserves
anti-invariants. Its value on $\Delta_\Phi^\vee H$ is therefore
divisible by $\Delta_\Phi^\vee$, with invariant polynomial
quotient. Conjugating composition gives
$\mathscr D_{st}=\mathscr D_s\mathscr D_t$ and
$\mathscr D_1=1$. The space $\mathcal D_{\alpha,m}$ is stable
under $S(\mathsf d)$ by
Theorem~\ref{str:thm:rootpower-geometry}\textup{(ii)}.
Transport this action through
$\mathcal D_{\alpha,m}=\Delta_\Phi^\vee
\mathcal E_{\alpha,m-\nu_\alpha}$ and use the annihilator
identity in that theorem. This proves both assertions.
\end{proof}

\begin{proposition}[A criterion for a weighted invariant basis]
\label{str:prop:weighted-invariants}
Suppose $T^W=\mathbb C[f_1,\ldots,f_l]$, with each $f_j$
homogeneous of polynomial degree $d_j$. Fix a root $\alpha$.
Assume that $\deg_\alpha f_j=2\omega_j$, where the
$\omega_j$ are positive integers and $\omega_1=1$.
Let $a_j$ be the leading coefficient of $f_j(v+t\alpha)$
as a polynomial in $t$. Suppose $a_1$ is a nonzero constant
and $a_2,\ldots,a_l$ are algebraically independent.
Then, for every nonzero polynomial $F$,
\[
 \deg_\alpha F(f_1,\ldots,f_l)=2\deg_\omega F,
\]
where $\deg_\omega$ assigns weight $\omega_j$ to the $j$th
argument. Put $r=\lfloor(m-\nu_\alpha)/2\rfloor$. For
$m\geq\nu_\alpha$, a basis of the graded dual center is
indexed by
\[
 f_1^{b_1}\cdots f_l^{b_l},\qquad
 b_j\geq0,\quad\sum_j\omega_jb_j\leq r.
\]
In particular,
\begin{equation}\label{str:eq:weighted-hilbert-series}
 \sum_{r\geq0}H_{Z(\alpha,\nu_\alpha+2r)}(q)t^r
 =\frac{1}{(1-t)\prod_{j=1}^l(1-t^{\omega_j}q^{d_j})}.
\end{equation}
The monomials describe a vector-space basis; multiplication in
the center is given by Proposition~\ref{str:prop:conjugated-action}.
\end{proposition}

\begin{proof}
Let $F_L$ be the highest part of $F$ for the auxiliary weights.
The coefficient of $t^{2L}$ in the translated polynomial is
$F_L(a_1,\ldots,a_l)$. It is nonzero. Indeed, in a polynomial
of fixed weighted degree the exponents of the last $l-1$
arguments determine the exponent of the first, since
$\omega_1=1$. Substituting a nonzero constant in that first
argument therefore cannot make a nonzero weighted homogeneous
polynomial vanish identically. Algebraic independence of the
remaining leading coefficients finishes the argument.
Weyl invariance gives the same directional degree at every root
in $W\alpha$. Discriminant removal now gives the displayed
monomial basis. A monomial of auxiliary weight $w$ contributes
in every $r\geq w$; summing these geometric series proves
\eqref{str:eq:weighted-hilbert-series}.
\end{proof}

\begin{corollary}[The first nonzero root-power quotient]
\label{str:cor:first-rootpower}
For every root in an irreducible root system,
\begin{equation}\label{str:eq:first-rootpower}
 Z(\alpha,\nu_\alpha)=\mathbb C,
 \qquad
 A_{\alpha,\nu_\alpha}=A_{\alpha,\nu_\alpha+1}
 \cong\End_{\mathbb C}(\mathcal R_W)
 \cong\Mat_{|W|}(\mathbb C).
\end{equation}
In particular, the first nonzero quotient is semisimple and
has dimension $|W|^2$. Its polynomial content ideal is $S_+$.
\end{corollary}

\begin{proof}
At $m=\nu_\alpha$, Theorem~\ref{str:thm:discriminant-removal}
identifies the graded dual of the center with
$\mathcal E_{\alpha,0}$. A polynomial in this space has zero
directional derivative in every direction in $W\alpha$.
Those roots span $V$, so all its first derivatives vanish.
It is constant. Thus $Z(\alpha,\nu_\alpha)=\mathbb C$ and
$J^{\rm pol}_{\alpha,\nu_\alpha}=S_+$.
The matrix description gives the last isomorphism in
\eqref{str:eq:first-rootpower}. The closed-fiber polynomial
action identifies it with $\End_{\mathbb C}(\mathcal R_W)$.
The threshold is even by Corollary~\ref{str:cor:even-threshold},
so Theorem~\ref{str:thm:parity} gives the equality of the
two successive quotients.
\end{proof}

\begin{remark}[Power ideals and the invariant part]
\label{str:rem:power-ideal-comparison}
The inverse system of an ideal generated by powers of linear forms
and its fat-point description are classical
\cite{EmsalemIarrobino,ArdilaPostnikovPower}. The passage to the
determinant component in Theorem~\ref{str:thm:rootpower-geometry}
uses the Weyl-equivariance of this pairing and the content identity
of Theorem~\ref{str:thm:alternation}. Crowley--Larson
\cite[Theorem 1.1]{CrowleyLarsonPower} realize inverse systems
of certain arrangement power ideals as sections of line bundles
on augmented wonderful varieties. Their generators are indexed
by directions in arrangement strata, with bounds attached to
those strata. Here the ideal has a fixed exponent on one Weyl
orbit of root lines. Formula~\eqref{str:eq:shifted-invariant-hilbert}
specifies the invariant part and the exponent shift needed for
the root-power center. The algebra structure still uses the
discriminant-weighted pairing
\eqref{str:eq:rootpower-pairing}.
\end{remark}

\section{The rational Cherednik interpretation}
\label{ch:sec:interpretation}

The root-power quotients also arise from the polynomial
representation of a rational Cherednik algebra at $t=0$.
We give the map, prove a sufficient parameter condition for
surjectivity, and identify the induced central support.
All arguments in this section concern the additive
nil-Hecke algebra; the formal multiplicative comparison
is given in Section~\ref{ii:sec:comparison}.

\subsection{The polynomial action at \texorpdfstring{$t=0$}{t=0}}

Put $\mathfrak h=V^*$, so that $P=\mathbb C[\mathfrak h]$.
A reflection parameter is a family
$\mathbf k=(k_\alpha)_{\alpha\in\Phi}$ satisfying
$k_{w\alpha}=k_\alpha$; in particular $k_{-\alpha}=k_\alpha$.
Use the convention for $H_{t,\mathbf k}(W)$ with commuting
generators $x_\lambda$ ($\lambda\in V$), commuting generators
$y_\xi$ ($\xi\in\mathfrak h$), the natural conjugation
action of $W$, and the mixed relation
\begin{equation}\label{ch:eq:relations}
 [y_\xi,x_\lambda]
 =t\lambda(\xi)
 -\sum_{\alpha\in\Phi^+}
 k_\alpha\alpha(\xi)
       \langle\lambda,\alpha^\vee\rangle s_\alpha.
\end{equation}
Its polynomial action is multiplication by $x_\lambda$,
the usual $W$-action, and
\begin{equation}\label{ch:eq:dunkl}
 y_\xi\longmapsto D_\xi^{(t,\mathbf k)}
 =t\,\partial_\xi^{\mathrm{ord}}
 -\sum_{\alpha\in\Phi^+}k_\alpha\alpha(\xi)
                  \frac{1-s_\alpha}{x_\alpha}.
\end{equation}
Here $\partial_\xi^{\mathrm{ord}}(x_\lambda)=\lambda(\xi)$
is ordinary differentiation. These conventions and the
commutativity of the Dunkl operators are standard; see
\cite[Definition 2.11 and Proposition 3.2]{EtingofMa}.
We use the specialization of this polynomial representation,
without assuming that it is faithful at $t=0$.

For every root put $\partial_\alpha=(1-s_\alpha)/x_\alpha$.
The identities
$s_i=1-x_{\alpha_i}\partial_i$ and
$w\partial_\alpha w^{-1}=\partial_{w\alpha}$ show that
$W$ and all $\partial_\alpha$ belong to $NH_W$.
At $t=0$, formula~\eqref{ch:eq:dunkl} therefore defines
\begin{equation}\label{ch:eq:polynomial-map}
 \phi_{\mathbf k}:H_{0,\mathbf k}(W)\longrightarrow NH_W,
 \qquad
 \mathfrak a_{\mathbf k}:=\ker\phi_{\mathbf k}
                      =\operatorname{Ann}_{H_{0,\mathbf k}(W)}P.
\end{equation}
The image is $S$-linear. Equip both algebras with the
grading $\deg x=1$, $\deg y=-1$, $\deg w=0$; the map
preserves degree.

For $\tau\in\operatorname{Irr}(W)$ with character $\chi_\tau$
define the linear function
\begin{equation}\label{ch:eq:parameter-hyperplanes}
 a_\tau(\mathbf k)=
 \sum_{\alpha\in\Phi^+}k_\alpha
       \left(1-\frac{\chi_\tau(s_\alpha)}{\dim\tau}\right),
 \qquad
 \mathfrak K^\circ=
 \{\mathbf k:a_\tau(\mathbf k)\ne0
                          \text{ for every }\tau\ne\mathbf1\}.
\end{equation}

\begin{theorem}[The polynomial image]
\label{ch:thm:surjectivity}
The map $\phi_{\mathbf k}$ is onto $NH_W$ if and only if
the coinvariant algebra $\mathcal R_W=P/S_+P$ is irreducible
under its induced $H_{0,\mathbf k}(W)$-action.
It is onto for every $\mathbf k\in\mathfrak K^\circ$.
The set $\mathfrak K^\circ$ is a nonempty Zariski-open
subset of the parameter space and contains every family
with $k_\alpha>0$ real.
\end{theorem}

\begin{proof}
Let $\mathcal B_{\mathbf k}$ be the image in~\eqref{ch:eq:polynomial-map}.
It contains $S$ and is a graded $S$-submodule of $NH_W$.
The latter is finite free over the noetherian ring $S$, so
$NH_W/\mathcal B_{\mathbf k}$ is a finite graded module,
bounded below. The natural reduction is
\[
 NH_W/S_+NH_W\cong\operatorname{End}_{\mathbb C}(\mathcal R_W).
\]
If $\mathcal R_W$ is irreducible for $\mathcal B_{\mathbf k}$,
Burnside's theorem identifies its image on this vector space
with the full endomorphism algebra. Consequently
$NH_W=\mathcal B_{\mathbf k}+S_+NH_W$.
Graded Nakayama gives $NH_W=\mathcal B_{\mathbf k}$.
Equivalently, a nonzero quotient would have a nonzero
homogeneous element of least degree, incompatible with
being generated from itself by positive-degree invariants.
The converse follows because the natural module for its
full endomorphism algebra is irreducible.

It remains to prove irreducibility under the stated
parameter condition. Each $D_\xi=D_\xi^{(0,\mathbf k)}$
lowers polynomial degree and hence is nilpotent on
$\mathcal R_W$. They commute. Every nonzero invariant
subspace $U$ therefore contains a nonzero vector $v$
killed by all of them: take a nonzero value of a product
of the $D_\xi$ of maximal possible length on $U$.
For dual bases $x_1,\ldots,x_r$ and $\xi_1,\ldots,\xi_r$,
\begin{equation}\label{ch:eq:euler}
 \sum_{j=1}^r x_jD_{\xi_j}
 =-C_{\mathbf k},\qquad
 C_{\mathbf k}=\sum_{\alpha\in\Phi^+}k_\alpha(1-s_\alpha).
\end{equation}
Indeed, $\sum_jx_j\alpha(\xi_j)=x_\alpha$ cancels
the denominator in each summand of~\eqref{ch:eq:dunkl}.
The group-algebra element $C_{\mathbf k}$ is central,
and its scalar on $\tau$ is $a_\tau(\mathbf k)$, by
taking traces. For $\mathbf k\in\mathfrak K^\circ$ its
kernel in $\mathcal R_W$ is therefore
$\mathcal R_W^W=\mathbb C\cdot1$. Thus $v$ is a nonzero
constant. Since multiplication by $P$ generates
$\mathcal R_W$ from $1$, we have $U=\mathcal R_W$.

Finally, take all $k_\alpha>0$ and give any $W$-module
a $W$-invariant Hermitian inner product. Each reflection
is a self-adjoint involution, so
\[
 \langle C_{\mathbf k}v,v\rangle
 =\frac12\sum_{\alpha\in\Phi^+}
                  k_\alpha\|v-s_\alpha v\|^2.
\]
Its kernel is the invariant subspace, because reflections
generate $W$. Hence $a_\tau(\mathbf k)>0$ for
$\tau\ne\mathbf1$. None of these finitely many linear
functions is identically zero; their nonvanishing defines
the asserted nonempty Zariski-open set.
\end{proof}

\begin{remark}[Scope of the parameter condition]
The hyperplane condition is sufficient; we do not assert
that its complement is the exact nonsurjectivity locus.
If $W$ has one conjugacy class of reflections, every
nonzero constant parameter belongs to $\mathfrak K^\circ$,
since $a_\tau(k)=k a_\tau(1)$.
At $\mathbf k=0$ the image is $P\rtimes W$, which is
proper in $NH_W$ for $W\ne1$: the coefficient
$x_{\alpha_i}^{-1}$ in a simple divided difference
cannot occur in a polynomial skew-group element.
\end{remark}

\subsection{Root-power quotients and their central support}

\begin{corollary}[Cherednik presentation of the quotients]
\label{ch:cor:quotients}
For $\mathbf k\in\mathfrak K^\circ$ and any finite set
of imposed root powers, there are algebra isomorphisms
\begin{equation}\label{ch:eq:quotient}
 \frac{H_{0,\mathbf k}(W)}
 {\mathfrak a_{\mathbf k}+
         \langle x_\alpha^{m_\alpha}:\alpha\in\Xi\rangle}
 \cong A_{\Xi,\mathbf m}
 \cong\operatorname{Mat}_{|W|}
               (S/J^{\rm pol}_{\Xi,\mathbf m}).
\end{equation}
For a single root in irreducible type this quotient is
finite-dimensional and is nonzero exactly when
$m\geq\nu_\alpha$. Its center, Hilbert series, and
self-injectivity are therefore those calculated in this paper.
\end{corollary}

\begin{proof}
Surjectivity in Theorem~\ref{ch:thm:surjectivity}
identifies $H_{0,\mathbf k}(W)/\mathfrak a_{\mathbf k}$
with $NH_W$ and sends each imposed power to its
multiplication operator. Taking the two-sided quotient
gives the first isomorphism. Theorem~\ref{rc:thm:cyclotomic}
gives the matrix form and finiteness, and
Theorem~\ref{rc:thm:threshold} gives the exponent criterion.
\end{proof}

Write $\mathfrak Z_{\mathbf k}=Z(H_{0,\mathbf k}(W))$
and $X_{\mathbf k}=\operatorname{Spec}\mathfrak Z_{\mathbf k}$.
This is the generalized Calogero--Moser space. The standard
$t=0$ center theory and central supports of Verma modules
are discussed in~\cite[\S8.6]{EtingofMa}
and~\cite{BellamyVerma}. The following description uses
only the central action on the polynomial representation.

\begin{proposition}[A closed section and root-power subschemes]
\label{ch:prop:central-section}
For every reflection parameter $\mathbf k$, there is a
split surjective algebra map
\[
 \chi_{\mathbf k}:\mathfrak Z_{\mathbf k}\longrightarrow S,
 \qquad z\longmapsto z\cdot1,
 \qquad \chi_{\mathbf k}|_S=\operatorname{id}_S.
\]
Its kernel is $\mathfrak a_{\mathbf k}\cap\mathfrak Z_{\mathbf k}$.
It defines a closed section
$\Lambda_{\mathbf k}=\operatorname{Spec}S\hookrightarrow
X_{\mathbf k}$ of the projection to $\mathfrak h/W$.
The scheme $\operatorname{Spec}Z(\alpha,m)$ is the closed
subscheme of $\Lambda_{\mathbf k}$ defined by
$J^{\rm pol}_{\alpha,m}$, with ambient ideal
\[
 \chi_{\mathbf k}^{-1}(J^{\rm pol}_{\alpha,m})
       \subseteq\mathfrak Z_{\mathbf k}.
\]
When it is nonzero and finite-dimensional, it is supported
at the point of this section above the origin of
$\mathfrak h/W$.
\end{proposition}

\begin{proof}
The defining relations, by induction on polynomial degree,
give at $t=0$
\[
 [y_\xi,f]=
 -\sum_{\alpha\in\Phi^+}k_\alpha\alpha(\xi)
       \frac{f-s_\alpha(f)}{x_\alpha}s_\alpha
       \qquad(f\in P).
\]
Hence $S\subseteq\mathfrak Z_{\mathbf k}$.
If $z\in\mathfrak Z_{\mathbf k}$, its action on $P$
commutes with every multiplication operator. It is therefore
multiplication by $z\cdot1$. Commutation with $W$ implies
$z\cdot1\in S$. This proves that $\chi_{\mathbf k}$
is an algebra map, and multiplication by an element of
$S$ gives its splitting. Such a central element kills
$P$ exactly when $z\cdot1=0$, proving the kernel formula.
The statements about spectra follow by taking quotients.
The support assertion follows from the graded local
description in Theorem~\ref{rc:thm:cyclotomic}.
\end{proof}

The polynomial module for $A_{\alpha,m}$ is
$P/J^{\rm pol}_{\alpha,m}P$, free of rank $|W|$ over
its center. Via~\eqref{ch:eq:polynomial-map}, its central
annihilator is exactly
$\chi_{\mathbf k}^{-1}(J^{\rm pol}_{\alpha,m})$.
For $\mathbf k\in\mathfrak K^\circ$ it is the faithful
column module of the matrix quotient in~\eqref{ch:eq:quotient}.
Thus the central subscheme description records the full
content ideal, including its nonreduced structure.
It does not require a smoothness or Lagrangian assertion
about the ambient Calogero--Moser space.

\begin{example}[Rank one]
For $W=\{1,s\}$ acting by $s(x)=-x$, relation
\eqref{ch:eq:relations} is $[y,x]=-2ks$ at $t=0$, and
\[
 \phi_k(y)=-k\frac{1-s}{x}=-k\partial.
\]
Thus surjectivity is immediate for $k\ne0$.
The quotient by the polynomial annihilator and $x^m$
is the matrix algebra in Example~\ref{ex:a1-centre}.
For instance, at $m=4,5$ its center is
$\mathbb C[c]/(c^2)$, $c=x^2$, giving a nonreduced
length-two subscheme on the section $\Lambda_k$.
\end{example}

\subsection{Why the finite-dimensional interpretation uses
\texorpdfstring{$t=0$}{t=0}}

\begin{proposition}[An obstruction at nonzero $t$]
\label{ch:prop:trace-obstruction}
Assume $\operatorname{rank}\Phi>0$.
No nonzero finite-dimensional $NH_W$-module admits an
$H_{t,\mathbf k}(W)$-action for $t\ne0$ extending its
given polynomial and $W$-actions.
\end{proposition}

\begin{proof}
The harmonic decomposition gives a $W$-equivariant
$S$-module isomorphism $P\cong S\otimes\mathcal R_W$,
where $\mathcal R_W$ is the regular $W$-module
\cite{Steinberg}. By the nil-Hecke Morita equivalence,
every finite-dimensional module is
$P\otimes_S L$ for some finite-dimensional $S$-module $L$.
Its restriction to $W$ is therefore a direct sum of
regular representations. In particular, every reflection
has trace zero. If the asserted action existed, taking
traces in~\eqref{ch:eq:relations} would give
\[
 0=t\lambda(\xi)\dim_{\mathbb C}M.
\]
Choose $\lambda(\xi)=1$. This contradicts $t\ne0$
and $M\ne0$ over $\mathbb C$.
\end{proof}

\section{Reflection subgroups and finite differential models}
\label{str:sec:structural-transfer}

The content identity applies to any homogeneous invariant ideal.
This gives a comparison between reflection groups before any
type-specific calculation. We then describe the directional
filtration and a finite complex that computes the relations and
socles of the centers. These constructions keep track of the
module structure as well as the Hilbert function.

\subsection{Content for invariant ideals}

Let $W'\subseteq W$ be a subgroup generated by reflections, acting
on the same space $V$. Put $S_W=P^W$ and $S_{W'}=P^{W'}$.
Choose products of defining linear forms for their reflecting
hyperplanes so that
\[
 \Delta_W=\Delta_{W'}\delta.
\]
The factors are determined up to nonzero scalars. Their degrees
are denoted by $N_W,N_{W'}$, and $e=N_W-N_{W'}$, respectively.
The quotient $\delta$ is $W'$-invariant, since both discriminants
transform by the determinant character on $W'$.

For a homogeneous $W$-stable ideal $I\subseteq P$, define
\[
 J_W(I)=\bigl(\varphi(f):f\in I,\ 
                         \varphi\in\Hom_{S_W}(P,S_W)\bigr),
 \qquad Z_W(I)=S_W/J_W(I),
\]
and define $J_{W'}(I)$ and $Z_{W'}(I)$ in the same way. For
$I=I_{\alpha,m}$, the ideal $J_W(I)$ equals
$J^{\rm pol}_{\alpha,m}$. Indeed, content is constant on the
$W$-orbit of $x_\alpha^m$, and multiplying an element of $P$
cannot enlarge its content ideal.

\begin{theorem}[Reflection-subgroup transfer]
\label{str:thm:subgroup-transfer}
With the preceding notation,
\begin{equation}\label{str:eq:general-content}
 \Delta_W J_W(I)=I\cap P^{\det_W},\qquad
 \Delta_{W'}J_{W'}(I)=I\cap P^{\det_{W'}}.
\end{equation}
Consequently,
\begin{equation}\label{str:eq:subgroup-colon}
 J_W(I)=\bigl(J_{W'}(I):_{S_{W'}}\delta\bigr)\cap S_W.
\end{equation}
Multiplication by $\delta$ gives an injective $S_W$-linear map
\[
 Z_W(I)\longrightarrow Z_{W'}(I)
\]
raising degrees by $e$.

Assume in addition that $W'$ is normal in $W$, and put
$\Gamma=W/W'$. Then $\delta$ is a relative invariant for a
character $\chi:\Gamma\to\{1,-1\}$. If a superscript $\chi$
denotes the subspace on which $\gamma$ acts as multiplication
by $\chi(\gamma)$, then
\begin{equation}\label{str:eq:subgroup-character}
 S_{W'}^\chi=\delta S_W,\qquad
 Z_W(I)_k\xrightarrow{\ \sim\ }
                   Z_{W'}(I)_{k+e}^\chi.
\end{equation}
The second map is a graded $S_W$-module identification with the
indicated degree shift. It does not identify multiplication with
the ordinary product of elements in the character component.
\end{theorem}

\begin{proof}
The divided-alternation pairing for $P/S_W$ is perfect, as in
the proof of Theorem~\ref{str:thm:alternation}. The same proof
applies to $W'$, including a reducible reflection subgroup or a
subgroup with a fixed subspace. Every functional on $P$ is
therefore multiplication by a polynomial followed by divided
alternation. Since $I$ is an ideal, its content is the image of
its divided alternants. Taking the determinant component is
exact in characteristic zero. This proves
\eqref{str:eq:general-content}.

For $f\in S_W$, these identities give
\[
 \begin{aligned}
 f\in J_W(I)
 &\Longleftrightarrow \Delta_W f\in I\\
 &\Longleftrightarrow \Delta_{W'}\delta f\in I\\
 &\Longleftrightarrow \delta f\in J_{W'}(I).
 \end{aligned}
\]
This proves the colon formula and the injectivity assertion.

Now suppose that $W'$ is normal. The set of reflecting
hyperplanes of $W'$ is $W$-stable, so $\Delta_{W'}$ is a
relative invariant under $W$. Hence so is $\delta$. Its
character is trivial on $W'$, and therefore factors through
$\Gamma$; call the resulting character $\chi$.
For a reflection $s$ in $W'$, one has $s\delta=\delta$.
If $s$ is a reflection in $W\setminus W'$, its hyperplane is
not a reflecting hyperplane of $W'$. Evaluating the relative
invariant $\Delta_{W'}$ at a generic point on that hyperplane
shows that $s\Delta_{W'}=\Delta_{W'}$. Thus $s\delta=-\delta$.
Since reflections generate $W$, the character takes values
in $\{1,-1\}$.

If $h\in S_{W'}^\chi$, then $h$ changes sign under every
reflection in $W\setminus W'$. It vanishes on each corresponding
hyperplane and is divisible by their product $\delta$.
The quotient $h/\delta$ is $W$-invariant. This proves
$S_{W'}^\chi=\delta S_W$.
The ideal $J_{W'}(I)$ is $\Gamma$-stable: this follows either
from its content definition or from
\eqref{str:eq:general-content}. Exactness of the character
projection now gives
\[
 Z_{W'}(I)^\chi
 =\delta S_W/\bigl(J_{W'}(I)\cap\delta S_W\bigr).
\]
The colon formula identifies the denominator with
$\delta J_W(I)$, proving~\eqref{str:eq:subgroup-character}.
\end{proof}

When $W'\alpha=W\alpha$, both groups have the same root-power
ideal, so the theorem compares their individual root-power
centers. This hypothesis is essential if one writes
$J^{W'}_{\alpha,m}$ for the content of a single root power.
For a general subgroup the theorem still applies to the common
$W$-stable ideal, which can contain several $W'$-orbits.

\begin{corollary}[Transfer of equivariant resolutions]
\label{str:cor:resolution-transfer}
Assume $W'$ is normal, and let
\[
 \cdots\longrightarrow F_1\longrightarrow F_0
       \longrightarrow Z_{W'}(I)\longrightarrow0,
 \qquad F_i=S_{W'}\otimes_{\mathbb C}E_i,
\]
be a graded $\Gamma$-equivariant free resolution, with finite
graded generator spaces $E_i$. Taking the $\chi$ component
and lowering every degree by $e$ gives an $S_W$-free resolution
of $Z_W(I)$.

More explicitly, choose a graded $W$-stable space $\mathcal H_W$
of coinvariant lifts in $P$. The free generator space in its
$i$th term, before the degree shift, is
\[
 \bigl(\mathcal H_W^{W'}\otimes E_i\bigr)^\chi.
\]
The transferred resolution need not be minimal; scalar entries
must be cancelled before reading off the minimal Betti numbers.
\end{corollary}

\begin{proof}
A $W$-equivariant graded section of the coinvariant quotient
exists by averaging. The usual homogeneous-basis argument gives
$P\cong S_W\otimes\mathcal H_W$. Taking $W'$-invariants gives
\[
 S_{W'}\cong S_W\otimes\mathcal H_W^{W'}
\]
as graded $\Gamma$-modules, since $\Gamma$ acts trivially on
$S_W$. It follows that every $F_i^\chi$ is free over $S_W$
with the stated generator space. Character projection is exact,
and Theorem~\ref{str:thm:subgroup-transfer} identifies the last
term after shifting. Equivariant resolutions can be obtained
by choosing equivariant generator lifts successively, using
averaging at each step. Exactness does not exclude degree-zero
entries after restriction of scalars, so minimality requires
the stated cancellations.
\end{proof}

\begin{corollary}[The $D_n$ to $B_n$ comparison]
\label{str:cor:dn-bn-transfer}
Let $n\geq3$, with $D_3$ understood as $A_3$. Use
$P=\mathbb C[\xi_1,\ldots,\xi_n]$ and
$\alpha=\epsilon_1-\epsilon_2$. Put
\[
 c_j=e_j(\xi_1^2,\ldots,\xi_n^2),\qquad
 \eta=\xi_1\cdots\xi_n.
\]
For the two Weyl groups,
\[
 S_D=\mathbb C[c_1,\ldots,c_{n-1},\eta],\qquad
 S_B=\mathbb C[c_1,\ldots,c_{n-1},\eta^2].
\]
If $J^D_{\alpha,m}$ and $J^B_{\alpha,m}$ denote their
polynomial content ideals, then
\begin{equation}\label{str:eq:dn-bn-colon}
 J^B_{\alpha,m}=(J^D_{\alpha,m}:\eta)\cap S_B,
 \qquad
 Z_B(\alpha,m)_k\cong Z_D(\alpha,m)_{k+n}^{-},
\end{equation}
where the superscript $-$ means the negative eigenspace for
$\eta\mapsto-\eta$. For odd $n$ this gives
\[
 H_{Z_B(\alpha,m)}(q)
 =q^{-n}\frac{H_{Z_D(\alpha,m)}(q)
                  -H_{Z_D(\alpha,m)}(-q)}2.
\]
For even $n$, the eigenspace is still defined by the sign of
$\eta$ and must not be replaced by total-degree parity.
The same comparison applies to the short-root orbit in $C_n$.
\end{corollary}

\begin{proof}
The even signed-permutation group $W(D_n)$ is normal of index
two in $W(B_n)$. Both groups have the orbit
$\{\pm\epsilon_i\pm\epsilon_j:i<j\}$ on $\alpha$.
Their discriminants have ratio $\eta$, and the quotient group
fixes each $c_j$ and changes the sign of $\eta$.
Apply Theorem~\ref{str:thm:subgroup-transfer}. When $n$ is odd,
this action agrees with total-degree parity because the $c_j$
have even degrees and $\eta$ has degree $n$.
The Weyl action and chosen root direction in $C_n$ are the same.
\end{proof}

\subsection{Several root orbits}

The intrinsic quotients of Section~\ref{rc:sec:cyclotomic}
also fit the discriminant description. In an irreducible root
system it is enough to impose one exponent on each selected
Weyl orbit of roots: within a fixed orbit, only the smallest
imposed exponent matters.

\begin{theorem}[Simultaneous directional bounds]
\label{str:thm:multi-orbit}
Let $\alpha_1,\ldots,\alpha_s$ represent distinct Weyl orbits
of roots in an irreducible root system, and let $m_j\geq1$.
Put
\[
 I_{\mathbf m}=\sum_{j=1}^s I_{\alpha_j,m_j},\qquad
 J_{\mathbf m}=\sum_{j=1}^s J^{\rm pol}_{\alpha_j,m_j},
 \qquad Z_{\mathbf m}=S/J_{\mathbf m}.
\]
The corresponding polynomial root-power quotient is
$\Mat_{|W|}(Z_{\mathbf m})$. It is nonzero exactly when
$m_j\geq\nu_{\alpha_j}$ for every $j$. In this range put
$b_j=m_j-\nu_{\alpha_j}$ and
\[
 \mathcal E_{\mathbf b}
 =\{H\in T^W:\deg_{w\alpha_j}H\leq b_j
                   \text{ for every }w\in W,\ 1\leq j\leq s\}.
\]
There are perfect pairings
\begin{equation}\label{str:eq:multi-orbit-duality}
 Z_{\mathbf m,k}\times\mathcal E_{\mathbf b,k}
 \longrightarrow\mathbb C,\qquad
 (\bar f,H)\longmapsto
       (\Delta_\Phi f)(\mathsf d)(\Delta_\Phi^\vee H)(0).
\end{equation}
Moreover,
\[
 J_{\mathbf m}=\Ann_S(\mathcal E_{\mathbf b})
\]
for the conjugated action $\mathscr D$ of
Proposition~\ref{str:prop:conjugated-action}. If some $b_j=0$,
then $Z_{\mathbf m}=\mathbb C$.
\end{theorem}

\begin{proof}
The matrix description follows from
Theorem~\ref{rc:thm:cyclotomic}.
The exponent criterion is the criterion in that theorem.
Alternatively, below a threshold one summand of $J_{\mathbf m}$
is the unit ideal. If all thresholds are met, all summands
are homogeneous proper ideals of $S$ and their sum lies in $S_+$.

The inverse system of a sum of ideals is the intersection of
their inverse systems. Taking the determinant component and
using Theorem~\ref{str:thm:discriminant-removal} identifies it
with $\Delta_\Phi^\vee\mathcal E_{\mathbf b}$. The content
identity and apolarity give~\eqref{str:eq:multi-orbit-duality}.
Transporting the differential action proves the annihilator
statement. If $b_j=0$, all derivatives of $H$ along
$W\alpha_j$ vanish. This orbit spans $V$, so $H$ is constant.
Finally, content ideals are constant on a root orbit and
$J^{\rm pol}_{\alpha,m+1}\subseteq J^{\rm pol}_{\alpha,m}$.
Thus repeated conditions on the same orbit reduce to the
smallest exponent, as asserted before the theorem.
\end{proof}

\subsection{The directional filtration}

Assume that $\Phi$ is irreducible. For a fixed root $\alpha$, set
\[
 \mathcal F_r T^W=\{H\in T^W:\deg_\alpha H\leq2r\},
 \qquad r\geq0,\qquad \mathcal F_{-1}T^W=0.
\]
This is an exhaustive multiplicative filtration. Its degree-zero
part is $\mathbb C$: invariance propagates a zero directional
degree to the spanning orbit $W\alpha$. The filtration index
is separate from the ordinary polynomial grading.

\begin{proposition}[Leading coefficients and the root stabilizer]
\label{str:prop:leading-stabilizer}
Let $W_\alpha=\{w\in W:w\alpha=\alpha\}$. Write
$v=u+t\alpha$ with $u\in\alpha^\perp$. If
$H\in\mathcal F_rT^W$, let $a_r(H)$ be the coefficient of
$t^{2r}$ in $H(u+t\alpha)$. Then
\begin{equation}\label{str:eq:leading-embedding}
 \operatorname{gr}_{\mathcal F}T^W
 \lhook\joinrel\longrightarrow
 \mathbb C[\alpha^\perp]^{W_\alpha}[\zeta],\qquad
 [H]_r\longmapsto a_r(H)\zeta^r
\end{equation}
is an injective bigraded algebra map, where $\zeta$ has
polynomial degree two and filtration degree one.
\end{proposition}

\begin{proof}
The reflection $s_\alpha$ fixes $u$ and replaces $t$ by $-t$.
Thus every invariant has only even powers of $t$ in these
coordinates. Its $t^{2r}$ coefficient is zero exactly when
it belongs to $\mathcal F_{r-1}T^W$. This proves that the map
is well defined and injective. Every element of $W_\alpha$
fixes the direction $\alpha$ and preserves $\alpha^\perp$,
so the coefficient lies in the stated invariant ring.
Leading coefficients multiply, proving the algebra assertion.
For a homogeneous $H$ of degree $d$, its coefficient has
degree $d-2r$, giving the asserted bigrading.
\end{proof}

\begin{proposition}[Filtered polynomial structure]
\label{str:prop:filtered-invariants}
Under the hypotheses of
Proposition~\ref{str:prop:weighted-invariants},
\[
 \mathcal F_r T^W
 =\operatorname{span}_{\mathbb C}
       \{f_1^{a_1}\cdots f_l^{a_l}:\textstyle\sum_j\omega_ja_j\leq r\}.
\]
Its Rees algebra and associated graded algebra have presentations
\begin{equation}\label{str:eq:rees-invariants}
 \bigoplus_{r\geq0}(\mathcal F_rT^W)t^r
   =\mathbb C[t,f_1t^{\omega_1},\ldots,f_lt^{\omega_l}],
 \qquad
 \operatorname{gr}_{\mathcal F}T^W
   \cong\mathbb C[\bar f_1,\ldots,\bar f_l].
\end{equation}
All displayed generators are algebraically independent. The
bidegrees are $(0,1)$ for $t$ and $(d_j,\omega_j)$ for the
other generators. The Rees algebra is flat over $\mathbb C[t]$,
with fibers $T^W$ at $t=1$ and
$\operatorname{gr}_{\mathcal F}T^W$ at $t=0$.
If the nonconstant leading coefficients $a_2,\ldots,a_l$
generate $\mathbb C[\alpha^\perp]^{W_\alpha}$, then the
embedding~\eqref{str:eq:leading-embedding} becomes an equality
after inverting $\zeta$.
\end{proposition}

\begin{proof}
The equality for $\mathcal F_r$ is the noncancellation
conclusion of Proposition~\ref{str:prop:weighted-invariants}.
Each filtered monomial multiplied by $t^r$ can be written
uniquely as
\[
 t^{r-\sum_j\omega_ja_j}
       \prod_j(f_jt^{\omega_j})^{a_j}.
\]
This proves generation. Independence follows from the
independence of $f_1,\ldots,f_l$ in $T^W[t]$, by comparing
the powers of $t$ and the monomials in the $f_j$.
The Rees algebra is therefore a polynomial algebra over
$\mathbb C[t]$. Its two specializations give the assertions.
For the final statement, the image contains $a_1\zeta$,
with $a_1\in\mathbb C^\times$, and all
$a_j\zeta^{\omega_j}$. After inverting $\zeta$ it therefore
contains each $a_j$, and hence the entire stated target.
\end{proof}

This proposition describes the invariant filtration whose finite
pieces are dual to the centers. It does not assert a flat
degeneration of the center algebras to monomial quotients:
their multiplication is encoded by the conjugated differential
action, not by multiplication inside $\mathcal F_rT^W$.

\subsection{Finite complexes for relations and socles}

Fix a nonzero root-power center $Z=S/J$, with
$S=\mathbb C[c_1,\ldots,c_l]$ and $\deg c_i=d_i$.
Let $E=\mathcal E_{\alpha,m-\nu_\alpha}$ and put
$D_i=\mathscr D_{c_i}|_E$. Thus
$D_i:E_k\to E_{k-d_i}$, and the $D_i$ commute.
The same construction works for $E=\mathcal E_{\mathbf b}$
and the simultaneous quotient in
Theorem~\ref{str:thm:multi-orbit}.

\begin{proposition}[A differential complex for the graded Betti numbers]
\label{str:prop:betti-complex}
For each integer $j$, form the cochain complex
\[
 C^i(E)_j=\bigoplus_{\substack{A\subseteq\{1,\ldots,l\}\\|A|=i}}
            E_{j-d_A},\qquad d_A=\sum_{a\in A}d_a,
\]
For $b\notin A$, the differential from the $A$ component to
the $A\cup\{b\}$ component is
\[
 (-1)^{\#\{a\in A:a<b\}}D_b.
\]
Then
\begin{equation}\label{str:eq:betti-complex}
 H^i(C^\bullet(E)_j)
 \cong\operatorname{Tor}_i^S(Z,\mathbb C)_j^*.
\end{equation}
In particular, $H^1$ determines the degrees and numbers of
minimal generators of $J$. If $d_\Sigma=\sum_i d_i$, then
\begin{equation}\label{str:eq:socle-complex}
 (\operatorname{Soc}Z)_k^*
 \cong H^l(C^\bullet(E)_{k+d_\Sigma})
 \cong\frac{E_k}{\sum_iD_i(E_{k+d_i})}.
\end{equation}
Thus the minimal Betti numbers and the graded socle can be
computed by ranks of finite matrices on the explicit invariant
inverse system.
\end{proposition}

\begin{proof}
The pairing in Theorem~\ref{str:thm:rootpower-geometry}, after
discriminant removal, identifies $E_k$ with $Z_k^*$.
For $f\in Z$ and $H\in E$, it satisfies
\[
 \langle c_i f,H\rangle=\langle f,D_iH\rangle.
\]
The displayed cochain complex is therefore the degreewise
vector-space dual of the Koszul complex on
$c_1,\ldots,c_l$ with coefficients in $Z$. Commutativity
of the $D_i$ makes its differential square to zero.
The Koszul complex resolves the residue field over $S$,
giving~\eqref{str:eq:betti-complex}. Since $J\subseteq S_+$,
$\operatorname{Tor}_1^S(Z,\mathbb C)=J/S_+J$.
The last Koszul homology is the socle, with degrees raised
by $d_\Sigma$. Dualizing its last differential proves
\eqref{str:eq:socle-complex}.
\end{proof}

The homological construction is the usual dual Koszul complex.
Its use here is that the spaces $E_k$ have the explicit bases
provided by the directional calculation, while each $D_i$ is
the concrete operator
$(\Delta_\Phi^\vee)^{-1}c_i(\mathsf d)\Delta_\Phi^\vee$.
This provides one finite procedure for computing the algebraic
invariants left open by a Hilbert-series calculation.

\section{Centers in the classical root systems}
\label{str:sec:classical-centers}

We determine the Hilbert series for every classical root orbit
and every exponent. Type $A_n$ also admits an exact classification
of the self-injective quotients, and the coordinate-root families
in $B_n,C_n$ have Grassmannian centers. For the remaining
classical roots, adapted invariant generators turn the directional
filtration into a weighted monomial filtration. We then use the
conjugated differential action to determine the first nonsemisimple
center algebras, including the exceptional socle dimension in $D_4$.
We continue over $\mathbb C$ and put $a=\lfloor m/2\rfloor$.
For a finite-dimensional graded algebra $Z$ we write
$H_Z(q)=\sum_k(\dim_{\mathbb C}Z_k)q^k$.

\subsection{Type \texorpdfstring{$A_n$}{An}: Hilbert series and self-injectivity}

\begin{theorem}[Type $A$ in every rank and exponent]
\label{str:thm:an-centers}
Let $\Phi=A_n$, $n\geq1$, and let $\alpha$ be any root.
The quotient $A_{\alpha,m}$ is nonzero exactly when $a\geq n$.
In that range its center has Hilbert series and dimension
\begin{equation}\label{str:eq:an-center-hilbert}
 \begin{split}
 H_{Z(\alpha,m)}(q)
 &=\sum_{b_2+\cdots+b_{n+1}\leq a-n}
                  q^{\,2b_2+3b_3+\cdots+(n+1)b_{n+1}},\\
 \dim Z(\alpha,m)&=\binom an,
 \qquad
 \dim A_{\alpha,m}=((n+1)!)^2\binom an.
 \end{split}
\end{equation}
The sum is over nonnegative integers. If $n\geq2$, the
nonzero center is Gorenstein exactly when $a=n$.
Equivalently, $A_{\alpha,m}$ is self-injective, Frobenius, or
symmetric exactly when $m=2n$ or $2n+1$.
For $n=1$ all nonzero centers are
$\mathbb C[c_2]/(c_2^a)$, with $\deg c_2=2$.
\end{theorem}

\begin{proof}
Use dual coordinates to realize
\[
 T=\mathbb C[z_1,\ldots,z_{n+1}]/(z_1+\cdots+z_{n+1}),
 \qquad W=\mathfrak S_{n+1},
\]
with $W$ permuting the coordinates. If $e_j$ is the elementary
symmetric polynomial of degree $j$, then
$T^W=\mathbb C[e_2,\ldots,e_{n+1}]$.
Choose the root direction $\alpha=\epsilon_1-\epsilon_2$.
Translation in this direction replaces $z_1,z_2$ by
$z_1+t,z_2-t$ and fixes the other coordinates.
Let $e_j'$ be elementary symmetric in $z_3,\ldots,z_{n+1}$,
with $e_0'=1$. The identity
\[
 e_j=e_j'+(z_1+z_2)e_{j-1}'+z_1z_2e_{j-2}'
\]
shows that the coefficient of $t^2$ in the translated $e_j$
is $-e_{j-2}'$ for $2\leq j\leq n+1$.

Write a nonzero invariant as $H=f(e_2,\ldots,e_{n+1})$, and
let $L$ be the ordinary total degree of $f$ in its $n$
arguments, each argument being assigned degree one for this
purpose. If $f_L$ is its highest part for that degree, the
coefficient of $t^{2L}$ in $H(v+t\alpha)$ is
\[
 f_L(-1,-e_1',\ldots,-e_{n-1}').
\]
This coefficient is nonzero. Indeed, $e_1',\ldots,e_{n-1}'$
are algebraically independent, and a nonzero homogeneous
polynomial cannot vanish identically after its first argument
is set to $-1$. To see the latter assertion, scale any point
with nonzero first coordinate to that hyperplane; vanishing
there would imply vanishing on a dense open set. Hence
\begin{equation}\label{str:eq:an-directional-degree}
 \deg_\alpha H=2\deg_{\rm tot} f.
\end{equation}
Weyl invariance gives the same bound in every root direction.

Here $\nu_\alpha=2n$ by
Corollary~\ref{geo:cor:minimal-orbit}, since
$h^\vee(A_n)=n+1$. Theorem~\ref{str:thm:discriminant-removal}
and~\eqref{str:eq:an-directional-degree} show that the
invariant polynomials indexing the dual center are precisely
\[
 e_2^{b_2}\cdots e_{n+1}^{b_{n+1}},
 \qquad b_2+\cdots+b_{n+1}\leq a-n.
\]
This proves the Hilbert series. The number of these monomials
is $\binom an$, and $|W|=(n+1)!$ gives the algebra dimension.
The nonvanishing assertion also follows from the threshold.

Suppose $n\geq2$ and $a>n$, and put $r=a-n$.
The largest degree in~\eqref{str:eq:an-center-hilbert} is
$D=(n+1)r$, with coefficient one. The coefficient of $q$ is
zero, but that of $q^{D-1}$ is nonzero: the monomial
$e_n e_{n+1}^{r-1}$ contributes to it.
An Artinian graded local Gorenstein algebra has a perfect
multiplication pairing between degrees $k$ and $D-k$.
For completeness, a nonzero homogeneous element generates an
ideal meeting the one-dimensional socle in degree $D$;
therefore its product with some element of complementary
degree is nonzero. This proves perfection and symmetry of
the Hilbert series. The two coefficients just found violate
that symmetry, so the center is not Gorenstein.
At $a=n$ its Hilbert series is $1$, so it is $\mathbb C$.
The equivalences in Theorem~\ref{str:thm:rootpower-geometry}
\textup{(iv)} prove the asserted classification.
For $n=1$, $P=\mathbb C[x]$ is free over
$S=\mathbb C[x^2]$ with basis $1,x$; the content of
$x^{2a}$ is $((x^2)^a)$. Parity gives the final assertion.
\end{proof}

\subsection{Coordinate roots in types \texorpdfstring{$B_n$ and $C_n$}{Bn and Cn}}

\begin{theorem}[Grassmannian centers in every rank]
\label{str:thm:bc-grassmannian}
Let $n\geq2$. Suppose that $\Phi=B_n$ and $\alpha$ is short,
or that $\Phi=C_n$ and $\alpha$ is long. Then $A_{\alpha,m}$
is nonzero exactly when $a\geq n$. Choose invariant coordinates
$c_1,\ldots,c_n$, of degrees $2,4,\ldots,2n$, and define
$h_b$ by
\begin{equation}\label{str:eq:bc-complete-symmetric}
 \sum_{b\geq0}h_b u^b
 =(1-c_1u+c_2u^2-\cdots+(-1)^n c_nu^n)^{-1}.
\end{equation}
For $a\geq n$ the polynomial content ideal and center are
\begin{equation}\label{str:eq:bc-grassmannian-center}
 \begin{split}
 J^{\rm pol}_{\alpha,m}&=(h_{a-n+1},\ldots,h_a),\\
 Z(\alpha,m)&\cong
 \mathbb C[c_1,\ldots,c_n]/(h_{a-n+1},\ldots,h_a)
 \cong H^*(\operatorname{Gr}(n,a);\mathbb C).
 \end{split}
\end{equation}
The last isomorphism preserves the displayed degrees and the
ordinary cohomological grading. In particular,
\begin{equation}\label{str:eq:bc-center-hilbert}
 \begin{split}
 H_{Z(\alpha,m)}(q)
 &=\prod_{j=1}^n\frac{1-q^{2(a-n+j)}}{1-q^{2j}},\\
 \dim Z(\alpha,m)&=\binom an,
 \qquad
 \dim A_{\alpha,m}=(2^n n!)^2\binom an.
 \end{split}
\end{equation}
Every nonzero center is a complete intersection, with
one-dimensional socle in degree $2n(a-n)$.
Every corresponding nil-Hecke quotient is symmetric.
\end{theorem}

\begin{proof}
In both cases $W$ acts by signed permutations on
$P=\mathbb C[\xi_1,\ldots,\xi_n]$. The chosen root coordinate
is a nonzero scalar multiple of $\xi_1$: it is $\xi_1$ in
the usual $B_n$ normalization and $2\xi_1$ in type $C_n$.
Scaling does not change content. Put
\[
 z_i=\xi_i^2,\qquad B_0=\mathbb C[z_1,\ldots,z_n],
 \qquad c_j=e_j(z_1,\ldots,z_n).
\]
Thus $S=P^W=\mathbb C[c_1,\ldots,c_n]$.
Let $c_j'=e_j(z_2,\ldots,z_n)$ and $t=z_1$.
The ring
\[
 R_0=B_0^{\mathfrak S_{\{2,\ldots,n\}}}
     =\mathbb C[t,c_1',\ldots,c_{n-1}']
\]
equals $S[t]$, because $c_j=c_j'+tc_{j-1}'$ recursively
expresses every $c_j'$ in that ring. More precisely,
\begin{equation}\label{str:eq:bc-characteristic-polynomial}
 R_0\cong
 S[t]/(t^n-c_1t^{n-1}+c_2t^{n-2}-\cdots+(-1)^n c_n).
\end{equation}
The inverse substitution is
\[
 c_j'=\sum_{k=0}^j(-t)^k c_{j-k},\qquad c_0=1,
\]
which verifies the presentation directly.
Thus $R_0$ is free over $S$ with basis $1,t,\ldots,t^{n-1}$.

The polynomial ring in $z_2,\ldots,z_n$ is free over its
symmetric subring. Choose a homogeneous coinvariant basis
including $1$. Extending that basis by $t$ makes $B_0$ free
over $R_0$ with a basis including $1$.
Also $P$ is free over $B_0$ with basis
$(\xi_1^{\epsilon_1}\cdots\xi_n^{\epsilon_n})_{
\epsilon_i\in\{0,1\}}$, again including $1$.
Combining these bases shows that the content of
$\xi_1^{2a}=t^a$ in $P/S$ is its content in $R_0/S$.
This use of a basis containing $1$ ensures that there are
no additional coordinates outside this direct summand.

Let $q_j(a)$ be the coefficient of $t^j$ in the remainder
of $t^a$ modulo the polynomial in
\eqref{str:eq:bc-characteristic-polynomial}.
For $0\leq j\leq n-1$, expansion of
$(1-tu)^{-1}$ in that quotient gives
\begin{equation}\label{str:eq:bc-remainder-coefficients}
 q_{n-1-j}(a)
 =\sum_{k=0}^j(-1)^k c_k h_{a-n+1+j-k},
\end{equation}
where $h_b=0$ for $b<0$. To check the expansion directly,
multiply
\[
 \sum_{j=0}^{n-1}u^j
       \sum_{k=0}^j(-1)^k c_k t^{j-k}
\]
by $1-tu$; the result is the denominator on the right of
\eqref{str:eq:bc-complete-symmetric}, using the relation
for $t^n$. Formula~\eqref{str:eq:bc-remainder-coefficients}
is triangular in $h_{a-n+1},\ldots,h_a$, with diagonal
coefficient one. Consequently its coefficients generate
exactly $(h_{a-n+1},\ldots,h_a)$.
Lemma~\ref{str:lem:content} and
Theorem~\ref{str:thm:parity} give the first two statements
in~\eqref{str:eq:bc-grassmannian-center}.
If $a<n$, the remainder $t^a$ has a coordinate equal to
$1$, so the content ideal is the unit ideal.

The usual Chern-class presentation of the Grassmannian is
the last quotient in~\eqref{str:eq:bc-grassmannian-center};
see~\cite[Section 3.6]{Lauda}. Here $c_j$ has polynomial
degree $2j$, which equals the cohomological degree of the
$j$th Chern class. This grading is separate from the
degree-doubling characteristic map for the original flag variety.
For $a\geq n$ the quotient is nonzero and finite-dimensional,
by its root-power interpretation and the threshold
$\nu_\alpha=2n$. Its $n$ homogeneous defining generators
therefore form a regular sequence in the $n$-variable
polynomial ring. Their degrees are
$2(a-n+1),\ldots,2a$, giving
\eqref{str:eq:bc-center-hilbert} and socle degree $2n(a-n)$.
The complete intersection is Gorenstein. Equivalently, integration
on the Grassmannian supplies a Frobenius functional; matrix trace
transports its symmetric form to the nil-Hecke quotient.
The Weyl-group order is $2^n n!$, completing the proof.
\end{proof}

\begin{corollary}[Directional form of the coordinate-root Hilbert series]
\label{str:cor:coordinate-hilbert}
Under the hypotheses of Theorem~\ref{str:thm:bc-grassmannian},
put $r=a-n\geq0$. Then
\[
 \begin{aligned}
 H_{Z(\alpha,m)}(q)
 &=\sum_{b_1+\cdots+b_n\leq r}q^{2b_1+4b_2+\cdots+2nb_n},\\
 \sum_{r\geq0}H_{Z(\alpha,2n+2r)}(q)t^r
 &=\frac{1}{(1-t)\prod_{j=1}^n(1-tq^{2j})}.
 \end{aligned}
\]
\end{corollary}

\begin{proof}
In standard dual coordinates use
$p_j=e_j(z_1^2,\ldots,z_n^2)$ as invariant generators.
For translation in the direction $\epsilon_1$, every $p_j$
has directional degree two. Its leading coefficient is
$e_{j-1}(z_2^2,\ldots,z_n^2)$, with $e_0=1$.
Apart from the first constant coefficient these are algebraically
independent. Scaling the root to $2\epsilon_1$ changes no
directional degree. Apply
Proposition~\ref{str:prop:weighted-invariants}, with
$d_j=2j$, $\omega_j=1$, and $\nu_\alpha=2n$.
\end{proof}

\begin{remark}[The two root-length families]
\label{str:rem:classical-root-lengths}
The Grassmannian theorem concerns the coordinate-root orbits
$\{\pm\epsilon_i\}$ in $B_n$ and
$\{\pm2\epsilon_i\}$ in $C_n$.
The other root-length orbit has threshold $4n-4$ and is treated
in Theorem~\ref{str:thm:bd-hilbert}. For $n\geq3$ these
thresholds already distinguish the two families. In rank two an
additional normalizer symmetry compares their centers, as proved in
Proposition~\ref{str:prop:root-independence}.
\end{remark}

\subsection{The remaining classical root orbits}

We now compute the Hilbert series for long roots in $B_n$,
short roots in $C_n$, and all roots in $D_n$. Together with
the preceding theorems this covers every root orbit in every
irreducible classical type.

Use standard dual coordinates $z_1,\ldots,z_n$ and put
\[
 p_j=e_j(z_1^2,\ldots,z_n^2),\qquad
 g_4=p_1^2-4p_2,\qquad \eta=z_1\cdots z_n.
\]
The subscript on $g_4$ records its polynomial degree.
The change from $p_2$ to $g_4$ is invertible over $\mathbb C$.
Thus the invariant rings and auxiliary weights are
\[
\begin{array}{c|c|c}
 W&T^W&\text{auxiliary weights}\\ \hline
 B_n=C_n&\mathbb C[p_1,g_4,p_3,\ldots,p_n]
          &(1,1,2,\ldots,2)\\
 D_n&\mathbb C[p_1,g_4,p_3,\ldots,p_{n-1},\eta]
          &(1,1,2,\ldots,2,1).
\end{array}
\]
Generators with empty index ranges are omitted.

\begin{theorem}[The remaining classical root orbits]
\label{str:thm:bd-hilbert}
\begin{enumerate}[label=\textup{(\roman*)}]
\item Suppose $\Phi=B_n$ and $\alpha$ is long, or
$\Phi=C_n$ and $\alpha$ is short, with $n\geq2$.
Put $r=a-2n+2$. The quotient is nonzero exactly when $r\geq0$,
and then
\begin{equation}\label{str:eq:bc-other-hilbert}
 H_{Z(\alpha,m)}(q)
 =\sum_{b_1+b_2+2\sum_{j=3}^n b_j\leq r}
          q^{2b_1+4b_2+\sum_{j=3}^n2jb_j}.
\end{equation}
Equivalently,
\[
 \sum_{r\geq0}H_{Z(\alpha,4n-4+2r)}(q)t^r
 =\frac{1}{(1-t)(1-tq^2)(1-tq^4)
                  \prod_{j=3}^n(1-t^2q^{2j})}.
\]
\item Suppose $\Phi=D_n$, $n\geq4$, and let $\alpha$ be any root.
Put $r=a-2n+3$. The quotient is nonzero exactly when $r\geq0$,
and then
\begin{equation}\label{str:eq:d-hilbert}
 H_{Z(\alpha,m)}(q)
 =\sum_{b_1+b_2+2\sum_{j=3}^{n-1}b_j+c\leq r}
 q^{2b_1+4b_2+\sum_{j=3}^{n-1}2jb_j+nc}.
\end{equation}
Equivalently,
\[
 \sum_{r\geq0}H_{Z(\alpha,4n-6+2r)}(q)t^r
 =\frac{1}{(1-t)(1-tq^2)(1-tq^4)(1-tq^n)
                  \prod_{j=3}^{n-1}(1-t^2q^{2j})}.
\]
\end{enumerate}
All summation indices are nonnegative integers; an empty product
is $1$, and an empty sum is $0$. The dimension of the nil-Hecke quotient is $|W|^2$ times
the value at $q=1$.
\end{theorem}

\begin{proof}
In each case we may take $\alpha=\epsilon_1-\epsilon_2$.
Translation replaces $z_1,z_2$ by $z_1+t,z_2-t$.
Put $s=z_1+z_2$ and let
$E_j=e_j(z_3^2,\ldots,z_n^2)$, with $E_0=1$ and
$E_j=0$ for $j<0$ or $j>n-2$.
Since
\[
 p_j=E_j+(z_1^2+z_2^2)E_{j-1}+z_1^2z_2^2E_{j-2},
\]
the leading translation coefficients are
\begin{equation}\label{str:eq:bd-leading-coefficients}
 \begin{array}{c|c|c}
 \text{invariant}&\text{directional degree}&
                  \text{leading coefficient}\\ \hline
 p_1&2&2\\
 g_4&2&4(s^2-E_1)\\
 p_j\ (j\geq3)&4&E_{j-2}\\
 \eta&2&-z_3\cdots z_n.
 \end{array}
\end{equation}
For $g_4$, the fourth- and third-degree terms cancel; its
quadratic coefficient is the nonzero polynomial displayed.

For $B_n,C_n$, the remaining leading coefficients are
$4(s^2-E_1),E_1,\ldots,E_{n-2}$. They are algebraically
independent: the elementary symmetric polynomials in the
$n-2$ squares are independent, and $s$ is an independent
variable. For $D_n$, the remaining coefficients are
$4(s^2-E_1),E_1,\ldots,E_{n-3},-z_3\cdots z_n$.
The last $n-2$ polynomials are independent as well.
To see this, separate a proposed relation into even and odd
powers of $z_3\cdots z_n$ and change the sign of $z_3$.
Both parts vanish separately. Replacing its square by
$E_{n-2}$ then contradicts the independence of
$E_1,\ldots,E_{n-2}$. Adjoining $s^2-E_1$ preserves
independence.

Proposition~\ref{str:prop:weighted-invariants} now applies
with the displayed weights. The thresholds are $4n-4$ and
$4n-6$, respectively, by
Corollary~\ref{geo:cor:minimal-orbit}. Substituting the
polynomial degrees and auxiliary weights gives both formulas.
Parity gives the same answer at the following odd exponent.
\end{proof}

\begin{corollary}[Dimensions in the remaining classical families]
\label{str:cor:bd-dimensions}
The center dimensions in Theorem~\ref{str:thm:bd-hilbert}
are the coefficients of $t^r$ in
\[
 \frac{1}{(1-t)^3(1-t^2)^{n-2}}
 \quad\text{and}\quad
 \frac{1}{(1-t)^4(1-t^2)^{n-3}},
\]
respectively. In particular, each dimension sequence is a
quasipolynomial in $r$ of degree $n$ and period dividing two.
For the first family with $n\geq3$ it is
\[
 \sum_{k=0}^{\lfloor r/2\rfloor}
       \binom{n+k-3}{k}\binom{r-2k+2}{2},
\]
and for $D_n$ it is
\[
 \sum_{k=0}^{\lfloor r/2\rfloor}
       \binom{n+k-4}{k}\binom{r-2k+3}{3}.
\]
For $B_2,C_2$ the first expression reduces to $\binom{r+2}{2}$.
\end{corollary}

\begin{proof}
Set $q=1$ and expand the two binomial series. Their only
poles are at $1$ and $-1$; the pole at $1$ has order $n+1$.
Partial fractions therefore give the stated quasipolynomials.
\end{proof}

\subsection{The first nonsemisimple centers in types
\texorpdfstring{$B$, $C$, and $D$}{B, C, and D}}

The preceding formulas determine dimensions. The conjugated
differential action also determines the multiplication at the
first exponent pair above the semisimple quotient.

\begin{theorem}[First nonsemisimple centers]
\label{str:thm:first-bd-centers}
Let $m=\nu_\alpha+2$ or $\nu_\alpha+3$.
\begin{enumerate}[label=\textup{(\roman*)}]
\item For a long root in $B_n$ or a short root in $C_n$,
$n\geq2$,
\[
 Z(\alpha,m)\cong\mathbb C[u]/(u^3),\qquad\deg u=2.
\]
The corresponding matrix quotient is symmetric.
\item In type $D_4$,
\[
 Z(\alpha,m)\cong\mathbb C[u,v,w]/(u,v,w)^2,
 \qquad(\deg u,\deg v,\deg w)=(2,4,4).
\]
Its socle has dimension three.
\item In type $D_n$, $n\geq5$,
\[
 Z(\alpha,m)\cong\mathbb C[u,v]/(u^3,uv,v^2),
 \qquad(\deg u,\deg v)=(2,n).
\]
Its socle is spanned by $u^2,v$, and has dimension two.
\end{enumerate}
Thus the first nonsemisimple root-power quotient in every
type $D_n$ is not self-injective.
\end{theorem}

\begin{proof}
Use the coordinates of the preceding subsection and the dual
coordinates $\xi_i$ on $P$. The invariant
$c_1=\sum_i\xi_i^2$ acts on the inverse system by the
ordinary Laplacian $\mathsf L=\sum_i\partial_{z_i}^2$.
Write $\Delta=\Delta_\Phi^\vee$ and $N=\deg\Delta$.
The discriminant is harmonic: $\mathsf L\Delta$ is an
anti-invariant of degree $N-2$, whereas every anti-invariant
is divisible by $\Delta$.
Consequently the action after removal of the discriminant is
\[
 \mathscr L(H)=\Delta^{-1}\mathsf L(\Delta H)
 =\mathsf L H+2\sum_i(\partial_{z_i}\log\Delta)
                                      \partial_{z_i}H.
\]
The output is polynomial by
Proposition~\ref{str:prop:conjugated-action}.
Euler's identity gives
\[
 \mathscr L(p_1)=2n+4N,\qquad
 \mathscr L(p_1^2)=4(n+2+2N)p_1.
\]
For signed permutations take
$\Delta=(\prod_i z_i)\prod_{i<j}(z_i^2-z_j^2)$,
up to a nonzero constant; for $D_n$ omit $\prod_i z_i$.
Pairing the terms with indices $i,j$ in its logarithmic
derivative gives
\[
 \mathscr L\left(\sum_i z_i^4\right)
 =\begin{cases}
 4(4n+1)p_1,&B_n,C_n,\\
 4(4n-1)p_1,&D_n.
 \end{cases}
\]
Since $g_4=2\sum_i z_i^4-p_1^2$, it follows that
\begin{equation}\label{str:eq:first-center-laplacian}
 \mathscr L(g_4)=
 \begin{cases}
 4n(7-2n)p_1,&B_n,C_n,\\
 -4(2n-1)(n-4)p_1,&D_n.
 \end{cases}
\end{equation}

At the exponents in the statement the invariant inverse
system has basis $1,p_1,g_4$ in the first family. Its Hilbert
series is $1+q^2+q^4$. The first coefficient in
\eqref{str:eq:first-center-laplacian} is nonzero for every
integer $n\geq2$, so $\mathscr L^2(g_4)\ne0$.
Thus the class $u$ of $c_1$ has $u^2\ne0$. All degrees
above four vanish, and $1,u,u^2$ give the required presentation.

In type $D_n$ the inverse system has basis $1,p_1,g_4,\eta$.
Moreover $\Delta\eta$ is the $B_n$ discriminant, so it is
harmonic by the same minimum-degree argument. Hence
$\mathscr L(\eta)=0$.
For $n=4$, equation~\eqref{str:eq:first-center-laplacian}
shows that $\mathscr L^2$ kills this entire basis. Thus
$u^2=0$. The Hilbert series is $1+q^2+2q^4$, so every
other product of positive-degree elements vanishes by degree.
This proves \textup{(ii)}.

For $n\geq5$, the same equation gives $u^2\ne0$, while
$\mathscr L^3$ kills the basis and hence $u^3=0$.
The class $v$ of $\xi_1\cdots\xi_n$ is nonzero. Indeed,
choose matching real normalizations of the two discriminants.
The pairing~\eqref{str:eq:rootpower-pairing} of this class
with $\Delta\eta$ is the apolar pairing of
$\Delta_\Phi(\xi)\prod_i\xi_i$ with its dual polynomial
$\Delta\eta$. It is a positive sum of squared coefficients
times factorials.
The Hilbert series is $1+q^2+q^4+q^n$; products $uv,v^2$
vanish by degree. Thus $1,u,u^2,v$ are a basis with exactly
the stated relations. The socles follow directly from these
presentations. The equivalence between Gorenstein centers
and self-injective matrix quotients is
Theorem~\ref{str:thm:rootpower-geometry}\textup{(iv)}.
\end{proof}

\subsection{A rank-three comparison}

\begin{example}[The first non-Gorenstein center in type $A_3$]
\label{str:ex:a3-first}
Use $P=\mathbb C[\xi_1,\xi_2,\xi_3,\xi_4]/
(\xi_1+\xi_2+\xi_3+\xi_4)$ and
$S=\mathbb C[c_2,c_3,c_4]$, where $c_j=e_j(\xi_1,\ldots,\xi_4)$.
For $\alpha=\epsilon_1-\epsilon_2$ and $m=8,9$,
\begin{equation}\label{str:eq:a3-first-center}
 \begin{split}
 J^{\rm pol}_{\alpha,m}
  &=(20c_4-41c_2^2,c_2c_3,c_2^3,c_3^2),\\
 Z(\alpha,m)&\cong
   \mathbb C[u,v]/(u^3,uv,v^2),\qquad\deg(u,v)=(2,3).
 \end{split}
\end{equation}
Thus $A_{\alpha,8}=A_{\alpha,9}\cong
\Mat_{24}(\mathbb C[u,v]/(u^3,uv,v^2))$ has dimension
$2304$. Its center has Hilbert series $1+q^2+q^3+q^4$
and socle $\mathbb C v\oplus\mathbb C u^2$.

Here is a direct verification of the presentation.
Take $\Delta_\Phi=\prod_{i<j}(\xi_i-\xi_j)$ and
$f=(\xi_1-\xi_2)(\xi_3-\xi_4)$.
The divided alternation in
Theorem~\ref{str:thm:alternation} gives
\[
 \operatorname{tr}_\Phi(x_\alpha^8 f)
       =984c_2^2-480c_4.
\]
To verify the constants, the left side is an invariant of
degree four, hence a linear combination of $c_2^2,c_4$.
At $(3,1,-1,-3)$ its value is $94080$, with
$(c_2^2,c_4)=(100,9)$; at $(4,2,-1,-5)$ its value is
$501336$, with $(c_2^2,c_4)=(529,40)$.
The two evaluations determine the displayed coefficients.
Theorem~\ref{str:thm:an-centers} gives the stated Hilbert
series, so all homogeneous elements of degree greater than
four vanish. It follows that the four displayed generators
belong to the content ideal. Their quotient has basis
$1,c_2,c_3,c_2^2$, with exactly that Hilbert series, so
there can be no further relations. Parity gives $m=9$.
\end{example}

For a coordinate root in $B_3$ or $C_3$, at the same exponents
$m=8,9$, Theorem~\ref{str:thm:bc-grassmannian} instead gives
\[
 Z(\alpha,m)\cong H^*(\operatorname{Gr}(3,4);\mathbb C)
 \cong\mathbb C[c]/(c^4),\qquad\deg c=2.
\]
Its Hilbert series is $1+q^2+q^4+q^6$ and its socle is
one-dimensional. Thus the two centers have the same dimension
four and different Gorenstein behavior already in rank three.
The corresponding $B_3$ or $C_3$ matrix quotient has size
$48$ and dimension $9216$.

\section{Type \texorpdfstring{$F_4$}{F4}: triality and all exponents}
\label{str:sec:f4}

The normal reflection subgroup $W(D_4)$ gives an exceptional-type
application of the structural comparison. Triality supplies
invariant generators adapted to the directional filtration.
They determine the Hilbert series for every exponent and both
root lengths. The differential action then gives the complete
self-injectivity classification.

\subsection{Triality and the relative discriminant}

Use an orthonormal basis $\epsilon_1,\ldots,\epsilon_4$ and
the standard realization
\[
 \begin{split}
 \Phi_{\rm long}&=\{\pm\epsilon_i\pm\epsilon_j:i<j\},\\
 \Phi_{\rm short}&=\{\pm\epsilon_i:1\leq i\leq4\}
 \ \cup\ 
 \{\tfrac12(\pm\epsilon_1\pm\epsilon_2
                         \pm\epsilon_3\pm\epsilon_4)\}.
 \end{split}
\]
The long roots form a subsystem of type $D_4$. In the dual
coordinates $z_1,\ldots,z_4$, put
\[
 p_j=e_j(z_1^2,z_2^2,z_3^2,z_4^2),\qquad
 \eta=z_1z_2z_3z_4,\qquad g=p_1^2-4p_2.
\]
The even signed permutations form $W(D_4)$, and
$T^{W(D_4)}=\mathbb C[p_1,p_2,p_3,\eta]$.
Adjoin the reflections $\tau$ and $\sigma$ in
$\epsilon_1$ and $\frac12(\epsilon_1+\epsilon_2+
\epsilon_3+\epsilon_4)$, respectively. The latter acts by
\[
 \sigma(z_i)=z_i-\tfrac12(z_1+z_2+z_3+z_4).
\]
Together with $W(D_4)$ these generate $W(F_4)$. The subgroup
$W(D_4)$ is normal, and its quotient is the triality group
$\mathfrak S_3$.

\begin{lemma}[Adapted triality invariants]
\label{str:lem:f4-invariants}
Define
\begin{equation}\label{str:eq:f4-generators}
 \begin{aligned}
 f_2&=p_1,&
 f_6&=p_3+\frac1{24}p_1g,\\
 f_8&=g^2+192\eta^2,&
 f_{12}&=g^3-576g\eta^2.
 \end{aligned}
\end{equation}
Then
\[
 T^{W(F_4)}=\mathbb C[f_2,f_6,f_8,f_{12}],
 \qquad
 T^{W(D_4)}=\mathbb C[f_2,f_6,g,\eta].
\]
The triality group fixes $f_2,f_6$ and acts on $g,\eta$ by
\begin{equation}\label{str:eq:f4-triality-action}
 \begin{aligned}
 \tau(g)&=g,&\tau(\eta)&=-\eta,\\
 \sigma(g)&=-\tfrac12g-12\eta,&
 \sigma(\eta)&=-\tfrac1{16}g+\tfrac12\eta.
 \end{aligned}
\end{equation}
Up to a nonzero scalar, the relative discriminant is
\begin{equation}\label{str:eq:f4-relative-disc}
 \delta=\Delta_{F_4}/\Delta_{D_4}
       =\eta(g^2-64\eta^2),\qquad\deg\delta=12,
\end{equation}
where the same expressions in the corresponding coordinates
are used in $P$ or $T$.
\end{lemma}

\begin{proof}
Expanding the two substitutions gives
\eqref{str:eq:f4-triality-action} and shows that $f_6$ is
fixed. The changes from $p_2$ to $g$ and from $p_3$ to
$f_6$ are triangular and invertible. Thus the displayed
presentation of $T^{W(D_4)}$ follows.

On the variables $g,8\sqrt3\eta$, the two transformations
in~\eqref{str:eq:f4-triality-action} are reflections whose
product has order three. They give the standard reflection
representation of $\mathfrak S_3$. Its quadratic and cubic
invariants are $g^2+192\eta^2$ and
$g^3-576g\eta^2$. For example, diagonalizing the order-three
rotation identifies this calculation with the dihedral
invariant ring $\mathbb C[zw,z^3+w^3]$.
Taking invariants proves the first presentation.

The polynomial $\eta(g^2-64\eta^2)$ changes sign under both
reflections and has degree three in the triality variables.
It is their discriminant. Viewed in $T$, its degree is twelve.
Alternatively, the product of the four coordinate-root forms
and the eight positive half-root forms is a nonzero constant
times this expression. These are precisely the short-root
hyperplanes omitted by $D_4$, proving
\eqref{str:eq:f4-relative-disc}.
\end{proof}

\begin{corollary}[The $D_4$ to $F_4$ transfer]
\label{str:cor:d4-f4-transfer}
Let $\alpha=\epsilon_1-\epsilon_2$. Write $S_D=P^{W(D_4)}$
and $S_F=P^{W(F_4)}$. For every $m\geq1$,
\[
 J^{F_4}_{\alpha,m}
   =(J^{D_4}_{\alpha,m}:\delta)\cap S_F,
 \qquad
 Z_{F_4}(\alpha,m)_k
   \cong Z_{D_4}(\alpha,m)_{k+12}^{\operatorname{sgn}},
\]
where the sign component is taken for the triality action.
Equivariant $S_D$-free resolutions transfer to $S_F$-free
resolutions by taking that component and lowering degrees by
twelve, followed by any necessary cancellations.
\end{corollary}

\begin{proof}
The $D_4$ and $F_4$ orbits of $\alpha$ are both the set of
long roots. The relative character is the sign character
by~\eqref{str:eq:f4-triality-action} and
\eqref{str:eq:f4-relative-disc}. Apply
Theorem~\ref{str:thm:subgroup-transfer} and
Corollary~\ref{str:cor:resolution-transfer}.
\end{proof}

\subsection{The Hilbert series for both root lengths}

\begin{theorem}[Type $F_4$ in every exponent]
\label{str:thm:f4-hilbert}
Let $\alpha$ be any root of $F_4$, and put
$r=\lfloor m/2\rfloor-8$. The root-power quotient is nonzero
exactly when $m\geq16$. In that range,
\begin{equation}\label{str:eq:f4-hilbert}
 H_{Z(\alpha,m)}(q)
 =\sum_{a+2b+2c+3d\leq r}q^{2a+6b+8c+12d},
\end{equation}
where all four indices are nonnegative integers. Equivalently,
\begin{equation}\label{str:eq:f4-series}
 \sum_{r\geq0}H_{Z(\alpha,16+2r)}(q)t^r
 =\frac1{(1-t)(1-tq^2)(1-t^2q^6)
                    (1-t^2q^8)(1-t^3q^{12})}.
\end{equation}
The centers for the two root lengths are isomorphic as graded
$\mathbb C$-algebras, although their ideals in fixed invariant
coordinates need not be equal. The nil-Hecke quotient dimension is
$1152^2\dim Z(\alpha,m)$.
\end{theorem}

\begin{proof}
First take the long root $\alpha=\epsilon_1-\epsilon_2$.
Translation in its direction replaces $z_1,z_2$ by
$z_1+t,z_2-t$. Put
\[
 s=z_1+z_2,\quad E_1=z_3^2+z_4^2,\quad
 E_2=z_3^2z_4^2,\quad q_0=s^2-E_1.
\]
The leading-coefficient calculation from
\eqref{str:eq:bd-leading-coefficients} gives
\[
 \begin{array}{c|c|c}
  \text{generator}&\text{directional degree}&
                     \text{leading coefficient}\\ \hline
  f_2&2&2\\
  f_6&4&(s^2+2E_1)/3\\
  f_8&4&16(q_0^2+12E_2)\\
  f_{12}&6&64(q_0^3-36q_0E_2).
 \end{array}
\]
For example, the leading coefficients of $p_3$ and $p_1g$
are $E_1$ and $8(s^2-E_1)$, giving the second row.

The polynomials $s^2,E_1,E_2$ are algebraically independent.
The change from $s^2,E_1$ to
$(s^2+2E_1)/3,q_0$ is invertible. Finally,
\[
 \det\frac{\partial(q_0^2+12E_2,\ q_0^3-36q_0E_2)}
              {\partial(q_0,E_2)}
       =-108(q_0^2-4E_2)\ne0.
\]
The three nonconstant leading coefficients are therefore
algebraically independent. Proposition~\ref{str:prop:weighted-invariants}
applies with degree--weight pairs
\[
 (2,1),\quad(6,2),\quad(8,2),\quad(12,3).
\]
There are $24$ positive $F_4$ roots. Those orthogonal to a
long root form a rank-three subsystem with $9$ positive roots,
so Theorem~\ref{rc:thm:threshold} gives
$\nu_\alpha=24-9+1=16$. Discriminant removal proves
\eqref{str:eq:f4-hilbert} and~\eqref{str:eq:f4-series}; parity
gives the same answer at the following odd exponent.

For the comparison of root lengths, consider the orthogonal map
\[
 U=\frac1{\sqrt2}
 \begin{pmatrix}
  1&1&0&0\\1&-1&0&0\\0&0&1&1\\0&0&1&-1
 \end{pmatrix}.
\]
It sends every short-root line to a long-root line, and
conversely. This follows directly from the root list: a
coordinate root is sent to a scalar multiple of
$\epsilon_i\pm\epsilon_j$, and a half root has one nonzero
coordinate in each of the two coordinate pairs after applying
$U$. Since $U^2=1$, the converse follows too.
Conjugation sends a root reflection to the reflection in the
image line. Thus $U$ normalizes $W(F_4)$.

The induced automorphism of $P$ preserves $S$ and transports
$c_S(x_\alpha^m)$ to $c_S(x_{U\alpha}^m)$: conjugate every
functional in $\Hom_S(P,S)$ by the two induced ring maps.
Scaling a root coordinate by a nonzero constant does not
change content. Weyl transitivity within each root length
therefore proves the graded algebra isomorphism across the
two lengths. Finally, $|W(F_4)|=|W(D_4)|\cdot6=192\cdot6=1152$,
and the matrix theorem gives the nil-Hecke quotient dimension.
\end{proof}

\begin{corollary}[Dimensions and growth]
\label{str:cor:f4-dimensions}
For $r=\lfloor m/2\rfloor-8\geq0$,
\[
 \dim Z(\alpha,m)
 =\sum_{d=0}^{\lfloor r/3\rfloor}
    \sum_{k=0}^{\lfloor(r-3d)/2\rfloor}
           (k+1)(r-3d-2k+1).
\]
Its generating function is
\[
 \frac1{(1-t)^2(1-t^2)^2(1-t^3)}.
\]
The dimension is a quasipolynomial of degree four and period
dividing six, with leading term $r^4/288$.
For $r=0,1,2,3,4$, the dimensions are $1,2,5,9,16$.
\end{corollary}

\begin{proof}
Set $q=1$ in~\eqref{str:eq:f4-series}. For fixed $d$ and
$k=b+c$, there are $k+1$ choices of $b,c$ and
$r-3d-2k+1$ choices of $a$, proving the sum.
Every pole of the generating function is a root of unity of
order dividing six. Its pole at $1$ has order five and leading coefficient
$1/12$ in powers of $(1-t)^{-1}$. Hence the leading term
of the coefficient quasipolynomial is $r^4/(12\cdot4!)$.
\end{proof}

\subsection{Multiplication and the self-injectivity range}

\begin{theorem}[Complete $F_4$ self-injectivity classification]
\label{str:thm:f4-symmetry}
For either root length, the first nonzero centers are
\begin{equation}\label{str:eq:f4-small-centers}
 \begin{array}{c|c}
 m&Z(\alpha,m)\\ \hline
 16,17&\mathbb C\\
 18,19&\mathbb C[u]/(u^2)\\
 20,21&\mathbb C[u]/(u^5),
 \end{array}
 \qquad\deg u=2.
\end{equation}
For every $m\geq22$ the center is not Gorenstein.
Thus, among the nonzero root-power quotients in type $F_4$,
self-injectivity, the Frobenius property, and symmetry are
equivalent and hold precisely for $16\leq m\leq21$.
\end{theorem}

\begin{proof}
It is enough to use a long root, by
Theorem~\ref{str:thm:f4-hilbert}. Put
$\Delta=\Delta_{F_4}^\vee$ and
$\mathsf L=\sum_{i=1}^4\partial_{z_i}^2$.
The conjugated action of $c_2=\sum_i\xi_i^2\in S$ is
\[
 \mathscr L(H)=\Delta^{-1}\mathsf L(\Delta H)
   =\mathsf L H+2\sum_{\beta\in\Phi^+}
                  \frac{\mathsf d_\beta H}{(\beta,z)}.
\]
Here $\mathsf L\Delta=0$, since it would be an anti-invariant
of degree less than $\deg\Delta=24$. Each quotient in the
last sum is polynomial when $H$ is invariant: its numerator
changes sign under $s_\beta$.

The relevant operator identities are
\begin{equation}\label{str:eq:f4-laplacian}
 \mathscr L(f_2)=104,\qquad
 \mathscr L(f_2^2)=216f_2,\qquad
 \mathscr L(f_6)=7f_2^2,\qquad
 \mathscr L(f_8)=3072f_6.
\end{equation}
We give details that fix their constants. The first two are
Euler's identities
$\mathscr L(p_1)=2\dim V+4\deg\Delta$ and
$\mathscr L(p_1^2)=4(\dim V+2+2\deg\Delta)p_1$.
For the other two, stability of the inverse system at $r=2$
and the Hilbert series show that the outputs are scalar
multiples of $f_2^2$ and $f_6$, respectively.
They can therefore be evaluated at $z=\epsilon_1$.
Separate the root sum into the coordinate, long, and half
roots in the displayed realization. Direct differentiation
of~\eqref{str:eq:f4-generators} gives
\[
 \begin{array}{c|r|r|r|r|r}
 H&\mathsf L H&\text{coordinate}&\text{long}&\text{half}
                                      &\mathscr L(H)\\ \hline
 f_6&1&0&2&4&7\\
 f_8&32&-32&0&128&128.
 \end{array}
\]
All entries are evaluated at $\epsilon_1$, and the three
middle entries include the factor $2$ in the root sum.
If $(\beta,\epsilon_1)=0$, evaluate the polynomial quotient by
\[
 \left.\frac{\mathsf d_\beta H}{(\beta,z)}
                    \right|_{z=\epsilon_1}
   =\frac{\mathsf d_\beta^2H(\epsilon_1)}{(\beta,\beta)}.
\]
This follows by approaching along $\epsilon_1+t\beta$.
Since $f_2^2(\epsilon_1)=1$ and $f_6(\epsilon_1)=1/24$,
the table proves the remaining constants in
\eqref{str:eq:f4-laplacian}.

At $r=0$ the center is $\mathbb C$. At $r=1$ the invariant
inverse system has basis $1,f_2$, and
$\mathscr L(f_2)\ne0$. Thus the class $u$ of $c_2$ is
nonzero and gives $\mathbb C[u]/(u^2)$.
At $r=2$ the basis is
\[
 1,f_2,f_2^2,f_6,f_8.
\]
Equation~\eqref{str:eq:f4-laplacian} gives
$\mathscr L^4(f_8)=3072\cdot7\cdot216\cdot104\ne0$.
Hence $u^4\ne0$. Degrees above eight vanish, so $u^5=0$.
The five distinct-degree elements $1,u,u^2,u^3,u^4$ exhaust
the five-dimensional center, proving
\eqref{str:eq:f4-small-centers}.

It remains to exclude every $r\geq3$. Write
$H_r(q)$ for~\eqref{str:eq:f4-hilbert}. Its highest degree is
$4r$ for $r\geq2$, since $r$ is then a nonnegative
combination of $2$ and $3$. The coefficient of $q^{4r}$ is
\[
 p(r)=\#\{(c,d)\in\mathbb Z_{\geq0}^2:2c+3d=r\}.
\]
Writing $r=6b+s$, with $0\leq s\leq5$, gives $p(r)=b+1$
for $s\ne1$ and $p(r)=b$ for $s=1$. Thus it is at least
two for $r=6$ and every $r\geq8$.
The top graded piece already has dimension
at least two in these cases, so the center is not Gorenstein.

For the remaining values $r=3,4,5,7$, the top coefficient
is one. At $r=3$, the coefficients of $q^4$ and $q^8$ are
$1$ and $2$. For $r=4,5,7$, the coefficient of $q^2$ is $1$,
whereas the coefficient of $q^{4r-2}$ is
$p(r-1)+p(r-2)=2,2,3$, respectively.
The Hilbert series is therefore not palindromic.
A graded Artinian Gorenstein algebra has a perfect pairing
in complementary degrees, so these cases are also excluded.
The equivalences with the properties of the matrix quotient
follow from Theorem~\ref{str:thm:rootpower-geometry}\textup{(iv)}.
\end{proof}

\section{Rank-two root-power quotients and their centers}
\label{str:sec:ranktwo}

Theorems~\ref{str:thm:an-centers} and~\ref{str:thm:bc-grassmannian}
already determine the dimensions and Gorenstein behavior in
types $A_2$ and $B_2$. We now obtain the detailed rank-two
presentations, resolutions, and socles, including type $G_2$.
The projective root orbits lie on $\mathbb P^1$. We first use
this observation to obtain all three Hilbert series by one argument.
We then determine explicit presentations and resolutions, the socle,
and the Gorenstein range. The matrix description transfers these
results to the corresponding root-power nil-Hecke quotients.

\subsection{Independence of the chosen root}

For a root $\alpha\in\Phi$ and $a\geq1$, put
\[
 Z_{\alpha,a}:=Z(\alpha,2a)=S/J^{\rm pol}_{\alpha,2a}.
\]
This is a finite-dimensional graded algebra by
Theorem~\ref{str:thm:rootpower-geometry}. We first justify suppressing
the chosen root from the notation.

\begin{proposition}[Independence of the chosen root]
\label{str:prop:root-independence}
Let $\Phi$ be an irreducible crystallographic root system of
rank two. For every $\alpha,\beta\in\Phi$ and $a\geq1$,
there is a graded automorphism $\sigma$ of $S$ such that
\begin{equation}\label{str:eq:root-independence}
 \sigma(J^{\rm pol}_{\alpha,2a})=J^{\rm pol}_{\beta,2a},
 \qquad Z_{\alpha,a}\cong Z_{\beta,a}.
\end{equation}
If $\alpha$ and $\beta$ are in the same $W$-orbit, the
ideals are equal and $\sigma$ can be taken to be the identity.
For different root lengths in types $B_2$ and $G_2$, $\sigma$
is induced by a linear transformation normalizing $W$.
Thus the graded isomorphism class of $Z_{\alpha,a}$ depends
only on $\Phi$ and $a$.
\end{proposition}

\begin{proof}
Let $g\in\operatorname{GL}(V)$ normalize $W$. Its extension
to $P=\Sym_{\mathbb C}(V)$ preserves $S=P^W$; denote its
restriction to $S$ by $\sigma$. For
$\varphi\in\Hom_S(P,S)$ the map
$\sigma\circ\varphi\circ g^{-1}$ is again $S$-linear, and
this assignment is bijective. The content definition therefore
gives, for every $f\in P$,
\begin{equation}\label{str:eq:content-transport}
 c_S(g(f))=\sigma(c_S(f)).
\end{equation}
If $g(x_\alpha)=c\,x_\beta$ with $c\in\mathbb C^\times$,
apply this to $f=x_\alpha^{2a}$. Multiplication by $c^{2a}$
does not change the generated ideal, so
\eqref{str:eq:root-independence} follows. When $\beta=w\alpha$
with $w\in W$, take $g=w$. Then $\sigma=\operatorname{id}_S$,
which proves equality of the ideals.

In type $A_2$, all roots lie in one $W$-orbit. For the other
two types write $W=I_2(h)$, where $h=4$ for $B_2$ and
$h=6$ for $G_2$. Choose a primitive $2h$-th root of unity
$\eta$ and put $\zeta=\eta^2$. In linear coordinates $z,w$,
the group is generated by
\[
 r_0(z)=\zeta z,\quad r_0(w)=\zeta^{-1}w,\qquad
 s_0(z)=w,\quad s_0(w)=z.
\]
The transformation $g(z)=\eta z$, $g(w)=\eta^{-1}w$ satisfies
\[
 gr_0g^{-1}=r_0,\qquad gs_0g^{-1}=r_0s_0.
\]
Thus $g$ normalizes $W$. The root line for $r_0^k s_0$ is
$\mathbb C(z-\zeta^{-k}w)$, its $-1$ eigenspace. The two
$W$-orbits of root lines are indexed by the parity of $k$.
Since
\[
 g(z-\zeta^{-k}w)=\eta(z-\zeta^{-(k+1)}w),
\]
$g$ exchanges those orbits. Composing with an element of $W$
therefore sends any chosen root of one length to a nonzero
scalar multiple of any prescribed root of the other length.
Equation~\eqref{str:eq:content-transport} proves the claim.
In the invariant coordinates
$u=zw$, $v=z^h+w^h$, this particular normalizing map induces
$u\mapsto u$, $v\mapsto-v$.
\end{proof}

Fix one root $\alpha_\Phi$ for each of the three types and write
\[
 Z_{\Phi,a}:=Z_{\alpha_\Phi,a}.
\]
This denotes a chosen graded algebra, whose isomorphism class
is independent of the choice by
Proposition~\ref{str:prop:root-independence}. By
Theorems~\ref{str:thm:alternation} and~\ref{str:thm:parity},
\begin{equation}\label{str:eq:ranktwo-matrix}
 \begin{gathered}
 Z(A_{\alpha,m})\cong Z_{\alpha,a}\cong Z_{\Phi,a},\\
 A_{\alpha,m}\cong\Mat_{|W|}(Z_{\Phi,a}),
 \qquad a=\lfloor m/2\rfloor,\quad m\geq2.
 \end{gathered}
\end{equation}
The exponent $m=1$ gives the zero algebra. The grading on the
center in~\eqref{str:eq:ranktwo-matrix} comes from its polynomial
quotient presentation. Across root lengths the isomorphism need
not respect the identity map on $S$; the proposition gives
isomorphism of graded $\mathbb C$-algebras, not equality of
ideals in fixed invariant coordinates.

\subsection{Uniform Hilbert series from the projective root orbit}

For a finite-dimensional graded algebra $Z$, put
$H_Z(q)=\sum_{k\geq0}(\dim_{\mathbb C}Z_k)q^k$.
The variable $q$ in this series records degree and is independent
of the affine Hecke parameter. The following consequence of
Theorem~\ref{str:thm:rootpower-geometry} supplies the common
calculation behind the three rank-two Hilbert series.

\begin{corollary}[Uniform rank-two Hilbert series]
\label{str:cor:root-orbit-ranktwo}
Let $\Phi$ have type $A_2$, $B_2$, or $G_2$, and put
$a=\lfloor m/2\rfloor$.  There are homogeneous generators
$u,v$ of $T^W$ such that a vector-space basis of
$\mathcal D_{\alpha,m}$ is given by the following table:
\[
\begin{array}{c|c|c}
 \Phi & (\deg u,\deg v)&
   \text{basis elements }\Delta_\Phi^\vee u^iv^j
       \quad(i,j\geq0)\\ \hline
 A_2&(2,3)&i+j\leq a-2\\
 B_2&(2,4)&i+j\leq a-2\\
 G_2&(2,6)&i+2j\leq a-3
\end{array}
\]
The generators $u,v$ here belong to the dual invariant ring
$T^W$; they need not coincide with the normalized generators of $S$
used in the recurrence presentations below. An empty index set
means that the quotient is zero.
Consequently, with $Z_{\Phi,a}$ as fixed above,
\begin{align}
 H_{Z_{A_2,a}}(q)
  &=\sum_{i+j\leq a-2}q^{2i+3j},
 &\dim Z_{A_2,a}&=\binom a2\quad(a\geq2),
       \label{str:eq:root-orbit-a2-sum}\\
 H_{Z_{B_2,a}}(q)
  &=\sum_{i+j\leq a-2}q^{2i+4j},
 &\dim Z_{B_2,a}&=\binom a2\quad(a\geq2),
       \label{str:eq:root-orbit-b2-sum}\\
 H_{Z_{G_2,a}}(q)
  &=\sum_{i+2j\leq a-3}q^{2i+6j},
 &\dim Z_{G_2,a}&=\left\lfloor\frac{(a-1)^2}{4}\right\rfloor
                       \quad(a\geq3).
       \label{str:eq:root-orbit-g2-sum}
\end{align}
All sums are over nonnegative integers.  These formulas
hold for either root length in $B_2$ and $G_2$.
The quotient is nonzero precisely when $m\geq4$ in
types $A_2,B_2$, and when $m\geq6$ in type $G_2$.
The corresponding nil-Hecke quotient dimensions are $|W|^2$ times
the displayed center dimensions.
\end{corollary}

\begin{proof}
Use the fixed $W$-invariant bilinear form to identify $V$ and
$V^*$. In rank two the orbit $\Sigma_\alpha$ is a finite
set of points on $\mathbb P^1$.  Write $Q_\alpha$ for
the product of one linear equation for each of these
points.  Since distinct point ideals in
$\mathbb C[x,y]$ are generated by relatively prime
linear forms, for every $t\geq0$,
\begin{equation}\label{str:eq:root-orbit-principal-fat-points}
 \bigcap_{\xi\in\Sigma_\alpha}\mathfrak p_\xi^t
       =(Q_\alpha^t).
\end{equation}

The required invariant generators can be seen explicitly
in real orthogonal coordinates.  In type $A_2$ take
\[
 u=x^2+y^2,\qquad Q_\alpha=y(3x^2-y^2),
 \qquad R=x(x^2-3y^2).
\]
The root lines are the zero set of $Q_\alpha$, whereas
the reflecting lines are the zero set of $R$.  Thus,
up to a nonzero scalar,
\[
 \Delta_\Phi^\vee=R,\qquad T^W=\mathbb C[u,Q_\alpha],
 \qquad v=Q_\alpha.
\]
The root points and the reflecting points are disjoint.
For type $G_2$ the same coordinates give
\[
 \Delta_\Phi^\vee=RQ_\alpha,
 \qquad T^W=\mathbb C[u,Q_\alpha^2],
 \qquad v=Q_\alpha^2.
\]
The other root orbit is obtained by interchanging
$Q_\alpha$ and $R$.  The identity
$Q_\alpha^2+R^2=u^3$ shows that either choice gives
generators of the invariant ring.  In type $B_2$ take
\[
 u=x^2+y^2,\qquad Q_\alpha=xy,
 \qquad \Delta_\Phi^\vee=xy(x^2-y^2),
 \qquad v=Q_\alpha^2.
\]
Then $T^W=\mathbb C[u,v]$.  For the other root orbit
use $Q_\alpha=x^2-y^2$ and $v=Q_\alpha^2$;
the identity $(x^2-y^2)^2=u^2-4x^2y^2$ again gives
the same invariant ring.  The asserted invariance of
these generators follows directly from the planar
reflection actions; their degrees are the fundamental
degrees of the respective dihedral groups.

Every anti-invariant form is uniquely
$\Delta_\Phi^\vee h(u,v)$.  Let $e$ be the vanishing
order of $\Delta_\Phi^\vee$ at each point of
$\Sigma_\alpha$, let $c$ be that of $v$, and put
$b=\deg v$.  The preceding coordinates give
\[
\begin{array}{c|c|c|c|c}
 \Phi&N&e&c&b\\ \hline
 A_2&3&0&1&3\\
 B_2&4&1&2&4\\
 G_2&6&1&2&6
\end{array}
\]
The quadratic form $u$ is nonzero at every root point,
since the roots are real and the chosen inner product
is positive-definite.  Consequently,
$\Delta_\Phi^\vee u^iv^j$ has exact vanishing order
$e+cj$ at each root point and total degree
$n=N+2i+bj$.

In a homogeneous combination of such monomials, distinct
values of $j$ give distinct vanishing orders.  There is
at most one corresponding $i$ for each $j$, so the term
with least order cannot cancel.  Membership in the
fat-point space in~\eqref{str:eq:rootpower-fat-points}
is therefore equivalent to the termwise inequalities
\[
 e+cj\geq N+2i+bj-m+1,
 \qquad\text{or equivalently}\qquad
 2i+(b-c)j\leq m-1-N+e.
\]
There is no need to retain the maximum with zero,
because the left side of the first inequality is
already nonnegative.  Substitution of the table yields
\[
 \begin{array}{ll}
 A_2,B_2:&2i+2j\leq m-4,\\
 G_2:&2i+4j\leq m-6.
 \end{array}
\]
These are exactly the bounds in the statement when
$a=\lfloor m/2\rfloor$.  This proves the basis
description.  The degree shift in
\eqref{str:eq:rootpower-fat-points} subtracts $N$,
so the same basis gives the displayed Hilbert series.
The vectors listed in the table form a basis of the
\emph{dual space}; no assertion that the center has a
monomial defining ideal is being made.

For $A_2$ and $B_2$, the number of pairs with
$i+j\leq a-2$ is $a(a-1)/2$.  For $G_2$ it is
\[
 \sum_{j=0}^{\lfloor(a-3)/2\rfloor}(a-2-2j)
    =\left\lfloor\frac{(a-1)^2}{4}\right\rfloor.
\]
The first admissible pair is $(0,0)$, which gives the
nonvanishing thresholds.  Finally,
Theorem~\ref{str:thm:parity} identifies the quotients
for exponents $2a$ and $2a+1$, and
\eqref{str:eq:rootpower-matrix} gives their dimensions.
\end{proof}

The formulas in Corollary~\ref{str:cor:root-orbit-ranktwo}
determine the dimensions in every exponent before any choice of a
multiplication matrix. Summing the finite geometric series gives
the rational Hilbert series stated in the type-specific theorems
below. The differential action in
\eqref{str:eq:rootpower-differential-annihilator} contains the
additional multiplication and socle information. We now make that
information explicit through recurrence presentations and resolutions.

\subsection{Explicit recurrence presentations}

Realize the dihedral group $I_2(h)$ on $P=\mathbb C[z,w]$ by
\[
 r_0(z)=\zeta z,\quad r_0(w)=\zeta^{-1}w,\quad
 s_0(z)=w,\quad s_0(w)=z,
 \qquad \zeta^h=1,
\]
where $\zeta$ is primitive, $h=3,4,6$, respectively. We may scale
the chosen root coordinate $x_{\alpha_\Phi}$ to $x=z-w$.
Scaling it by a nonzero constant does not change its content
ideal. Put $y=z+w$ and $t=x^2$. These $x,y$ are degree-one generators of $P$.
This fixes all constants in the following formulas.

\begin{theorem}[Uniform recurrence presentations]
\label{str:thm:recurrences}
The ideal defining $Z_{\Phi,a}$ is generated by the entries of
$F_\Phi^a e_1$, where $e_1$ is the first coordinate vector and
the coefficient rings and matrices are
\begin{align}
 S_{A_2}&=\mathbb C[u,v],&
 F_{A_2}&=\begin{pmatrix}-4u&v&0\\0&-u&v\\1&0&-u\end{pmatrix},
 &\deg(u,v)&=(2,3),\label{str:eq:recurrence-a}\\
 S_{B_2}&=\mathbb C[p,r],&
 F_{B_2}&=\begin{pmatrix}0&r\\1&p\end{pmatrix},
 &\deg(p,r)&=(2,4),\label{str:eq:recurrence-b}\\
 S_{G_2}&=\mathbb C[p,r],&
 F_{G_2}&=\begin{pmatrix}0&0&r\\1&0&-p^2/4\\0&1&p\end{pmatrix},
 &\deg(p,r)&=(2,6).\label{str:eq:recurrence-g}
\end{align}
In particular, two or three scalar recurrences suffice for every
cyclotomic exponent.
\end{theorem}

\begin{proof}
For $A_2$, take $u=zw$ and $v=z^3+w^3$. Then
$y^3-3uy=v$ and $t=y^2-4u$. The ring $P^{\langle s_0\rangle}$
is free over $S$ with basis $(1,y,y^2)$, and $P$ is free over
$S$ with basis $(1,y,y^2,x,xy,xy^2)$. Multiplication by $t$
on the first three basis elements is precisely
$F_{A_2}$. Hence the entries of $F_{A_2}^a e_1$ are the
nonzero coordinates of $x^{2a}$ in this basis.

For $B_2$, first put $u=zw$ and $v=z^4+w^4$. The identity
\[
 v=t^2+4ut+2u^2
\]
gives $t^2=pt+r$ after setting $p=-4u$ and $r=v-2u^2$.
The ring $P$ is free over $S$ with basis
$(t^i x^j y^k)_{0\leq i,j,k\leq1}$: use this monic quadratic
for $t$, followed by $x^2=t$ and $y^2=t+4u$.
The coordinates of $t^a$ therefore lie in the subbasis $(1,t)$
and are $F_{B_2}^a e_1$.

For $G_2$, with $u=zw$ and $v=z^6+w^6$, one has
\[
 v=t^3+6ut^2+9u^2t+2u^3.
\]
Set $p=-6u$ and $r=v-2u^3$. Then
$t^3=pt^2-(p^2/4)t+r$. A free $S$-basis of $P$ is
$(t^i x^j y^k)_{0\leq i\leq2,\,0\leq j,k\leq1}$, by the
same monic-extension argument. The subbasis $(1,t,t^2)$ gives
$F_{G_2}$. Lemma~\ref{str:lem:content} proves each presentation.
The same presentations apply to either root length after the
coordinate transport in
Proposition~\ref{str:prop:root-independence}.
\end{proof}

\subsection{Type \texorpdfstring{$A_2$}{A2}: the resolution of the center}

Corollary~\ref{str:cor:root-orbit-ranktwo} already gives the
Hilbert series. We now determine a resolution and the socle.
We use the convention that $S(-b)$ is free with a generator in
degree $b$, and write
$\operatorname{Soc}(Z)=\{z:S_+z=0\}$.

\begin{theorem}[Type $A_2$ in every exponent]
\label{str:thm:a2}
For $a\geq2$ the center $Z_{A_2,a}$ has a graded free resolution
over $S=\mathbb C[u,v]$, $\deg(u,v)=(2,3)$,
\begin{equation}\label{str:eq:a2-resolution}
\begin{split}
0\longrightarrow&S(-(3a-2))\oplus S(-(3a-1))\\
\longrightarrow&S(-(2a-2))\oplus S(-(2a-1))\oplus S(-2a)\\
\longrightarrow&S\longrightarrow Z_{A_2,a}\longrightarrow0.
\end{split}
\end{equation}
It is minimal for $a\geq3$. The Hilbert series of
Corollary~\ref{str:cor:root-orbit-ranktwo} can be written as
\begin{equation*}\label{str:eq:a2-hilbert}
 \begin{gathered}
 H_{Z_{A_2,a}}(q)=
 \frac{1-q^{2a-2}-q^{2a-1}-q^{2a}+q^{3a-2}+q^{3a-1}}
 {(1-q^2)(1-q^3)},\\
 \dim Z_{A_2,a}=\binom a2.
 \end{gathered}
\end{equation*}
The center at $a=2$ is $\mathbb C$. For every $a\geq3$ its
socle has dimension two, with one generator in each of degrees
$3a-7$ and $3a-6$.
\end{theorem}

\begin{proof}
There are three root lines. After a linear change of variables
their even-power orbit ideal is
\[
 I=(X^{2a},Y^{2a},(X+Y)^{2a})\subseteq P=\mathbb C[X,Y].
\]
We first compute its equivariant resolution. Multiplication by
$(X+Y)^{2a}$ has maximal rank in every degree of
$P/(X^{2a},Y^{2a})$. For completeness, this follows from the
elementary $\mathfrak{sl}_2$ structure on
$\mathbb C[X]/(X^d)$, $d=2a$: on the basis $X^j$ use
\[
 E(X^j)=X^{j+1},\quad
 F(X^j)=j(d-j)X^{j-1},\quad
 H(X^j)=(2j-d+1)X^j,
\]
with the boundary values zero. On the tensor square $E$ is
multiplication by $X+Y$. Complete reducibility gives Lefschetz
strings centered in the same degree, so every power of $E$ has
maximal rank between graded pieces. Subtracting the relevant
Hilbert functions now gives
\[
 H_{P/I}(q)=\frac{1-3q^{2a}+2q^{3a}}{(1-q)^2}.
\]
The three generators are minimal. The graded Hilbert--Burch
theorem, or the minimal resolution over the two-variable
polynomial ring, therefore gives two syzygies, both in degree
$3a$.

Let $U$ be the permutation representation on the three root
lines. The even powers of the roots give this representation
without signs. In characteristic zero the resolution can be
chosen $W$-equivariantly and has the form
\begin{equation}\label{str:eq:a2-equivariant}
 0\longrightarrow P(-3a)\otimes K
 \longrightarrow P(-2a)\otimes U
 \longrightarrow I\longrightarrow0,
\end{equation}
where $\dim K=2$. We determine $K$ rather than assuming it.
For $g\in W\cong\mathfrak S_3$ the graded character of $P/I$ is
\[
 \frac{1-\operatorname{tr}_U(g)q^{2a}
          +\operatorname{tr}_K(g)q^{3a}}
 {\det(1-qg\mid V)}.
\]
It is a polynomial. At a reflection the denominator is $1-q^2$
and $\operatorname{tr}_U(g)=1$; evaluating its numerator at
$q=1$ gives $\operatorname{tr}_K(g)=0$. At a three-cycle the
denominator is $1+q+q^2$ and $\operatorname{tr}_U(g)=0$;
evaluation at a primitive cube root gives
$\operatorname{tr}_K(g)=-1$. Hence $K$ is the standard
two-dimensional representation, and $U=\mathbf1\oplus K$.

The graded coinvariant representation of $\mathfrak S_3$ has
the trivial representation in degree zero, the standard
representation in degrees one and two, and the sign representation
in degree three. This also follows directly from
$P=\mathbb C[z,w]$ with invariants of degrees two and three.
Taking anti-invariants in~\eqref{str:eq:a2-equivariant} is exact.
The $S$-module $(P\otimes U)^{\det}$ has generator degrees
$1,2,3$, and $(P\otimes K)^{\det}$ has generator degrees $1,2$.
Divide the resulting anti-invariant ideal by $\Delta_\Phi$, of
degree three, and use Theorem~\ref{str:thm:alternation}. This
gives precisely~\eqref{str:eq:a2-resolution}.

For $a\geq3$ every possible entry of its first differential
has positive degree; its second differential also has
positive-degree entries. The resolution is therefore minimal.
Taking alternating sums gives the displayed rational form of
the Hilbert series, in agreement with
Corollary~\ref{str:cor:root-orbit-ranktwo}. That corollary also
gives the dimension $a(a-1)/2$. For a finite-length module over
$\mathbb C[u,v]$, the last Koszul homology is its socle shifted
by $\deg u+\deg v=5$. The two last free summands consequently
give socle degrees $3a-7$ and $3a-6$. Finally,
\eqref{str:eq:recurrence-a} at $a=2$ gives
$J^{\rm pol}_{\alpha,4}=(u,v)$.
\end{proof}

\begin{example}[The first non-Gorenstein center]
\label{str:ex:a2-first}
At $a=3$ the vector in~\eqref{str:eq:recurrence-a} is
\[
 \begin{pmatrix}-64u^3+v^2\\-6uv\\21u^2\end{pmatrix}.
\]
It generates $(u^2,uv,v^2)$. Thus
\begin{equation*}\label{str:eq:a2-radical-square-zero}
 A_{\alpha,6}\cong A_{\alpha,7}
 \cong\Mat_6\bigl(\mathbb C[u,v]/(u,v)^2\bigr).
\end{equation*}
The same algebra structure follows directly from
Corollary~\ref{str:cor:root-orbit-ranktwo}: its Hilbert series is
$1+q^2+q^3$, so every product of two positive-degree elements
vanishes. These algebras have dimension $108$ and are not
self-injective.
\end{example}

\subsection{Type \texorpdfstring{$B_2$}{B2}: Grassmannian centers}

\begin{theorem}[Type $B_2$ in every exponent]
\label{str:thm:b2}
Define polynomials in $\mathbb C[p,r]$ by
\[
 U_0=0,\qquad U_1=1,\qquad U_{b+1}=pU_b+rU_{b-1}.
\]
For $a\geq2$ one has
\begin{equation*}\label{str:eq:b2-grassmannian}
 Z_{B_2,a}\cong\mathbb C[p,r]/(U_a,U_{a+1})
 \cong H^*(\operatorname{Gr}(2,a),\mathbb C),
\end{equation*}
where $p$ has degree two and $r$ degree four. In particular,
\begin{equation*}\label{str:eq:b2-hilbert}
 H_{Z_{B_2,a}}(q)=
 \frac{(1-q^{2a-2})(1-q^{2a})}{(1-q^2)(1-q^4)},
 \qquad \dim Z_{B_2,a}=\binom a2.
\end{equation*}
Every nonzero center is a complete intersection with
one-dimensional socle, and every corresponding matrix quotient
$A_{\alpha,m}$ is symmetric. In the displayed cohomology
isomorphism, the center degrees $\deg p=2$ and $\deg r=4$
match the usual degrees of the first two Chern classes.
\end{theorem}

\begin{proof}
Apply Theorem~\ref{str:thm:bc-grassmannian} with $n=2$ to
the coordinate-root orbit, and use
Proposition~\ref{str:prop:root-independence} for the other
length. In its coordinates, set $p=c_1$, $r=-c_2$.
The complete symmetric polynomials satisfy
$h_b=c_1h_{b-1}-c_2h_{b-2}$, so $U_a=h_{a-1}$ and
$U_{a+1}=h_a$. The content ideal is therefore
$(U_a,U_{a+1})$.

This also agrees explicitly with the normalized coordinates
of~\eqref{str:eq:recurrence-b}. Put $X=z-w$ and
$Y=\sqrt{-1}(z+w)$. Then, in the signed-permutation
coordinates $X,Y$,
\[
 c_1=X^2+Y^2=-4zw=p,
 \qquad -c_2=-X^2Y^2=z^4+w^4-2(zw)^2=r.
\]
Thus the two presentations use the same invariant generators.
The Grassmannian identification, Hilbert series, and
one-dimensional socle follow from
Theorem~\ref{str:thm:bc-grassmannian}; its symmetric form
gives the final assertion.
\end{proof}

\subsection{Type \texorpdfstring{$G_2$}{G2}: an odd-part transfer
from type \texorpdfstring{$A_2$}{A2}}

The root-orbit corollary gives the Hilbert series in type $G_2$.
The following comparison with type $A_2$ determines the defining
ideal and transfers its resolution, so it also determines the socle. Let $W'$ be the subgroup generated by
reflections in the chosen $G_2$ root orbit. Then $W'\cong W(A_2)$
and $W=W'\times\langle-1\rangle$. In the dihedral coordinates,
\[
 P^{W'}=\mathbb C[u,v],\quad u=zw,\quad v=z^3+w^3,
 \qquad P^W=\mathbb C[u,b],\quad b=v^2.
\]
The notation $(-)_{\rm odd}$ below means the $-1$ eigenspace for
$v\mapsto-v$, equivalently the sum of the odd total-degree pieces.

\begin{theorem}[Type $G_2$ in every exponent]
\label{str:thm:g2}
Let $J^A_a\subseteq\mathbb C[u,v]$ be the ideal in
Theorem~\ref{str:thm:a2}, with the same chosen root coordinate.
The corresponding type-$G_2$ content ideal is
\begin{equation}\label{str:eq:g2-colon}
 J^G_a=(J^A_a:v)\cap\mathbb C[u,v^2].
\end{equation}
Multiplication by $v$ induces a graded $\mathbb C[u,v^2]$-module
isomorphism from $Z_{G_2,a}$ to $(Z_{A_2,a})_{\rm odd}$ with
degrees raised by three. In particular, for $a\geq3$,
\begin{equation}\label{str:eq:g2-odd-hilbert}
 H_{Z_{G_2,a}}(q)
 =q^{-3}\frac{H_{Z_{A_2,a}}(q)-H_{Z_{A_2,a}}(-q)}2.
\end{equation}
Put $b_1=\lfloor3a/2\rfloor-2$ and
$b_2=\lceil3a/2\rceil-1$. Over $S=\mathbb C[u,b]$, of
generator degrees $2,6$, there is a graded resolution
\begin{equation}\label{str:eq:g2-resolution}
\begin{split}
0\longrightarrow&S(-2b_1)\oplus S(-2b_2)\\
\longrightarrow&S(-(2a-4))\oplus S(-(2a-2))\oplus S(-2a)\\
\longrightarrow&S\longrightarrow Z_{G_2,a}\longrightarrow0.
\end{split}
\end{equation}
Consequently
\begin{align}
 H_{Z_{G_2,a}}(q)
 &=\frac{1-q^{2a-4}-q^{2a-2}-q^{2a}+q^{2b_1}+q^{2b_2}}
 {(1-q^2)(1-q^6)},\notag \\
 \dim Z_{G_2,a}&=\left\lfloor\frac{(a-1)^2}{4}\right\rfloor.
 \label{str:eq:g2-length}
\end{align}
The resolution is minimal for $a\geq6$, when the socle has
dimension two and degrees $2b_1-8$, $2b_2-8$. For $a=3,4,5$
the centers are, respectively,
\begin{equation}\label{str:eq:g2-small}
 \mathbb C,\qquad \mathbb C[p]/(p^2),\qquad
 \mathbb C[p]/(p^4),\qquad \deg p=2.
\end{equation}
\end{theorem}

\begin{proof}
The orbit ideal of $x^{2a}$ is the same for $W'$ and $W$,
because $-1$ fixes its even-power generators. Up to a nonzero
constant their discriminants satisfy
\[
 \Delta_W=(z^6-w^6)=(z^3-w^3)(z^3+w^3)=\Delta_{W'}v.
\]
Theorem~\ref{str:thm:subgroup-transfer}, with relative
discriminant $v$, gives~\eqref{str:eq:g2-colon} and the
module isomorphism. Here the relative character is the odd
character of $v\mapsto-v$. Since $u$ has even degree and
$v$ has odd degree, its character projection gives
\eqref{str:eq:g2-odd-hilbert}. This is the index-two case of
the general reflection-subgroup comparison.

All maps in~\eqref{str:eq:a2-resolution} are homogeneous and
are therefore equivariant for total-degree parity.
Corollary~\ref{str:cor:resolution-transfer} applies; concretely,
taking their odd parts is exact. A free summand $\mathbb C[u,v](-\delta)$
contributes an odd part generated over $S$ in degree $\delta$
if $\delta$ is odd, and in degree $\delta+3$ if $\delta$ is
even. Subtract three from all shifts, using multiplication by
$v$ on the final free term. The three generator shifts become
$2a-4,2a-2,2a$, and the two relation shifts become $2b_1,2b_2$.
This proves~\eqref{str:eq:g2-resolution}. Its alternating sum gives
the displayed rational Hilbert series, in agreement with
\eqref{str:eq:root-orbit-g2-sum}. Counting the pairs in that finite
sum gives~\eqref{str:eq:g2-length}.

If $a\geq6$, the smallest last shift is strictly larger than
the largest generator shift; hence the resolution is minimal.
The last Koszul homology shifts the socle by $2+6=8$, proving
the socle assertion. For the remaining nonzero cases the cubic
recurrence gives, in the coordinates of~\eqref{str:eq:recurrence-g},
\[
 J^G_3=(p,r),\qquad J^G_4=(p^2,r),\qquad
 J^G_5=(r+p^3/2,p^4).
\]
This proves~\eqref{str:eq:g2-small}.
\end{proof}

\subsection{Dimensions and self-injectivity}

\begin{corollary}[Complete rank-two symmetry classification]
\label{str:cor:symmetry}
Let $\Phi$ be irreducible of rank two, let $m\geq1$, and put
$a=\lfloor m/2\rfloor$. For the nonzero quotients the dimensions
and self-injectivity conditions are as follows:
\[
\begin{array}{c|c|c|c}
\Phi&\text{nonzero range}&\dim A_{\alpha,m}
 &\text{self-injective range}\\ \hline
A_2&a\geq2&36\binom a2&a=2\\
B_2&a\geq2&64\binom a2&\text{all }a\geq2\\
G_2&a\geq3&144\left\lfloor (a-1)^2/4\right\rfloor&3\leq a\leq5
\end{array}
\]
For these algebras, being self-injective, being Frobenius, and
being symmetric are equivalent. Outside the indicated
self-injective ranges the center has a two-dimensional socle.
Under the comparison hypothesis of
Corollary~\ref{rc:cor:sharp-hecke}, the same conclusions hold
for the corresponding zero-orbit cyclotomic Hecke summand;
it is a single block whenever it is nonzero.
\end{corollary}

\begin{proof}
The nonzero ranges and dimensions follow from
Corollary~\ref{str:cor:root-orbit-ranktwo} and
\eqref{str:eq:ranktwo-matrix}. The socle calculations in
Theorems~\ref{str:thm:a2}--\ref{str:thm:g2}, together with
Theorem~\ref{str:thm:rootpower-geometry}\textup{(iv)}, give
the self-injective ranges and all the stated equivalences. The final assertion
uses the stated Hecke orbit-summand isomorphism, with its
hypotheses unchanged. Its block assertion follows from the
locality of the computed center.
\end{proof}

\begin{remark}[Comparison with ordinary cyclotomic nil-Hecke algebras]
\label{str:rem:ordinary-cyclotomic}
Ordinary cyclotomic nil-Hecke algebras are matrix algebras over
Grassmannian cohomology rings and admit symmetrizing forms
\cite{Lauda,HuLiang}. The quotient here imposes a power of a
root coordinate. Even in type $A$, this is a different condition
from imposing a power of an individual strand variable. In
particular, in type $A_n$ with $n\geq2$, all root-power
quotients with $m\geq2n+2$, and the type-$G_2$ quotients
for $m\geq12$, cannot be Morita equivalent
to a nonzero ordinary cyclotomic nil-Hecke algebra. The same
holds for the first nonsemisimple type-$D_n$ quotients, at
$m=4n-4,4n-3$, by
Theorem~\ref{str:thm:first-bd-centers}. Self-injectivity is
Morita invariant, so the computed socles give the obstruction.
For coordinate roots in $B_n,C_n$, the Grassmannian
identification follows instead from the explicit content
calculation of Theorem~\ref{str:thm:bc-grassmannian}. The root-lattice setting of~\cite{KongLiND} likewise
distinguishes its Hecke algebras from the usual type-$A$
Brundan--Kleshchev setting.
\end{remark}

\begin{remark}[Scope of the cohomological interpretation]
\label{geo:rem:center-cohomology}
The general geometry of the center is the fat-point description
in~\eqref{str:eq:rootpower-fat-points}. It describes a graded dual
space, and its differential action recovers the center through
\eqref{str:eq:rootpower-differential-annihilator}. The partial flag
variety gives the exact nonvanishing threshold.

Theorem~\ref{str:thm:bc-grassmannian} identifies the
coordinate-root $B_n,C_n$ centers with Grassmannian cohomology. Its socle is the top cohomology,
and integration gives a Frobenius functional. This center
grading is separate from the degree-doubling characteristic
map in~\eqref{geo:eq:characteristic}; no map between flag
varieties is required for the algebra isomorphism.

The non-Gorenstein centers in type $A_n$ for $n\geq2$,
the first nonsemisimple type-$D_n$ centers, the type-$F_4$
centers for $m\geq22$, and the non-Gorenstein type-$G_2$
centers cannot be the full ordinary
cohomology rings of smooth connected complex projective varieties.
Poincar\'e duality would give a one-dimensional socle, contrary
to Theorems~\ref{str:thm:an-centers},
\ref{str:thm:first-bd-centers}, \ref{str:thm:f4-symmetry},
and~\ref{str:thm:g2}. Thus the flag-variety
interpretation of the threshold does not identify every
cyclotomic center with a smooth projective cohomology ring.
\end{remark}

\section{The zero-residue interpretation}
\label{ii:sec:comparison}

The preceding results were proved for the polynomial nil-Hecke
algebra. We now identify their completed multiplicative form.
This also states exactly what is needed to transfer them to a
cyclotomic Hecke summand.

Let $B_0=\mathbb C[Q]$, with characters $X_\lambda$ satisfying
$X_\lambda X_\mu=X_{\lambda+\mu}$, and put
$Y_\lambda=1-X_\lambda$. The identity character sends every
$X_\lambda$ to $1$. Its completed local ring is
$\widehat B_0=\mathbb C[[Y_{\alpha_i}:i\in I]]$.
The completed zero-residue operator algebra is generated by
these multiplication operators and
\[
 C_i(f)=\frac{f-s_i(f)}{1-X_{\alpha_i}}.
\]
These quotients preserve the coefficient ring. Write this algebra
as $\widehat{\mathcal A}_0$. Its definition is independent of
any residue-quiver arrows: at the singleton zero residue all
crossings have the displayed form, up to sign.

\begin{proposition}[Formal comparison of root-power quotients]
\label{ii:prop:formal-comparison}\label{thm:completed-matrix}
The formal exponential identifies
\[
 \widehat{\mathcal A}_0
 \cong\operatorname{End}_{\widehat S}(\widehat P)
 \cong\widehat S\otimes_S NH_W.
\]
If $W\Xi$ spans $V$, it induces an algebra isomorphism
\[
 \widehat{\mathcal A}_0/
       \langle M_{Y_\alpha^{m_\alpha}}:\alpha\in\Xi\rangle
 \cong A_{\Xi,\mathbf m}.
\]
\end{proposition}

\begin{proof}
The substitutions
\[
 X_\lambda\longmapsto e^{x_\lambda},\qquad
 x_{\alpha_i}\longmapsto\log(1-Y_{\alpha_i})
\]
are inverse continuous $W$-equivariant isomorphisms. Formal
power series are meaningful in the maximal-ideal topology;
no analytic convergence is required. Under this isomorphism,
\[
 C_i=\frac{x_{\alpha_i}}{1-e^{x_{\alpha_i}}}\partial_i,
 \qquad
 Y_\alpha=x_\alpha u_\alpha,
 \quad u_\alpha=\frac{1-e^{x_\alpha}}{x_\alpha}
                         \in\widehat P^\times.
\]
Both coefficients are units. Hence the completed crossings
generate the completed nil-Hecke algebra from
Lemma~\ref{lem:completed-nilhecke}. Also
$M_{Y_\alpha^{m_\alpha}}
=M_{x_\alpha^{m_\alpha}}M_{u_\alpha^{m_\alpha}}$,
and the second factor is an invertible endomorphism.
The corresponding two-sided ideals therefore agree.
Theorem~\ref{rc:thm:cyclotomic} shows that completion does
not change the polynomial quotient under the spanning
hypothesis, proving the result.
\end{proof}

In particular, the exponent parity, Hilbert series, center
presentations, and self-injectivity classifications above
apply to the completed zero-residue QDA quotients.
The matrix dimensions are unchanged, and the grading on
the center is transported from its polynomial presentation.

\begin{corollary}[Transfer to a Hecke orbit summand]
\label{rc:cor:sharp-hecke}
Assume $\Phi$ is irreducible. Suppose that a specified
cyclotomic affine Hecke orbit summand has been identified
with the multiplicative zero-residue quotient by $Y_\alpha^m$
and that its simple root coordinates are nilpotent.
Then it is isomorphic to $A_{\alpha,m}$. It is nonzero
exactly when $m\geq\nu_\alpha$; when nonzero it is a single
block with center $Z(\alpha,m)$ and a unique simple module
of dimension $|W|$.
\end{corollary}

\begin{proof}
If powers $Y_i^{n_i}$ vanish, then every monomial of degree
at least $1+\sum_i(n_i-1)$ in these commuting coordinates
vanishes. Each formal coefficient therefore acts through a
finite polynomial truncation, so completion leaves the quotient
unchanged. Apply Proposition~\ref{ii:prop:formal-comparison},
Theorem~\ref{rc:thm:cyclotomic}, and the exact threshold.
A matrix algebra over a local Artinian center has no nontrivial
central idempotents, which gives the block assertion.
\end{proof}

The hypothesis in this corollary is a comparison hypothesis.
The actual degenerate and non-degenerate Hecke isomorphisms,
with their field, parameter, residue, and path conditions, are
proved in~\cite{KongLiD,KongLiND}. No identification with an
arbitrary Hecke target is needed in this paper. In particular,
the zero-residue comparison above does not imply an unrestricted
nonzero-residue Hecke comparison.

\vskip5pt
\noindent\textbf{Declaration of generative AI and AI-assisted
technologies in the manuscript preparation process.}
ChatGPT-5.6 Sol was used during manuscript preparation for assistance
with literature searches and
editorial revision of the exposition. ChatGPT-6 Astra was used during manuscript preparation for checking of proofs.
Responsibility for verifying
the mathematical statements, references, and final content of the
manuscript rests with the author.

\end{document}